\documentclass[11.5pt]{article} 

\usepackage{amsthm}
\usepackage{amsfonts}
\usepackage{amssymb}
\usepackage{mathrsfs}
\usepackage{amsmath}
\usepackage{verbatim}
\usepackage{manfnt}

\usepackage[all,cmtip,2cell]{xy}
\UseTwocells

\usepackage{color}
\usepackage[usenames,dvipsnames,svgnames,table]{xcolor}

\usepackage[nottoc,numbib]{tocbibind}

\theoremstyle{definition}
\newtheorem {Definition} {Definition}[section]

\theoremstyle{plain}

\newtheorem {Theorem} [Definition]  {Theorem}
\newtheorem {Proposition}  [Definition]  {Proposition}
\newtheorem {Corollary} [Definition] {Corollary}
\newtheorem {Definition/Theorem} [Definition]  {Definition/Theorem}

\theoremstyle{remark}
\newtheorem {Remark}  [Definition]  {Remark}

\begin{document}

\title{Lie Algebroids over Differentiable Stacks}

\date{}

{\LARGE    

\begin{center}
$$$$
{\Huge
Lie Algebroids over Differentiable Stacks  
 }
 $$$$

{\huge
James Waldron    $$$$  $$$$
 }

PhD   $$$$

   University of York  \\

    Mathematics $$$$

   September 2014
\end{center}

}

\newpage

\section*{Abstract}
We develop a theory of Lie algebroids over differentiable stacks that extends the standard theory of Lie algebroids over manifolds. In particular we show that Lie algebroids satisfy descent for submersions, define the category of Lie algebroids over a differentiable stack, construct a cohomology theory for these objects, and explain the relation to the theory of $\mathcal{LA}$-groupoids. We construct a number of examples.

\addcontentsline{toc}{section}{Abstract}

\newpage

\tableofcontents

\newpage

\section*{Acknowledgements}

I would like to thank my PhD supervisor Eli Hawkins for his guidance over the past four years, for patiently explaining a large amount of Mathematics to me, and for allowing me to explore a number of topics that I have become interested in. For providing a fun and productive working environment I am grateful to my fellow PhD students in the Department of Mathematics in York, and in particular to Neil Stevens with whom I have spent many hours discussing Mathematics, and Leon Loveridge who encouraged me to apply to York in the first place. 

Finally, I must thank my parents Anne and Brind, and my brother Max, for all their support and encouragement, and my late grandmother Maire Br\"uck, who encouraged my interest in Science since before I could read.

\addcontentsline{toc}{section}{Acknowledgements}

\newpage

\section*{Author's Declaration}

The work contained in this thesis is entirely my own, and has not been submitted elsewhere for any other award.

\addcontentsline{toc}{section}{Author's Declaration}

\newpage

\section{Introduction}

Differentiable stacks are a generalisation of smooth manifolds that include, as examples, orbifolds, quotients of manifolds by Lie group actions, and classifying stacks of principal bundles. The definition of a differentiable stack is formulated in terms of the theory of stacks over a Grothendieck site, which can be seen as a `categorification' of the theory of sheaves. In contrast to manifolds, differentiable stacks form a 2-category, and studying their geometry necessarily involves 2-categorical ideas.

Lie algebroids are vector bundles equipped with a Lie bracket on their space of sections and a map to the tangent bundle of the underlying manifold, satisfying some conditions. Basic examples include tangent bundles, Lie algebras and integrable distributions. Other examples arise from actions of Lie algebras, Poisson structures, and complex structures. There are notions of representations and cohomology for Lie algebroids, and these unify many seemingly different objects, such as flat vector bundles, Poisson modules and holomorphic structures on vector bundles, and correspondingly de Rham cohomology, Poisson cohomology and Dolbeault cohomology.
$$$$

In this thesis we develop a theory of Lie algebroids over differentiable stacks which includes notions of representations and cohomology of these objects. There exists some work in this direction already. In the algebro-geometric setting, the theory of Lie algebroids over algebraic stacks has been developed by Beilinson and Bernstein in \cite{BeBe1}. In particular, they prove that Lie algebroids satisfy smooth descent and define Lie algebroids over Artin stacks. (In fact, they mostly work not with the Lie algebroids themselves, but with their enveloping algebras. The two are related by an adjunction.) 

In \cite{Meh1}, it was suggested by Mehta that `$\mathcal{LA}$-groupoids'  might be used to describe Lie algebroids over orbifolds, but he remarks that one needs to show that this definition does not depend on the groupoid one chooses to represent a given stack ($\mathcal{LA}$-groupoids are groupoid objects in the category of Lie algebroids, see section \ref{Lie algebroids over Lie groupoids}). 

In \cite{Ro1}, Roggiero Ayala gives a definition of Lie algebroids over orbifolds in terms of \'etale Lie groupoids, and a definition of Lie algebroids over more general Lie groupoids that uses the notion of an   `\'etalification'. However, it is not clear that whether these notions are Morita invariant and therefore whether or not they represent structures over the stacks defined by these groupoids.

The way we have chosen to define Lie algebroids over stacks is analogous to the way they are defined over algebraic stacks by Beilinson and Bernstein. We prove that in the differential-geometric setting Lie algebroids form a stack $\mathcal{LA}$ that satisfies descent for submersions (Theorem \ref{submersive descent}). This allows us to define Lie algebroids over a stack $\mathfrak X$ in terms of maps to $\mathcal{LA}$ (Definition \ref{Definition of Lie algebroids over stacks}). Such a map is determined by giving a Lie algebroid over $U$ for every submersion $U \to \mathfrak X$ from a manifold $U$, together with some compatibility conditions (see section \ref{Lie algebroids via test spaces}). 

Given an atlas $X \to \mathfrak X$, one gets a Lie groupoid $X \times _\mathfrak X X \rightrightarrows X$, and we show that a Lie algebroid over $\mathfrak X$ defines a Lie algebroid over $X$ together with an `action' of this groupoid (Proposition \ref{Proposition on Lie algebroids via atlases}). In the case of \'etale stacks, we show that this gives a definition of an action of an \'etale groupoid on a Lie algebroid that is similar to that given by Roggiero Ayala, but slightly stricter (Proposition \ref{G-sheaves of Lie algebras...}). In the case of general stacks the corresponding notion seems to be unrelated to the approach using \'etalifications that he suggests. 

In section \ref{Lie algebroids over Lie groupoids} we show how $\mathcal{LA}$-groupoids do arise naturally from Lie algebroids over stacks. Moreover, we characterise the $\mathcal{LA}$-groupoids that arise this way (Definition \ref{!-vacant definition}) and prove a theorem that describes their structure (Theorem \ref{inverse functor}). The description of this class of $\mathcal{LA}$-groupoids is similar to the theory of `vacant' $\mathcal{LA}$-groupoids and `matched pairs' developed by Mackenzie in \cite{Ma1}.

$$$$

We now give an overview of this thesis. In section \ref{Differentiable Stacks} we give a brief introduction to the general notions of presheaves of categories, fibered categories, descent and stacks, and recall the basic definitions concerning differentiable stacks and Lie groupoids. In section \ref{Lie Algebroids} we recall the relevant notions from the theory of Lie algebroids. We have not provided any proofs in these two sections, but have suggested a number of references.

In section \ref{The stack of Lie algebroids} we prove that the category of Lie algebroids is a stack over the category $\mathbf{Man}_\text{sub}$ of manifolds and submersions equipped with the submersion topology. We prove an analogous statement for the category $\mathrm{Rep}$ of representations of Lie algebroids. 

In section \ref{Lie algebroids over differentiable stacks} we define Lie algebroids over differentiable stacks. We use the results of section \ref{The stack of Lie algebroids} to show that such objects are determined by certain data on an atlas of a given stack. We then define representations of Lie algebroids over stacks, and give the analogous result involving atlases. We note that a Lie algebroid over a stack is not a vector bundle in the usual sense, in particular it does not have a well defined rank. Next, we use the sheaf cohomology for differentiable stacks, developed by Behrend and Xu, and by Bunke, Schick, and Spitzweck, to define Lie algebroid cohomology for Lie algebroids over differentiable stacks. Under some reasonable conditions we show that it can be computed using certain double complexes. In the case of Deligne-Mumford stacks / orbifolds the columns of these double complexes are acyclic except in degree $0$, and the total cohomology is isomorphic to the cohomology of a single cochain complex. We prove several other results in the \'etale setting. In particular, we show that Lie algebroids over \'etale stacks do have a well defined rank, and we describe the relation between such objects and the definitions given by Roggiero Ayala.

In section \ref{Lie algebroids over Lie groupoids} show how one can construct an $\mathcal{LA}$-groupoid from a Lie algebroid over a differentiable stack. We define a class of $\mathcal{LA}$-groupoids that we call `!-vacant' and prove that it is exactly these that arise from Lie algebroids over stacks. This clarifies the relationship between these objects, as suggested by Mehta. We then use the $\mathcal{LA}$-groupoids we have constructed to construct differentiable stacks that represent the `total spaces' of Lie algebroids over stacks.

In section \ref{Examples} we discuss several examples of Lie algebroids over differentiable stacks, including  tangent bundle Lie algebroids and Lie algebroids associated to Poisson structures on \'etale stacks.

\newpage

\section{Differentiable Stacks}  \label{Differentiable Stacks} 

We give a brief overview of the theory of sheaves and stacks. For the general theory of sites and stacks see \cite{MaMo1},\cite{Vi1}, and for the theory of differentiable stacks see \cite{BeXu1},\cite{He1},\cite{Me1}. For the theory of Lie groupoids see \cite{Ma2},\cite{MoMr1},\cite{MoMr2}.

\subsection{Presheaves}

\subsubsection{Presheaves of sets and the Yoneda lemma}

Let $\mathscr C$ be a category. A \emph{presheaf} over $\mathscr C$ is a contravariant functor $\mathscr C \to \mathbf{Set}$. Morphisms between presheaves are natural transformations of functors. The category of presheaves over $\mathscr C$, which we'll denote by $\mathrm{PSh}(\mathscr C)$, is then the functor category $[\mathscr C ^\mathrm{op} , \mathbf {Set} ] $.

Assuming $\mathscr C$ is locally small (i.e. that $\mathrm{Hom}_\mathscr C (x,y)$ is a set for all pairs of objects $x,y$ of $\mathscr C$) then we have, for each object $x$, the presheaf: 
$$\underline x \equiv \mathrm{Hom}_\mathscr C ( - , x )$$
Explicitly:
\begin{align*}
y & \mapsto \mathrm{Hom}_\mathscr C (y,x) \\
(f: z \to y) & \mapsto f^\ast : \mathrm{Hom}_\mathscr C (y,x) \to \mathrm{Hom}_\mathscr C (z,x) 
\end{align*}
where $f^\ast$ means precomposition with $f$. Presheaves isomorphic to a presheaf of this form are called \emph{representable}.

The Yoneda lemma states that for any presheaf $F$ over $\mathscr C$, and any object $x$ in $\mathscr C$, there is a bijection
\begin{align*}
\mathrm{Hom}_{\mathrm{PSh}(\mathscr C)} (\underline x, F) & \cong F(x) \\
(\phi: \underline x \to F) & \mapsto \phi(\mathrm{id}_x) 
\end{align*}
and in particular we have
$$\mathrm{Hom}_{\mathrm{PSh}(\mathscr C)} (\underline x, \underline y) \cong \mathrm{Hom}_\mathscr C (x,y)$$
These bijections are natural in both variables. Using the Yoneda lemma, we get a fully-faithful embedding, called the \emph{Yoneda embedding}:
\begin{align*}
\mathscr C & \to \mathrm{PSh} (\mathscr C) \\
x & \mapsto \underline x
\end{align*}

\subsubsection{Weak presheaves of categories and the 2-Yoneda lemma}

One can `categorify' the notion of presheaf by considering functors that take values in categories rather than sets, and replacing maps by functors and associativity `on the nose' by associativity up to natural isomorphism. Let $\mathscr C$ be a category. Then a \emph{weak presheaf of categories} is a contravariant weak 2-functor $\mathscr C \to \mathbf{Cat}$. Weak presheaves form a 2-category, where the morphisms are natural transformations of 2-functors, and the 2-morphisms are modifications. In particular, given two weak presheaves $\mathcal F,\mathcal G : \mathscr C \to \mathbf{Cat}$, there is a category
$$\mathrm{Hom} (\mathcal F, \mathcal G)$$
of morphisms between $\mathcal F$ and $\mathcal G$, rather than just a set.

Explicitly, a weak presheaf $\mathcal F$ over $\mathscr C$ is given by the following data:
\begin{itemize}
\item For each object $x$ of $\mathscr C$ a category $\mathcal F(x)$.
\item For each morphism $f: x \to y$ in $\mathscr C$ a functor $\mathcal F (f) : \mathcal F(y) \to \mathcal F(x)$.
\item For each composable pair of morphisms 
$$\xymatrix{
x \ar[r] ^ f & y \ar[r]^g & z 
}$$
in $\mathscr C$ a natural isomorphism of functors $\mathcal F (g,f) : \mathcal F(gf) \to \mathcal F (f) \circ \mathcal F(g)$ 
\end{itemize}
\[
\xymatrix@C+2pc{
\mathcal F(z) \rtwocell^{\mathcal F (gf) }_{\mathcal F(f) \circ \mathcal F(g)}{\;\;\;\;\;\;\;\; \mathcal F(g,f)} &  \mathcal F (x) 
}
\]
This data must satisfy the condition that for any composable triple of morphisms 
$$\xymatrix{
x \ar[r]^f & y \ar[r]^g & z \ar[r]^h & w
}$$
in $\mathscr C$ the following diagram must commute:
$$\xymatrix{
\mathcal F (hgf)  \ar[d] _ {\mathcal F(h,gf)}  \ar[r]  ^{\mathcal F (hg,f)}  &   \mathcal F(f) \mathcal F (hg)     \ar[d]  ^{\mathrm{id}_{\mathcal F(h)} \bullet \mathcal F (h,g)} \\
\mathcal F (gf) \mathcal F (h)  \ar[r] _{\mathcal F(g,f) \bullet \mathrm{id}_{\mathcal F(h)}}   &   \mathcal F(f)\mathcal F(g) \mathcal F(h)
}$$
(where $\bullet$ denotes the horizontal composition of natural transformations). A weak presheaf $\mathcal F$ is called \emph{strict} if the natural isomorphisms $\mathcal F(g,f)$ are all identities. A weak presheaf $\mathcal F$ is a \emph{weak presheaf of groupoids} if for every $x$ the category $\mathcal F(x)$ is a groupoid.

A morphism $\phi: \mathcal F \to \mathcal G$ between weak presheaves over $\mathscr C$ is given by:
\begin{itemize}
\item For each object $x$ of $\mathscr C$ a functor $\phi(x) : \mathcal F(x) \to \mathcal G(x)$  
\item For each morphism $f: x \to y$ in $\mathscr C$ a natural isomorphism $\phi(f) : \phi(x) \circ \mathcal F(f) \to \mathcal G(f) \circ \phi(y)$
\end{itemize}
$$\xymatrix{
\mathcal F (y)   \ar[d]_{\mathcal F(f)} \ar[r]^{\phi(y)} &   \mathcal G (y) \ar[d]^{\mathcal G(f)}  \\
\mathcal F (x)   \ar@2{->}[ur]_{\phi(f)}  \ar[r]_{\phi(x)} &   \mathcal G (x)  
}$$
This data must satisfy the condition that for any composable pair of morphisms
$$\xymatrix{
x \ar[r] ^ f & y \ar[r]^g & z 
}$$
in $\mathscr C$, the following diagram commutes:
$$\xymatrix{
&  \phi(x) \mathcal F(gf)  \ar[dl]_{\mathrm{id}_{\phi(f)} \bullet \mathcal F(g,f) }    \ar[rr]^{\phi(gf)}  &   &      \mathcal G(gf) \phi(z)    \ar[dr]^{\mathcal G(g,f) \bullet \mathrm{id}_{\phi(z)}} \\
\phi(x) \mathcal F(f) \mathcal F(g) \ar[rr]_{\phi(f) \bullet \mathrm{id}_{\mathcal F(x)}}  & & \mathcal G(f) \phi(y) \mathcal F(g)   \ar[rr]_{\mathrm{id}_{\mathcal G(f)} \bullet \phi(g)}  & &  \mathcal G(f) \mathcal G(g) \phi(z) 
}$$

A 2-morphism $\alpha:\phi \to \psi$ between morphisms $\phi,\psi: \mathcal F \to \mathcal G$ is given by a natural transformation $\alpha(x) : \phi(x) \to \psi(x)$ for each object $x$ of $\mathscr C$ such that for every morphism $f:x \to y$ in $\mathscr C$ the following diagram commutes:
$$\xymatrix{
\phi(x) \mathcal F(f)  \ar[d]_{\alpha(x) \bullet \mathrm{id}_{\mathcal F(f)} } \ar[r]^{\phi(f)}   &   \mathcal G(f)  \phi(y)     \ar[d]^{\mathrm{id}_{\mathcal F(g)} \bullet \alpha(y)  }   \\
\psi(x)  \mathcal F(f)   \ar[r]_{\psi(f)}  &    \mathcal  G(f)  \psi(y) 
}$$

We can consider any set as a `discrete' category, i.e. as a category with only identity morphisms. This gives a fully-faithful embedding 
$$\mathbf{Set} \to \mathbf{Cat}$$
and for any category $\mathscr C$ a fully-faithful embedding of the category of presheaves of sets into the category of weak presheaves of categories. From now on we'll denote the category of weak presheaves over $\mathscr C$ by $\mathrm{PSh}(\mathscr C)$ and consider the category of presheaves of sets as a subcategory of $\mathrm{PSh}(\mathscr C)$.

The 2-Yoneda lemma states that for any object $x$ of $\mathscr C$ and any weak presheaf $\mathcal F$ there is an equivalence of categories:
\begin{align*}
\mathrm{Hom}_{\mathrm{PSh}(\mathscr C)} (\underline x, \mathcal F)  &  \simeq  \mathcal F(x)  \\
(\phi : \underline x \to \mathcal F)  & \mapsto  \phi (\mathrm{id}_x)
\end{align*}

\subsubsection{Fibered categories and the Grothendieck construction}  
\label{Fibred categories and the Grothendieck construction}

Weak presheaves can be described in a different way using fibred categories. For some statements and constructions this description can be easier to work with.

Let $\mathscr C$ be a category, then a \emph{category over} $\mathscr C$ is a functor $\pi_\mathscr D: \mathscr D \to \mathscr C$. We can depict a category over $\mathscr C$ as
$$\xymatrix{
\mathscr D \ar[d] _{\pi_\mathscr D} \\
\mathscr C
}$$
If $x$ is an object of $\mathscr C$, then the fibre $\mathscr D_x$ of $\pi_\mathscr D : \mathscr D \to \mathscr C$ over $x$ is the subcategory of $\mathscr D$ consisting of objects $d$ such that $\pi_\mathscr D (d) = x$ and morphisms $f$ such that $\pi_\mathscr D (f) = \mathrm{id}_x$.

Categories over $\mathscr C$ form a (strict) 2-category denoted $\mathbf{Cat} / \mathscr C$ (the slice 2-category of $\mathbf{Cat}$ over $\mathscr C$): a morphism 
$$\phi:  (\pi_\mathscr D: \mathscr D \to \mathscr C) \to (\pi_\mathscr E: \mathscr E \to \mathscr C)$$
is a functor 
$$\phi: \mathscr D \to \mathscr E$$
such that $\pi_\mathscr D = \pi_\mathscr E \circ \phi$, i.e. the diagram
$$\xymatrix{
\mathscr D \ar[dr]_{\pi_\mathscr D} \ar[rr]^\phi &  &  \mathscr E \ar[dl]^{\pi_\mathscr E} \\
& \mathscr C
}$$
commutes, and a 2-morphism between a pair of morphisms $\phi,\psi:  (\pi_\mathscr D: \mathscr D \to \mathscr C) \to (\pi_\mathscr E: \mathscr E \to \mathscr C)$ is a natural transformation
$$\alpha : \phi \to \psi$$
which is `vertical' in the sense that 
$$\alpha \bullet \mathrm{id}_{\pi_\mathscr E} = \mathrm{id}_{\pi_\mathscr D}$$
In terms of components, this last condition is the condition that for each object $x$ of $\mathscr C$ and each object $d$ of the fibre $\mathscr D_x$, the component $\alpha_d \in \mathrm{Hom}_\mathscr E (\phi(d),\psi(d))$ of $\alpha$ satisfies $\pi_\mathscr E (\alpha_d) = \mathrm{id}_x$, i.e. that $\alpha_d$ is a morphism in the fibre $\mathscr E_{x}$.

If $\pi_\mathscr D : \mathscr D \to \mathscr C$ is a category over $\mathscr C$ then a morphism $F':d' \to d$ in $\mathscr D$ is called \emph{Cartesian with respect to $\pi_\mathscr D$} (or just \emph{Cartesian} if $\pi_\mathscr D$ is understood), if it satisfies the following property: for all morphisms $F'': d'' \to d$ in $\mathscr D$, and all morphisms $f : \pi_\mathscr D (d'') \to \pi_\mathscr D (d')$ in $\mathscr C$ such that $\pi_\mathscr D (F'') = \pi_\mathscr D(F') \circ f$, there exists a unique morphism $F: d'' \to d$ in $\mathscr D$ such that $\pi_\mathscr D (F) = f$ and $F'' = F' \circ F$. The situation can be depicted as follows:
$$\xymatrix{
\\
d''  \ar@/^{2pc}/[rr] ^{F''} \ar@{.>}[r]^{F}  & d'     \ar[r]^{F'} &     d   \\
\pi_\mathscr D (d'')  \ar@/_{2pc}/[rr]  _{\pi_\mathscr D (F'')} \ar[r]_{f} &    \pi_\mathscr D (d')    \ar[r]_{\pi_\mathscr D(F')} &    \pi_\mathscr D (d)  
}$$

A category $\pi_\mathscr D : \mathscr D \to \mathscr C$ over $\mathscr C$ is called a \emph{category fibred over} $\mathscr C$ and the functor $\pi_\mathscr D$ is called a \emph{Grothendieck fibration}, if for all morphisms $f:x \to y$ in $\mathscr C$ and all objects $d \in \mathscr D_y$ there is a Cartesian morphism $F:d' \to d$ such that $\pi_\mathscr D (F) = f$. Such a morphism is called a pullback of $d$ along $f$. A morphism between fibred categories over $\mathscr C$ is a morphism of categories over $\mathscr C$ that maps Cartesian morphisms to Cartesian morphisms, and a 2-morphism between a pair of such morphisms is a 2-morphism between morphisms of categories over $\mathscr C$. We'll denote the 2-category of fibred categories over $\mathscr C$ by  $\mathrm{Fib}(\mathscr C)$. 

The \emph{Grothendieck construction}, denoted $\int_\mathscr C$, gives an equivalence of 2-categories: 
$$\int_\mathscr C : \mathrm{PSh}(\mathscr C) \to \mathrm{Fib}(\mathscr C)$$
Given a weak presheaf $\mathcal F : \mathscr C \to \mathbf{Cat}$, the category $\int_\mathscr C \mathcal F$ has as objects pairs $(x,d)$, where $x$ is an object in $\mathscr C$ and $d$ is an object in $\mathcal F (x)$. A morphism from $(x,d)$ to $(x',d')$ is a pair $(f,F)$, where $f:x \to x'$ is a morphism in $\mathscr C$ and $F: d \to \mathcal F(f)(d')$ is a morphism in $\mathcal F(x)$.  Projecting onto the first factor gives a functor 
$$\xymatrix{
\int_\mathscr C \mathcal F  \ar[d]_{\mathrm{pr}_\mathscr C} \\
\mathscr C 
}$$ 
which makes $\int_\mathscr C \mathcal F$ into a category fibred over $\mathscr C$. 

Conversely, if $\pi_\mathscr D : \mathscr D \to \mathscr C$ is a category fibred over $\mathscr C$ then one can construct a weak presheaf  that maps an object $x$ to the fibre $\mathscr D_x$ and by choosing a pullback $F:d' \to d$ for all $f: x \to y$ and $d$ in $\mathscr D_y$, one can construct a functor $f^\ast : \mathscr D_y \to \mathscr D_x$ for each morphism $f:x \to y$ in $\mathscr C$. The fact that this data then determines a weak presheaf follows from the universal property of Cartesian morphisms.

If $\mathcal F$ is a weak presheaf then we'll sometimes denote $\int_\mathscr C \mathcal F$ also by $\mathcal F$. Whether we are considering $\mathcal F$ as a weak presheaf or fibred category will usually be clear from the context.

\subsubsection{Representable morphisms}

A morphism $\phi: \mathcal F \to \mathcal G$ of weak presheaves over $\mathscr C$ is called \emph{representable} if for all morphisms $\psi: \underline x \to \mathcal G$ the fibre product
$$\xymatrix{
\mathcal F \times _ \mathcal G \underline x  \ar[d] \ar[r]  &  \underline x \ar[d]^\psi  \\
\mathcal F  \ar[r]_\phi &  \mathcal G 
}$$
is representable, i.e. that $\mathcal F \times _ \mathcal G \underline x$ is equivalent to $\underline y$ for some object $y$ of $\mathscr C$.

\subsection{Sheaves and Stacks}

If one considers a weak presheaf as a categorified presheaf, then a stack can be thought of as a categorified sheaf. In order to formulate the gluing condition that distinguishes sheaves from presheaves we need the notion of a Grothendieck topology.

\subsubsection{Grothendieck topologies and sites}

A \emph{Grothendieck topology} on a category $\mathscr C$ is given by specifying for each object $x$ of $\mathscr C$ a collection of families of morphisms $\{f_i : y_i \to x \}_{i \in I}$, called \emph{covering families} of $x$, such that:
\begin{itemize}
\item If $f: y \to x$ is an isomorphism then $\{f: y \to x \}$ is a covering family 
\item If $\{f_i : y_i \to x \}_{i \in I}$ is a covering family and $g: z \to x$ is a morphism then all of the fibre products $y_i \times _x z$ exist in $\mathscr C$ and the family of projections $\{y_i \times_x z \to z\}_{i \in I}$ is a covering family of $z$
\item If $\{f_i :y_i \to x\}_{i \in I}$ is a covering family of $x$ and for each $i \in I$ we are given a covering family $\{g_j : z_j \to y_i\}_{j \in J_i}$ of $y_i$, then $\{ (f_i \circ g_j) : z_j \to x \} _ {i \in I, j \in J_i}$ is a covering family of $x$.
\end{itemize}
A category together with a choice of Grothendieck topology is called a \emph{site}.  

The \emph{covering sieve} $U$ determined by a covering family $\{f: y \to x \}$ is the sub-presheaf $U \hookrightarrow \underline x$ of $\underline x$ consisting of morphisms $z \to x$ that factor through one of the $f_i$'s. Two Grothendieck topologies are called \emph{equivalent} if they determine the same covering sieves.

\subsubsection{Sheaves of sets}

Let $\mathscr C$ be a site. Then a presheaf of sets $\mathcal F : \mathscr C \to \mathbf{Set}$ is a sheaf if it satisfies either of the following equivalent conditions:
\begin{itemize}
\item For any covering family $\{f_i: y_i \to x \}_{i \in I}$  the map
\begin{align*}
\mathcal F (x)  &  \to \mathrm{Des}( \{f_i: y_i \to x \}_{i \in I} , \mathcal F)  \\
d & \mapsto \{ \mathcal F (f_i) (d) \}_{i \in I} 
\end{align*}
is a bijection. Here, $ \mathrm{Des}( \{f_i: y_i \to x \}_{i \in I} , \mathcal F) $ is the set consisting of collections $\{ d_i \in \mathcal F(y_i) \}_{i\in I}$ such that 
$$\mathcal F (\mathrm{pr}_{y_i}) (d_i)  = \mathcal F (\mathrm{pr}_{y_j}) (d_j) $$ 
for all $i,j \in I$, where $\mathrm{pr}_{y_i}$ is the projection map $y_i \times _x y_j \to y_i$.

\item For any covering family $\{f_i: y_i \to x \}_{i \in I}$ with associated covering sieve $U \hookrightarrow \underline x$, the restriction map 
$$\mathrm{Hom}_{\mathrm{PSh}(\mathscr C)} ( \underline x , \mathcal F )       \to    
\mathrm{Hom}_{\mathrm{PSh}(\mathscr C)} ( U , \mathcal F ) $$
is a bijection.

\end{itemize}

Morphisms of sheaves are morphisms of the underlying presheaves. The category of sheaves over a site $\mathscr C$ is denoted by $\mathrm{Sh}(\mathscr C)$. A Grothendieck topology is called \emph{subcanonical} if all representable presheaves are sheaves.

It follows from the second of the characterisations of sheaves that two equivalent topologies on a fixed category determine the same sheaves.

\subsubsection{Stacks of categories}

Let $\mathscr C$ be a site. Then a weak presheaf of categories $\mathcal F : \mathscr C \to \mathbf{Cat}$ is a \emph{stack} if it satisfies either of the following equivalent conditions:
\begin{itemize}
\item For any covering family $\{f_i: y_i \to x \}_{i \in I}$  the functor
\begin{align*}
\mathcal F (x)  &  \to \mathrm{Des}( \{f_i: y_i \to x \}_{i \in I} , \mathcal F)  \\
d & \mapsto  \left( \{ \mathcal F (f_i) (d) \}_{i \in I}   ,   \{ \mathcal F(f _i  ,\mathrm{pr}_i)  \circ \mathcal F ( f_j , \mathrm{pr}_j ) ^{-1}   \}  _{i,j \in I}   \right)   \\
(g: d \to d')  & \mapsto \{  \mathcal F (f_i) (g)  \}_{i \in I}
\end{align*}
is an equivalence of categories. Here, $ \mathrm{Des}( \{f_i: y_i \to x \}_{i \in I} , \mathcal F) $ is the category whose objects are collections 
$$\left( \{ d_i  \}_{i\in I} ,  \{\gamma_{ij}  \}_{i,j \in I} \right) $$
where $d_i$ is an object in $\mathcal F(y_i)$ and $\gamma_{ij}$ is an isomorphism $\gamma_{ij} : \mathcal F (\mathrm{pr}_{y_j}) (d_j) \to (\mathrm{pr}_{y_i}) (d_i)$ in the category $\mathcal F(y_i \times_x y_j)$, such that the \emph{cocycle condition}
$$(\gamma_{ij} \vert_{y_{ijk}} ) ( \gamma_{jk} \vert_{y_{ijk}} ) = \gamma_{ik} \vert_{y_{ijk}} $$
holds, where $y_{ijk} = y_i \times _x y_j \times _x  y_k$ and we have used the shorthand that $\gamma_{ij} \vert_{y_{ijk}} = \mathcal F(\mathrm{pr}_{ij}) (\gamma_{ij})$ and we have suppressed several isomorphisms of the form $\mathcal F(\mathrm{pr}_i, \mathrm{pr}_{ij})$.  A morphism 
$$ \{ g_i \}_{i \in I}   : \left( \{ d_i  \}_{i\in I} ,  \{\gamma_{ij}  \}_{i,j \in I} \right)  \to 
\left( \{ d'_i  \}_{i\in I} ,  \{\gamma'_{ij}  \}_{i,j \in I} \right)  $$
in the category $ \mathrm{Des}( \{f_i: y_i \to x \}_{i \in I} , \mathcal F) $ is a collection  $ \{ g_i \}_{i \in I}$ where $g_i : d_i \to d'_i$ is a morphism in $\mathcal F(y_i)$, such that
$$ \gamma'_{ij} \circ \mathcal F (\mathrm{pr}_j) ( g_j )  =   \mathcal F (\mathrm{pr}_i)(g_i) \circ \gamma_{ij} $$

An object $\left( \{ d_i  \}_{i\in I} ,  \{\gamma_{ij}  \}_{i,j \in I} \right) $ is called a \emph{descent datum} for $\mathcal F$.

\item For any covering family $\{f_i: y_i \to x \}_{i \in I}$ with associated covering sieve $U \hookrightarrow \underline x$, the restriction map 
$$\mathrm{Hom}_{\mathrm{PSh}(\mathscr C)} ( \underline x , \mathcal F )       \to    
\mathrm{Hom}_{\mathrm{PSh}(\mathscr C)} ( U , \mathcal F ) $$
is an equivalence of categories.

\end{itemize}

The category $\mathrm{Des}( \{f_i: y_i \to x \}_{i \in I} , \mathcal F)$ is called the \emph{category of descent data} for $\mathcal F$ and the covering family $\{f_i: y_i \to x \}_{i \in I}$. If the functor described above is an equivalence then $\mathcal F$ is said to \emph{satisfy descent} for the covering $\{f_i: y_i \to x \}_{i \in I}$.

If $x$ is an object of $\mathscr C$, then the slice category $\mathscr C / x$ inherits a Grothendieck topology from that of $\mathscr C$: we have the functor 
$$\xymatrix{
\mathscr C / x  \ar[d] _{\mathrm{pr}_\mathscr C}  \\
\mathscr C
}$$
that maps an object $y \to x$ of $\mathscr C / x$ to $y$, then we define a family of morphisms in $\mathscr C / x$ to be a covering family if and only if its image under $\mathrm{pr}_\mathscr C$ is a covering family for the Grothendieck topology on $\mathscr C$. We then have that a weak presheaf $\mathcal F : \mathscr C \to \mathbf {Cat}$ is a stack if and only if it satisfies the two conditions:
\begin{enumerate}
\item For all objects $x$ of $\mathscr C$ and all elements $d,d'$ of $\mathcal F(x)$, the presheaf of sets
\begin{align*}
\underline {\mathrm{Hom}} _ \mathcal F (d,d') : \mathscr C / x & \to \mathbf{Set}  \\
(f:y \to x) & \mapsto \mathrm{Hom}_{\mathcal F (y)} (\mathcal F(f) (d) , \mathcal F(f) (d') ) 
\end{align*}
is a sheaf.
\item For all covering families $\{f_i : y_i \to x \}_{i \in I}$ and all descent datum $\left( \{ d_i  \}_{i\in I} ,  \{\gamma_{ij}  \}_{i,j \in I} \right) $ there exists an object $d$ of $\mathcal F(x)$ and isomorphisms 
$$ g_i : d_i \to \mathcal F(\mathrm{pr}_i) (d)$$
such that 
$$ \gamma_{ij} \circ \mathcal F(\mathrm{pr}_j) (g_j) = \mathcal F(\mathrm{pr}_i) (g_i) $$
\end{enumerate}

If a stack satisfies the first of these conditions then it is said to be \emph{separated}. The second condition is that `descent datum are effective'.  There is a `stackification' functor $\mathrm{PSh} (\mathscr C) \to \mathrm{PSh} (\mathscr C)$ that maps weak presheaves to stacks \cite{Sta1}.

Morphisms of stacks are morphisms of weak presheaves. The 2-category of stacks over a site $\mathscr C$ is denoted by $\mathrm{St}(\mathscr C)$.  As in the case of sheaves (of sets) we have that two equivalent topologies on the same category determine the same stacks.

The conditions for a weak presheaf to be a stack can be reformulated in terms of fibered categories.

\subsection{Stacks over the category of manifolds}

\subsubsection{The sites $\mathbf{Man}$, $\mathbf{Man}_\text{sub}$ and $\mathbf{Man}_\text{\'et}$}

Most of the stacks we will consider will be stacks over the category $\mathbf{Man}$ of manifolds and smooth maps. We will also need to consider the subcategories $\mathbf{Man}_\text{sub}$ and $\mathbf{Man}_\text{\'et}$, consisting of manifolds and submersions, and manifolds and \'etale maps (local diffeomorphisms) respectively. 

There are several natural topologies on $\mathbf{Man}$. A family $\{f_i : U_i \to X \}$ of jointly surjective smooth maps is a covering family for:
\begin{itemize}
\item the \emph{open cover topology} if each $f_i$ is an open embedding,
\item the \emph{\'etale topology} if each $f_i$ is an \'etale map, 
\item the \emph{submersion topology} if each $f_i$ is a submersion.
\end{itemize}

These topologies are all equivalent as topologies on $\mathbf{Man}$, because submersions admit local sections. However, when restricted to $\mathbf{Man}_\text{sub}$ the submersion topology is strictly stronger than the \'etale and open cover topologies. When restricted to $\mathbf{Man}_\text{\'et}$ the open cover and \'etale topologies are equivalent.

The open cover topology has the property that a weak presheaf $\mathcal F$ is a stack with respect to it if and only if $\mathcal F$ satisfies descent for all covering families of the form $\{U_i \hookrightarrow X\}_{i \in I}$, where $\{U_i\}_{i \in I}$ is an open covering of a manifold $X$.

These topologies on $\mathbf{Man}$ are subcanonical, so that for any manifold $X$ the presheaf
$$ \underline X = \mathrm{Hom}_{\mathbf{Man}} ( - , X ) $$
is a sheaf.

\subsubsection{Stacks over $\mathbf{Man}$}

We write $\mathrm{St} (\mathbf{Man})$ for the 2-category of stacks over $\mathbf{Man}$, defined with respect to any of the equivalent topologies mentioned above.  When considering stacks over $\mathbf{Man}_\text{sub}$ we will specify which topology we are referring to. We'll usually denote stacks over $\mathbf{Man}$ by the letters $\mathfrak X, \mathfrak Y, \mathfrak Z$, and use $U,V,X,Y,Z$ for manifolds. 

Via the Yoneda embedding 
$$ X \mapsto \underline X = \mathrm{Hom}_{\mathbf{Man}} ( - , X ) $$
we can consider $\mathbf{Man}$ as a full subcategory of $\mathrm{St}(\mathbf{Man})$, so we will sometimes not distinguish between a manifold $X$ and the sheaf $\underline X$.

A morphism $\mathfrak X \to \mathfrak Y$ of stacks over $\mathbf{Man}$ is called a \emph{representable submersion} if for all morphisms $X \to \mathfrak Y$, where $X$ is a manifold, the fibre product $\mathfrak X \times_ \mathfrak Y X$ is representable (it is equivalent to $\underline Y$ for some manifold $Y$), and the projection morphism $\mathfrak X \times _ \mathfrak Y X \to X$ is a submersion. Similarly we have the notion of \emph{representable \'etale map}. 

We define the subcategories $\mathrm{St}(\mathbf{Man})_\text{sub}$ and $\mathrm{St}(\mathbf{Man})_\text{\'et}$ to be the sub 2-categories of $\mathrm{St}(\mathbf{Man})$ consisting of stacks over $\mathbf{Man}$ and representable submersions / representable \'etale morphisms respectively, and 2-morphisms between them.

\subsection{Lie groupoids}

\subsubsection{Lie groupoids}

Given manifolds $X,Y,Z$ and smooth maps $f:X \to Z$, $g:Y \to Z$, we define their fibre product $X_f \! \times _g Y$ by
$$ X_f  \! \times _g Y = \{ (x,y) \in X \times Y | f(x)=g(y) \} $$
If either $f$ or $g$ is a submersion, then $X_f \! \times _g Y$ is a manifold. When there is no risk of confusion, we will sometimes drop the labels $f,g$ and denote the fibre product by $X \times _Z Y$.

A groupoid is a (small) category where every morphism is an isomorphism. A \emph{Lie groupoid} is a groupoid object in $\mathbf{Man}$, where the source and target maps are submersions. More explicitly, a Lie groupoid consists of two manifolds $G$ and $X$, and smooth maps $s,t,m,u,i$: 
\begin{align*}
s,t: G & \to X \\
m:  G  _s \! \times _t G & \to G \\
u: X & \to G \\
i: G & \to G 
\end{align*}
called respectively the source, target, multiplication, unit and inverse maps, which satisfy the axioms of a groupoid:
\begin{align*}
s(m(g,h)) & = s(h) \\
t(m(g,h)) & = t(g) \\
m((g,h),k) & = m(g,m(h,k)) \\
s(u(x)) &  = t(u(x))  = x \\
m(u(t(g)),g) & = m(g, u(s(g)) )  = g \\
s(i(g)) & = t(g) \\
t(i(g)) & = s(g) \\
m(i(g),g) & = u(s(g)) \\
m(g, i(g) ) & = u(t(g))
\end{align*}
 and where $s$ and $t$ are submersions. We'll often denote a Lie groupoid by $G \rightrightarrows X$, with the structure maps understood, denote $G_s \! \times_t G$ by $G_2$, and denote $m(g,h)$ by $gh$.

A morphism of Lie groupoids from $G\rightrightarrows X$ to $G'\rightrightarrows X'$ consists of smooth maps $\phi:G\to G'$ and $f:X \to X'$, which satisfy the axioms of a functor.
Given a pair of morphisms 
$$(\phi,f) , (\phi',f') : G\rightrightarrows X \to G' \rightrightarrows X'$$
a 2-morphism from $(\phi,f)$ to $(\phi',f')$ is a smooth map $\alpha:X \to G'$ which satisfies the axioms of a natural transformation. With these notions of morphisms and 2-morphisms Lie groupoids form a strict 2-category, which we'll denote by $\mathrm{LieGpd}$. 

\subsubsection{Actions of Lie groupoids on manifolds}

If $Y$ is a manifold then a \emph{left action} of $G\rightrightarrows X$ on $Y$ along a smooth map $\epsilon: Y \to X$ is a map 
\begin{align*}
  G _s \! \times _\epsilon Y & \to Y   \\
(g,y) & \mapsto gy 
\end{align*}
satisfying the axioms of a left action, i.e.
\begin{align*}
\epsilon (gy)  &  = t(g)  \\
   u_{\epsilon(y)} y  &  =  y  \\
g(g'y) & = (gg')y
\end{align*}
Right actions are defined similarly. 

\subsubsection{Representations of Lie groupoids}

If $\pi: E \to X$ is a vector bundle, then a (linear) action or \emph{representation} of $G\rightrightarrows X$ on $E$ is an action of $G\rightrightarrows X$ on $E$ along $\pi$, such that each morphism $g\in G$ acts as a linear map
$$E_{s(g)} \to E_{t(g)}$$
Such an action is equivalent to the data of an isomorphism 
$$s^\ast E \to t^\ast E$$
of vector bundles over $G$, that satisfies the cocycle condition over $G_2 = G _s \times _t G$.

\subsection{Differentiable stacks}

\subsubsection{Stacks associated to Lie groupoids}

Given any manifold $X$ there is a Lie groupoid $X\rightrightarrows X$, where the structure maps are all identities, so that as a category $X\rightrightarrows X$ has only identity morphisms. This gives a fully-faithful embedding:
\begin{align*}
\mathbf{Man} & \to \mathrm{LieGpd}  \\
X & \mapsto X \rightrightarrows X 
\end{align*}
Given any Lie groupoid $G\rightrightarrows X$ we then have a strict presheaf of groupoids:
$$Y \mapsto \mathrm{Hom}_{\mathrm{LieGpd}} ( Y \rightrightarrows Y , G \rightrightarrows X ) $$
which is in fact separated. The stackification of this strict presheaf is a stack over $\mathbf{Man}$ called the classifying stack of $G \rightrightarrows X$, and denoted $[G \rightrightarrows X]$.

The stack $[G\rightrightarrows X]$ can be described more concretely as follows. If $Y$ is a manifold then a \emph{principal $G\rightrightarrows X$ bundle over $Y$ } is a manifold $P$, together with a submersion $\pi:P \to Y$ and a right action of $G\rightrightarrows X$ on $P$ along a smooth map $\epsilon: P \to X$, such that the map
\begin{align*}
P _\epsilon \times _ t G  & \to P _ \pi \times _\pi P \\
(p,g)  & \mapsto (p,pg) 
\end{align*}
is a diffeomorphism. A morphism between principal bundles $P$ and $P'$ over $Y$ is a smooth map $P \to P'$ that commutes with the maps to $Y$ and $X$ and with the actions of $G \rightrightarrows X$. Such a morphism is necessarily a diffeomorphism, so that the category $\mathrm{Bun}_{G \rightrightarrows X} (Y)$ of principal $G\rightrightarrows X$ bundles over $Y$ is a groupoid. $[G \rightrightarrows X]$ is then the stack which on objects is
$$ Y  \mapsto \mathrm{Bun}_{G \rightrightarrows X} (Y) $$

\subsubsection{Atlases and differentiable stacks}

If $\mathfrak X$ is a stack over $\mathbf{Man}$ then an \emph{atlas} of $\mathfrak X$ is a morphism $X \to \mathfrak X$ which is a representable submersion and an epimorphism in the category $\mathrm{St}(\mathbf{Man})$. If $X \to  \mathfrak X$ is an atlas then the fibre product $X\times_ \mathfrak X X$ has the structure of a Lie groupoid over $X$, and there is an equivalence of stacks
$$ [ X \times _\mathfrak X X \rightrightarrows X ] \simeq \mathfrak X $$
Conversely, if $G \rightrightarrows X$ is a Lie groupoid then the natural morphism $(X \rightrightarrows X) \to (G \rightrightarrows X)$ induces a morphism 
$$X \to [G \rightrightarrows X]$$
which is an atlas. 

A \emph{differentiable stack} is a stack of groupoids over $\mathbf{Man}$ for which there exists an atlas, or equivalently, a stack which is equivalent to the classifying stack of some Lie groupoid. 

\subsubsection{Some classes of differentiable stacks} \label{Some classes of differentiable stacks}

A differentiable stack $\mathfrak X$ is called \emph{\'etale} if there exists an atlas $X \to \mathfrak X$ which is a representable \'etale map, or equivalently $\mathfrak X$ is equivalent to the classifying stack of an \'etale Lie groupoid, which is a Lie groupoid whose structure maps are all \'etale.

A differentiable stack is called \emph{proper} if it is equivalent to the classifying stack of a proper Lie groupoid, which is a Lie groupoid $G \rightrightarrows X$ where the map $(s,t) : G \to X \times X$ is proper.

A differentiable stack is \emph{Deligne-Mumford}, or an \emph{orbifold}, if it is equivalent to the classifying stack of a proper \'etale Lie groupoid.

\subsubsection{Examples}

Some examples of differentiable stacks are:
\begin{itemize}
\item The sheaf $\underline X$ associated to a manifold $X$ by the Yoneda embedding. $\underline X$ is the classifying stack of the trivial Lie groupoid $ X \rightrightarrows X$.
\item The sheaf $X \mapsto \mathrm C^\infty (X)$, which is equal to $\underline {\mathbb R}$. 
\item The classifying stack $\mathrm BG$ of a Lie group $G$. $\mathrm BG(X)$ is the groupoid of principal $G$-bundles over $X$. If $\phi:X \to Y$ is a smooth map then the functor $\mathrm BG(Y) \to \mathrm BG(X)$ is given by the pullback of principal $G$-bundles along $\phi$. $\mathrm BG$ is \'etale if $G$ is a discrete group. $\mathrm BG$ is the classifying stack of the Lie groupoid $G\rightrightarrows \ast$, where $G$ is considered as a Lie groupoid over a point $\ast$. 
\item The stack $\mathrm{Vect}_n$ that maps $X$ to the groupoid of rank n real vector bundles over $X$ and isomorphims between them. If $\phi:X \to Y$ is a smooth map then the functor $\mathrm {Vect}_n(Y) \to \mathrm {Vect}_n(X)$ is given by the pullback of vector bundles along $\phi$. The associated bundle construction gives a morphism $\mathrm B GL_n \to \mathrm{Vect}_n$ that is an equivalence. 
\item The Deligne-Mumford stack associated to an orbifold. One can construct a proper \'etale groupoid from an orbifold atlas and then take its classifying stack. Up to equivalence, this does not depend on the choice of atlas.
\item The quotient stack $X//G$ associated to an action of a Lie group $G$ on $X$. A morphism from a manifold $Y$ to $X//G$ is given by a principal $G$-bundle $P$ over $Y$ and a $G$-equivariant map $P \to X$. If $G$ acts freely and properly on $X$, then $X//G$ is equivalent to (the sheaf represented by) the quotient manifold $X/G$. $X//G$ is Deligne-Mumford if $G$ is compact and acts locally freely with finite stabilisers. With respect to the trivial action of $G$ on a point $\ast$, we have $\ast // G = \mathrm B G$. $X//G$ is the classifying stack of the Lie groupoid $G\ltimes X \rightrightarrows X$ associated to the action of $G$ on $X$.
\end{itemize}

Some examples of stacks that are not differentiable are:
\begin{itemize}
\item The sheaf $X \mapsto \Omega^k (X)$, which is not representable if $k \geq 1$.
\item The stack $\mathrm{Vect}$ of all vector bundles and morphisms between them.
\item The stack of Lie algebroids $\mathcal{LA}$, see section \ref{The stack of Lie algebroids}.
\end{itemize}

\newpage

\section{Lie Algebroids}  \label{Lie Algebroids}

For the theory of Lie algebroids see \cite{Ma2}. See also \cite{Cr1} for details about representations of Lie algebroids and their pullbacks, \cite{Vai1} for the dual characterisation of Lie algebroid morphisms in terms of de Rham algebras, and  \cite{HiMa1} and \cite{St1} for the construction of fibre products of Lie algebroids.

\subsection{Lie algebroids and morphisms between them}

\subsubsection{Lie algebroids}  \label{Lie algebroids over smooth manifolds}

A \emph{Lie algebroid} over a manifold $X$ is a vector bundle $A$ over $X$ together with a Lie bracket $[,]$ on $\Gamma(A)$ and a morphism $a:A \to TX$ of vector bundles over $X$, which induces a Lie algebra morphism $\Gamma(A) \to \Gamma(TX)$, and
$$[\xi,f\nu] = a(\xi)(f)\nu + f[\xi,\nu]$$ 
for all $\xi,\nu \in \Gamma(A)$ and $f \in \mathrm C ^\infty (X)$. The map $a$ is called the \emph{anchor map} of $A$, and the last condition in the definition is called the \emph{Leibniz rule}. We'll often drop the anchor map and write $\xi(f)$ for $a(\xi)(f)$.

A \emph{morphism of Lie algebroids over $X$} is a morphism of vector bundles over $X$ which induces a Lie algebra morphism between spaces of sections and commutes with the anchor maps. We'll denote the category of Lie algebroids over $X$ by $\mathcal {LA}_X$.

A Lie algebroid is called \emph{transitive} if its anchor map is surjective. We'll denote the category of transitive Lie algebroids over $X$ by $\mathcal{LA}_X ^\mathrm{tr}$.

Complex Lie algebroids are defined similarly, by replacing real by complex vector bundles, $TX$ by $TX^ \mathbb C$, and $\mathrm C ^\infty (X)$ by $\mathrm C ^\infty (X,\mathbb C)$. We'll denote the category of complex Lie algebroids over $X$ by $\mathcal {LA}^{\mathbb C}_X$.

\subsubsection{Examples}

Some basic examples of Lie algebroids are:
\begin{itemize}

\item For any manifold $X$, the tangent bundle $TX$ is a Lie algebroid over $X$. The Lie bracket is the usual Lie bracket of vector fields, and the anchor map is the identity.

\item If $F\subset TX$ is an involutive distribution, then $F$ is a Lie algebroid with anchor map the inclusion $F \hookrightarrow TX$.

\item Finite dimensional $\mathbb R$-Lie algebras are Lie algebroids over a point.

\item A bundle $V \to X$ of (finite dimensional real) Lie algebras is a Lie algebroid with Lie bracket given pointwise and anchor map the zero map $V \to TX$. 

\item If $X$ is a complex manifold then $T^{0,1}X$ is a complex Lie algebroid with Lie bracket restricted from $TX^\mathbb C$ and anchor map the inclusion $T^{0,1}X \hookrightarrow TX^\mathbb C$.

\end{itemize}

\subsubsection{The locality of the Lie bracket}

It follows from the Leibniz rule that the Lie bracket on a Lie algebroid $A$ over $X$ is automatically local: for any sections $\xi,\nu \in \Gamma(A)$, the germ at $x \in X$ of $[\xi,\nu]$ only depends on the germs of $\xi$ and $\nu$ at $x$. It follows that for any open subset $U\subset X$ the vector bundle $A|_U$ over $U$ inherits the structure of a Lie algebroid over $U$, and the sheaf $\mathcal A = \Gamma( -, A)$ is a sheaf of Lie algebras.

\subsubsection{Base Changing Morphisms}

Let $A$ be a Lie algebroid over $X$ and $B$ be a Lie algebroid over $Y$. Consider a vector bundle morphism $(\phi,f)$:
$$\xymatrix{
A \ar[d]  \ar[r]^\phi & B \ar[d] \\
X \ar[r]_f  & Y 
}$$
This induces a map of vector bundles over $X$:
$$\xymatrix{
A\ar[dr] \ar[r] ^{\phi'} & f^\ast B   \ar[d] \\
&  X
}$$
and therefore a morphism of $\mathrm C ^\infty (X)$ modules:
$$\xymatrix{
\Gamma(A) \ar[r]^>>>>>{\phi'} & \mathrm C ^\infty (X) \otimes_{\mathrm C ^\infty (Y)} \Gamma(B)
}$$
where we have used (the inverse of) the isomorphism
\begin{align*}
\mathrm C ^\infty (X) \otimes _{\mathrm C ^\infty (Y)} \Gamma(B)  & \to \Gamma(f^\ast B)  \\
g \otimes \xi & \mapsto g (f^\ast \xi)
\end{align*}
Let $\xi,\nu \in \Gamma(A)$, and
\begin{align*}
\phi' \xi & = \sum_i g_i \otimes \xi'_i \\
\phi' \nu & = \sum_i h_i \otimes \nu'_j 
\end{align*}
for some functions $g_i,h_i \in \mathrm C ^\infty (X)$ and sections $\xi'_i,\nu'_i \in \Gamma(B)$. Then $(\phi,f)$ is a \emph{morphism of Lie algebroids} if
\begin{align*}
\phi' ( [\xi,\nu] )  &  =   \sum_{i,j} g_i h_j \otimes [\xi'_i,\nu'_j]  + \sum_j \xi(h_j) \otimes \nu'_j  -  \sum_i \nu(g_i) \otimes \xi'_i
\end{align*}
and 
$$b \circ \phi = f_\ast \circ a$$
that is, the following diagram commutes:
$$\xymatrix{
A \ar[d]_a \ar[r]^\phi &  B \ar[d]^b \\
TX \ar[r]_{f_\ast} & TY 
}$$
With this definition of morphisms, Lie algebroids form a category that we'll denote by $\mathcal{LA}$. Similarly, we have the category of transitive Lie algebroids $\mathcal {LA}^\mathrm{tr}$ and the category of complex Lie algebroids  $\mathcal {LA}^\mathbb C$. 

There is a functor $\mathcal {LA} \to \mathbf{Man}$ that maps a Lie algebroid $A$ over $X$ to the manifold $X$. The category $\mathcal{LA}_X$ (see \ref{Lie algebroids over smooth manifolds}) of Lie algebroids over $X$ is then the fibre (see \ref{Fibred categories and the Grothendieck construction}) over $X$ of this functor.

If $f: X \to Y$ is a smooth map, then the bundle map 
$$\xymatrix{
TX \ar[d]  \ar[r]^{f_\ast} & TY \ar[d] \\
X \ar[r]_f  & Y 
}$$
is a morphism of Lie algebroids. This gives a fully faithful embedding
$$\mathbf{Man} \hookrightarrow \mathcal{LA}$$
of the category of smooth manifolds into the category of Lie algebroids.

\subsection{The de Rham functor}

\subsubsection{The de Rham algebra of a Lie algebroid}

If $A$ is a Lie algebroid over $X$ then there is a degree 1 differential 
$$\mathrm d_A : \Gamma \left({\bigwedge}^k A^\ast \right) \to
\Gamma \left({\bigwedge}^{k+1} A^\ast \right) $$
on the graded algebra $\Gamma \left( {\bigwedge}^\bullet A^\ast \right)$ that makes it into a DGA (differential graded algebra). $\mathrm d_A$ is determined by
\begin{align*}
\mathrm df (\xi) & = \xi(r)  \\
\mathrm d\alpha (\xi,\nu) & =  \xi (\alpha(\nu)) -  \nu(\alpha(\xi)) - \alpha([\xi,\nu])
\end{align*}
where $f \in \mathrm C^\infty (X)$, $\xi,\nu \in \Gamma(A)$ and $\alpha \in \Gamma(A^\ast)$, and the graded Leibniz rule: 
$$\mathrm d (\beta \wedge \gamma) = \mathrm d \beta \wedge \gamma + (-1)^{|\beta|} \beta \wedge \mathrm d \gamma$$
where $\beta,\gamma \in \Gamma \left( {\bigwedge}^\bullet A^\ast \right)$. We'll write $\Omega^\bullet(A)$ for this DGA, and call it the \emph{de Rham algebra of $A$}. The differential $\mathrm d_A$ is local, so it induces a differential on the sheaf of graded algebras ${\bigwedge}^\bullet _{\mathrm C^\infty _X} \mathcal A^\ast$, where $\mathcal A^\ast$ is the sheaf of sections of $A^\ast$. We'll write $\Omega^\bullet _A$ for this sheaf of DGAs, for which we have
$$\Gamma \left( \Omega^\bullet _A \right)  =  \Omega^\bullet (A)$$

For any manifold $X$, the de Rham algebra $\Omega^\bullet(TX)$ of the Lie algebroid $TX$ is the usual de Rham complex of differential forms on $X$.

\subsubsection{The de Rham functor}

If $A$ is a Lie algebroid over $X$ and $B$ a Lie algebroid over $Y$, then given a morphism of vector bundles
$$\xymatrix{
A \ar[d] \ar[r]^\phi & B \ar[d]  \\
X \ar[r]_f & Y 
}$$
there is the associated morphism of graded algebras
$$\phi^\ast :  \Omega^\bullet (B) \to \Omega^\bullet(A)$$
It is shown in \cite{Vai1} that $(\phi,f)$ is a morphism of Lie algebroids if and only if $\phi^\ast$ is a morphism of DGAs, i.e. if we have
$$\mathrm d_A \circ \phi^\ast = \phi^\ast \circ \mathrm d_B$$
Therefore there is a faithful (contravariant) functor
$$ \Omega^\bullet (-) : \mathcal {LA}   \to \mathbb R \text{-DGA} $$
that maps a Lie algebroid $A$ to its de Rham algebra $\Omega^\bullet(A)$. We'll call this the \emph{de Rham functor}.

The de Rham functor is local in the sense that a morphism of Lie algebroids $(\phi,f) : A \to B$ as above gives a morphism of sheaves of algebras
$$ \Omega_B^\bullet \to f_\ast \Omega^\bullet_A$$
because for any open $U \subset Y$ we have Lie algebroids $B|_U$ and $A|_{f^{-1}(U)}$.

\subsection{Representations and Cohomology}

\subsubsection{Representations of Lie algebroids}

If $A$ is a Lie algebroid over a manifold $X$ then a \emph{representation} of $A$ is a smooth vector bundle $E$ over $X$ together with an $\mathbb R$-bilinear map
\begin{align*}
\nabla : \Gamma(A) \times \Gamma(E)  & \to \Gamma(E) \\
(\xi,e) & \mapsto \nabla_\xi e 
\end{align*}
such that
\begin{align*}
\nabla_{f\xi} e & = f \nabla_\xi e \\
\nabla_\xi (fe) & = \xi(f)e + f\nabla_\xi e \\
\nabla_{[\xi,\nu]} e & = \nabla_\xi \nabla_\nu e - \nabla_\nu \nabla_\xi e 
\end{align*}
for all $\xi,\nu \in \Gamma(A)$, $f \in \mathrm C ^\infty (X)$, $e \in \Gamma(E)$. A morphism between two representations of $A$ is a morphism of vector bundles over $X$ that commutes with the actions of $A$. Representations are also called `flat $A$-connections'. We'll denote the category of representations of a Lie algebroid $A$ by $\mathrm{Rep}_A$.

It follows from the definition that representations of Lie algebroids are automatically local: the germ of $\nabla_\xi e$ at $x \in X$ only depends on the value of $\xi$ at $x$ and the germ of $e$ at $x$. The map $\nabla$ therefore determines a map of sheaves of $\mathbb R$-modules
$$\nabla : \mathcal A \times \mathcal E \to \mathcal E$$
where $\mathcal A$ and $\mathcal E$ are the sheaves of sections of $A$ and $E$. In particular, a representation can be restricted to an open set $U \subset X$ to give a representation
$$\nabla | _ U : \Gamma (A |_U) \times \Gamma(E|_U) \to \Gamma(E|_U)$$
of $A|_U$ on $E|_U$. 

\subsubsection{The de Rham complex with coefficients in a representation}

A representation $\nabla$ of $A$ on $E$ determines a map
$$\mathrm d _ {A,\nabla} : \Gamma(E) \to \Gamma(A^\ast \otimes E)$$
by the formula
$$ \mathrm d_{A,\nabla} (e)  (\xi) = \nabla_\xi e $$
for $e \in \Gamma(E)$ and $\xi \in \Gamma(A)$. $\mathrm d_{A,\nabla}$ satisfies
$$ \mathrm d_{A,\nabla} (fe) = \mathrm d_A (f) e + f \otimes \mathrm d_{A,\nabla} (e) $$
for all $e \in \Gamma(E)$ and $f \in \mathrm C^\infty (X)$, and extends to a differential 
$$ \mathrm d_{A,\nabla} :  \Gamma \left ({\bigwedge}^k A^\ast \otimes E \right) 
\to \Gamma \left ({\bigwedge}^{k+1} A^\ast \otimes E \right)$$
by the formula
$$ \mathrm d_{A,\nabla} ( \beta \otimes e) = \mathrm d_A \beta \otimes e + (-1)^{|\beta|} \beta \wedge \mathrm d_{A,\nabla} (e)$$
This gives a complex of vector spaces
$$ \Gamma \left ({\bigwedge}^\bullet A^\ast \otimes E \right) \cong \Omega^\bullet (A) \otimes _{\mathrm C^\infty (X)} \Gamma(E)$$
which is a DG-module over $\Omega^\bullet (A)$. We'll write $\Omega^\bullet (A,E,\nabla)$ (or just $\Omega^\bullet (A,E)$ when $\nabla$ is understood) for this complex, and refer to it as the \emph{de Rham complex of $A$ with coefficients in $E$}.

The differential $\mathrm d_{A,\nabla}$ associated to a representation $\nabla$ is local, so there is a complex of sheaves
$$\Omega^\bullet _A \otimes _{\mathrm C^\infty _X} \mathcal E$$
and we have
$$\Gamma \left( \Omega^\bullet \otimes _{\mathrm C^\infty _X} \mathcal E \right) 
= \Omega^\bullet (A) \otimes _{\mathrm C^\infty (X)} \Gamma(E)
$$

\subsubsection{Lie algebroid Cohomology}

If $A$ is a Lie algebroid then the cohomology of the de Rham algebra $\Omega^\bullet (A)$ of $A$ is called the \emph{Lie algebroid cohomology of $A$}, and denoted $\mathrm H^\bullet (A)$. Since $\Omega^\bullet (A)$ is a DGA, $\mathrm H^\bullet (A)$ is a graded algebra. Lie algebroid cohomology is a contravariant functor
$$\mathrm H ^\bullet (-) : \mathcal{LA} \to \mathbb R \text{-GrAlg}$$
where $\mathbb R \text{-GrAlg}$ is the category of graded algebras over $\mathbb R$, because it is given by the composition of the functors 
$$\xymatrix{
\mathcal{LA} \ar[r]^>>>>>>{\text{ }\Omega^\bullet (-)} & \mathbb R \text{-DGA} \ar[r]^-{\mathrm H ^\bullet (-)} & \mathbb R \text{-GrAlg}
}$$

If $\nabla$ is a representation of a Lie algebroid $A$ on $E$ then the \emph{cohomology of $A$ with coefficients in $E$}, denoted $ \mathrm H ^\bullet (A,E,\nabla)$ or just $\mathrm H^\bullet(A,E)$, is defined to be the cohomology of the complex  $\Omega^\bullet (A) \otimes _{\mathrm C^\infty (X)} \Gamma(E)$. $\mathrm H^\bullet(A,E)$ is a graded module over $\mathrm H^\bullet (A)$.

\subsection{Pullbacks of Lie algebroids}

\subsubsection{The pullback construction} \label{The pullback construction}

If $\phi : X \to Y$ is a smooth map and $E \to Y$ a vector bundle then we have the pullback bundle $\phi^\ast E \to X$, with total space the fibre product
$$\xymatrix{
X \times _Y E \ar[r] \ar[d] & E \ar[d] \\
X \ar[r]_\phi & Y
}$$
and fibre $E_{\phi(x)}$ over $x \in X$.

Let $\phi:X \to Y$ be a smooth map and $\pi:A \to Y$ be a Lie algebroid. Then, if it exists, the \emph{pullback Lie algebroid} $\phi^!A$ is defined as:
\begin{align*}
\phi^!A & = \{ (v,w) \in TX \times A \vert \phi_\ast(v) = a(w) \} 
\end{align*}
or equivalently as the kernel of the map
\begin{align*}
( \phi_\ast \oplus  - \phi^\ast a ): TX\oplus \phi^\ast A & \to \phi^\ast TY \\
(v, (x,\xi) ) & \mapsto ( x, \phi_\ast(v) - a(\xi) )
\end{align*}
Therefore $\phi^! A$ exists whenever $(\phi_\ast \oplus  - \phi^\ast a )$ has constant rank or if $\phi_\ast$ and $a$ are transverse. In particular, $\phi^! A$ exists if $\phi$ is a submersion or if $A$ is transitive.

Assuming $\phi^! A$ exists, its vector bundle projection is the restriction of that of $TX\oplus \phi^\ast A$, and its anchor is the projection onto $TX$. Considering $\phi^!A$ as a subbundle of $TX\oplus\phi^\ast A$ and using the isomorphism of $C^\infty (X)$ modules
\begin{align*}
C^\infty(X) \otimes _{C^\infty (Y)} \Gamma(Y,A) & \to \Gamma(X,\phi^\ast A)    \\
f \otimes \xi & \mapsto f \phi^\ast (\xi)
\end{align*}
we can view sections of $\phi^! A$ as elements $v \oplus \sum_i (f_i \otimes \xi_i)$ of $\Gamma(TX)\oplus (C^\infty(X) \otimes _{C^\infty (Y)} \Gamma(Y,A))$ satisfying the condition:
\begin{align*}
\phi_\ast (v_x) & = \sum_i f_i(x) a( ({\xi_i})_{\phi(x)}) 
\end{align*}
for all $x \in X$, and the Lie bracket on sections of $\phi^! A$ is then defined as:
\begin{align*}
\left[ v \oplus \sum_i (f_i \otimes \xi_i) , v' \oplus \sum _j (f_j' \otimes \xi_j') \right] & = [v,v'] \oplus \left( \sum_{i,j} f_i f_j' \otimes [\xi_i,\xi_j'] 
+ \sum_j v(f_j') \otimes \xi_j'  
-  \sum_i v'(f_i) \otimes \xi_i \right)
\end{align*}

If $f:X \to Y$ is a submersion, then by choosing a splitting of the exact sequence
$$\xymatrix{
0 \ar[r]  &  \mathrm{Ker} f_\ast  \ar[r] &  TX  \ar[r] ^{f_\ast}  & f^\ast TY \ar[r] & 0
}$$
of vector bundles over $X$, one can check that there is an isomorphism
$$ f^! A \cong \mathrm{Ker} f_\ast \oplus f^\ast A $$
of vector bundles over $X$. In particular, if $f:X \to Y$ has relative dimension $k$, then we have
$$\text{rank} \left( f^! A \right) = \text{rank} \left( A \right) + k$$

\subsubsection{The universal property of pullbacks} \label{cartesian property}

If $\phi^!A$ is the pullback of a Lie algebroid along a smooth map $\phi:X \to Y$ then there is a Lie algebroid morphism $\phi^\sharp : \phi^! A \to A$ covering $\phi$
$$\xymatrix{
\phi^! A \ar[d] \ar[r] ^ {\phi^\sharp} & A \ar[d] \\
X \ar[r]_\phi & Y 
}$$
which is given by projecting onto the second factor:
\begin{align*}
\phi^\sharp: \phi^! A & \to A \\
( v, \xi ) & \mapsto \xi
\end{align*}
This morphism satisfies the usual universal property for pullbacks: if we have a morphism of Lie algebroids:
$$\xymatrix{
A' \ar[d] \ar[r] ^ {\tilde \psi} & A \ar[d] \\
Z \ar[r]_\psi & Y 
}$$
and $\psi = \phi \circ \psi'$ for some smooth map $\psi'$, giving a commutative triangle:
$$\xymatrix{
Z \ar[dr]_{\psi} \ar[r]^{\psi'} & X \ar[d]^{\phi} \\
& Y
}$$
then $\tilde \psi$ factors uniquely into $\phi^\sharp \circ \tilde \psi'$ for some Lie algebroid morphism $\tilde \psi'$ covering $\psi'$:
$$\xymatrix{
A' \ar@/^1.5pc/[rr]^{\tilde \psi}   \ar[d] \ar[r] ^ {\tilde \psi'} & \phi^! A \ar[d] \ar[r]^{\phi^\sharp}  & A \ar[d] \\
Z \ar@/_1.5pc/[rr]_\psi \ar[r]_{\psi'} & X \ar[r]_{\phi} & Y 
}$$
The morphism $\tilde \psi'$ is given by the formula
\begin{align*}
\tilde \psi' :  \xi & \mapsto \left( \psi'_\ast (a'(\xi)) ,  \tilde \psi (\xi)  \right)
\end{align*}
In short, the morphism $\phi^\sharp : \phi^! A \to A$ is cartesian with respect to the functor $\mathcal {LA} \to \mathbf{Man}$ which maps a Lie algebroid to its base manifold.

\subsubsection{Pullbacks along \'etale maps} \label{etale pullbacks}

If $\phi: X \to Y$ is \'etale and $A$ a Lie algebroid over $Y$ then there is a vector bundle isomorphism:
\begin{align*}
\phi^\ast A & \to \phi^! A \\
(x,w) & \mapsto ((\phi_\ast \vert_x ^{-1} \circ a)(w),w)
\end{align*}
This induces a Lie algebroid structure on $\phi^\ast A$ which is determined by the statement that for any open set $U \subset X$ for which $\phi\vert_U$ is a diffeomorphism onto its image, the vector bundle isomorphism:
$$\xymatrix{
(\phi^\ast A) \vert _U  \ar[r]^{\mathrm{pr}_A} \ar[d] & A\vert_{\phi(U)} \ar[d] \\
U \ar[r] _{\phi|_U} & \phi(U) 
}$$ 
is an isomorphism of Lie algebroids. If we denote the sheaf of sections of $A$ by $\mathcal A$, then this is equivalent to the statement that the canonical isomorphism of sheaves of vector spaces:
\begin{align*}
\Gamma(-,\phi^\ast A) & \to \phi^\ast \mathcal A 
\end{align*}
is an isomorphism of sheaves of Lie algebras, and that the anchor map on $\phi^\ast A$, denoted $a^\ast$, is given by:
\begin{align*}
a^\ast : \phi^\ast A & \to TX \\
(x,w) & \mapsto (\phi_\ast\vert_x ^{-1} \circ a)(w)
\end{align*}

\subsubsection{Pullbacks of representations}  \label{Pullbacks of representations}

Let $A$ be a Lie algebroid over $X$, $\nabla$ a representation of $A$ on a vector bundle $E$ over $X$, and $f:Y \to X$ a submersion (or more generally a smooth map $f$ such that $f^! A$ exists). Then there is a representation
$$f^! \nabla : \Gamma(f^! A) \times \Gamma(f^\ast E) \to \Gamma(f^\ast E)$$
of $f^!A$ on $f^\ast E$ defined by
$$(f^! \nabla)_{v \oplus (g \otimes \xi)} (g'\otimes e) = v(g') \otimes e + gg' \otimes \nabla_\xi e$$
where $v \oplus (g \otimes \xi) \in \Gamma(f^!A)$ and $g'\otimes e \in \Gamma(f^\ast E)$ as in \ref{The pullback construction}.

Combining the map $\Omega^\bullet(A) \to \Omega^\bullet(f^! A)$
associated to the pullback morphism $f^! A \to A$ by the de Rham functor, with the pull back map $\Gamma(E) \to \Gamma(f^\ast E)$, gives a map
$$\Omega^\bullet(A,E) \to \Omega^\bullet(f^! A, f^\ast E)$$
which is a morphism of cochain complexes \cite{Cr1}.

\subsection{Fibre products of Lie algebroids}

Given Lie algebroids $A_1, A_2,B$ over manifolds $X_1,X_2,Y$, and a diagram in $\mathcal {LA}$ of the form
$$\xymatrix{
 &  &  & X_2  \ar[d]^{f_1}   \\
&  A_2 \ar[urr]  \ar[d]_<<<{\phi_2} & X_1 \ar[r]_<<{f_2} & Y \\
A_1 \ar[urr]  \ar[r]_{\phi_1} & B  \ar[urr]
}$$
the fibre product 
$$\xymatrix{
 &  & X_1 \times _ Y X_2 \ar[d] \ar[r]  & X_2  \ar[d]^{f_1}   \\
A_1 \times _B A_2 \ar[d] \ar[r] \ar[urr] &  A_2 \ar[urr]  \ar[d]_<<<{\phi_2} & X_1 \ar[r]_<<<{f_2} & Y \\
A_1 \ar[urr]  \ar[r]_{\phi_1} & B  \ar[urr]
}$$
exists in the category $\mathcal {LA}$ whenever the fibre products $A_1 \times _B A_2$ and $X_1 \times _Y X_2$ exist as manifolds. In particular, the fibre product exists if either both $\phi_1$ and $f_1$ or both $\phi_2$ and $f_2$ are submersions (such morphisms are called Lie algebroid fibrations in \cite{Ma2}).

The pullback of a Lie algebroid $A \to Y$ along a map $f:X \to Y$ can be identified with the fibre product of the morphisms
$$\xymatrix{
 &  &  & Y  \ar[d]^{\mathrm{id}}   \\
&  A \ar[urr]  \ar[d]_<<<{a} & X \ar[r]_<<{f} & Y \\
TX \ar[urr]  \ar[r]_{f_\ast} & TY  \ar[urr]
}$$

\newpage

\section{The stack of Lie algebroids}  \label{The stack of Lie algebroids}

In this section we show that the categories $\mathcal{LA}$ of Lie algebroids, and $\mathrm{Rep}$ of representations of Lie algebroids, are stacks over $\mathbf{Man}_\text{sub}$ with respect to the submersion topology. We prove this in two steps: Theorem \ref{LA is a stack} is the statement that $\mathcal{LA}$ satisfies descent for open covers, and Theorem \ref{submersive descent} is the statement that $\mathcal{LA}$ satisfies descent for surjective submersions.

\subsection{The category of Lie algebroids}

\subsubsection{$\mathcal{LA}$ as a fibered category}

There is the functor from $\mathcal{LA}$ to $\mathbf{Man}$ which takes a Lie algebroid to the base manifold, therefore $\mathcal{LA}$ is a category over $\mathbf{Man}$:
$$\xymatrix{
\mathcal{LA} \ar[d] \\
\mathbf{Man}
}$$
As described in \ref{cartesian property}, the morphism
$$\xymatrix{
\phi^! A \ar[d] \ar[r]^{\phi^\sharp} &  A  \ar[d]   \\
X \ar[r]_\phi & Y 
}$$
induced by the pullback of a Lie algebroid $A \to Y$ along a map $\phi:X \to Y$ is cartesian with respect to this functor. Since pullbacks of transitive Lie algebroids always exist we can restrict the functor to $\mathcal{LA}^\text{tr}$ and we have:
\begin{Proposition}
The functor 
$$\xymatrix{
\mathcal{LA}^\mathrm{tr} \ar[d] \\
\mathbf{Man}
}$$
is a Grothendieck fibration.
\end{Proposition}
In the general case we must restrict to the subcategory $\mathbf{Man}_{\text{sub}}$ of $\mathbf{Man}$ consisting of manifolds and submersions:
\begin{Proposition}
The functor 
$$\xymatrix{
\mathcal{LA} \ar[d] \\
\mathbf{Man}_\mathrm{sub}
}$$
is a Grothendieck fibration.
\end{Proposition}

We'll mostly work with $\mathcal{LA} \to \mathbf{Man}_\text{sub}$ and therefore restrict to $\mathbf{Man}_{\text{sub}}$, but all of our results apply to $\mathcal{LA}^\text{tr} \to \mathbf{Man}$ as well.

\subsubsection{$\mathcal{LA}$ as a weak presheaf of categories}  \label{LA as a weak presheaf of categories}

The pullback construction gives a `choice of pullbacks' in the sense of fibered categories, associated to which is a weak presheaf of categories 
$$\mathcal{LA} : \mathbf{Man}_\text{sub} \to \mathbf{Cat}$$
which we'll also denote by $\mathcal{LA}$. Explicitly, $\mathcal{LA}$ maps a manifold to its category of Lie algebroids:
$$X \mapsto \mathcal{LA}_X$$
and maps a submersion to the associated pullback functor:
$$ ( \phi: X \to Y )  \mapsto ( \phi^! : \mathcal{LA}_Y \to \mathcal{LA}_X ) $$
If $\rho: A \to B$ is a morphism of Lie algebroids over $Y$, then the functor $\phi^!$ maps $\rho$ to the morphism $\phi^! \rho : \phi^! A \to \phi^! B$ determined by the universal property of the pullback $\phi^! B$ and the morphism
$$\xymatrix{
\phi^! A \ar[d]  \ar[r]^{\rho \circ \phi^\sharp} & B \ar[d] \\
X \ar[r]_\phi & Y
}$$
and $\phi^! \rho$ is therefore given by the formula
$$(v,\xi)  \mapsto  (v, \rho(\xi))$$
For any composable pair of submersions:
$$\xymatrix{
X \ar[r]^{\phi_1} & Y \ar[r]^{\phi_2} & Z 
}$$
there is a natural isomorphism:
\begin{align*}
 c_{\phi_2,\phi_1}:(\phi_2\phi_1)^! & \to \phi_1^! \phi_2^! 
\end{align*}
such that for any composable triple of submersions:
$$\xymatrix{
X \ar[r]^{\phi_1} & Y \ar[r]^{\phi_2} & Z  \ar[r]^{\phi_3} & W 
}$$
the following diagram commutes:
$$\xymatrix{
(\phi_3\phi_2\phi_1)^! \ar[d]_{c_{\phi_3,\phi_2\phi_1}}  \ar[r]^{c_{\phi_3\phi_2,\phi_1}} &
 \phi_1^!(\phi_3\phi_2)^! \ar[d]^{\mathrm{id}_{\phi_1^!} \bullet c_{\phi_3,\phi_2}} \\
(\phi_2\phi_1)^!\phi_3^! \ar[r]_>>>>>{c_{\phi_2,\phi_1} \bullet \mathrm{id}_{\phi_3^!}} &
 \phi_1^!\phi_2^!\phi_3^! 
}$$
These natural isomorphisms are given by the formula:
\begin{align*}
c_{\phi_2,\phi_1}(A):(\phi_2\phi_1)^! A & \to \phi_1^! \phi_2^! A \\
(v,w) & \mapsto (v, ( (\phi_1)_\ast v, w) )
\end{align*}

\subsubsection{Restriction to \'etale maps and open inclusions} \label{restriction to etale maps}

In \ref{etale pullbacks} we showed that if $\phi:X \to Y$ is an \'etale map and $A \to Y$ is a Lie algebroid then there is a canonical isomorphism of vector bundles
$$\phi^\ast A \cong \phi^! A$$
Moreover, if $\rho:A \to B$ is a morphism of Lie algebroids over $Y$ then with respect to these canonical isomorphisms $\phi^! \rho$ corresponds to the map
$$\phi^\ast \rho : \phi^\ast A \to \phi^\ast B$$
It follows that there is a weak presheaf
$$\mathcal{LA}' : \mathbf{Man}_\text{\'et} \to \mathbf{Cat}$$
that agrees with $\mathcal{LA}$ on objects:
$$\mathcal{LA}' : X \mapsto \mathcal{LA}_X$$
and maps an \'etale map to the vector bundle pullback functor
$$\mathcal{LA}' : ( \phi: X \to Y) \mapsto ( \phi^\ast: \mathcal{LA}_Y \to \mathcal{LA}_X )$$
We then have:

\begin{Proposition} \label{equivalence over etale category}
The restriction of $\mathcal{LA}$ to the subcategory $\mathbf{Man}_\text{\emph{\'et}}$ of $\mathbf{Man}_\mathrm{sub}$ is equivalent to $\mathcal{LA}'$.
\end{Proposition}

\begin{proof}
We just need to check that the isomorphisms 
\begin{align*}
F_A : \phi^\ast A & \to \phi^! A \\
( x, \xi )  & \mapsto ( (\phi_\ast) |_x^{-1} (a(\xi)) , \xi) 
\end{align*}
are natural for all \'etale maps $X \to Y$ and Lie algebroids $A$ over $Y$. Let $\rho:A \to B$ be a morphism of Lie algebroids over $Y$, and $(x,\xi) \in \phi^\ast A$, then we have:
\begin{align*}
( \phi^! \rho \circ F_A )  (x,\xi)  & = \phi^! \rho  \left ( (\phi_\ast) |_x^{-1} (a(\xi)) , \xi \right) \\
& = \left ( (\phi_\ast) |_x^{-1} (a(\xi)) , \rho(\xi) \right) \\
\\
(F_A \circ \phi^\ast \rho) (x,\xi) & = F_A \left( x, \rho(\xi) \right)  \\
& =  \left ( (\phi_\ast) |_x^{-1} (a(\rho(\xi))) , \rho(\xi) \right) \\
& =  \left ( (\phi_\ast) |_x^{-1} (a(\xi)) , \rho(\xi) \right) 
\end{align*}
\end{proof}

If $X$ is a manifold and $U\subset X$ an open subset then restricting Lie algebroids from $X$ to $U$ gives a functor
$$|_U : \mathcal{LA}_X \to \mathcal{LA}_U$$
This gives a \emph{strict} presheaf:
$$\mathcal{LA}_{\text{Open}_X} : \text{Open}_X \to \mathbf{Cat}$$
We can consider the category $\text{Open}_X$ as a subcategory of $\mathbf{Man}_\text{\'et}$ and therefore also of $\mathbf{Man}_\text{sub}$, then we have:

\begin{Proposition} \label{equivalence over open category}
For any manifold $X$, the restrictions of $\mathcal{LA}$ and $\mathcal{LA}'$ to $\text{\emph{Open}}_X$ are equivalent to $\mathcal{LA}_{\text{Open}_X}$.
\end{Proposition}
\begin{proof}
By Proposition \ref{equivalence over etale category} it suffices to prove that the restriction of $\mathcal{LA}'$ to $\text{Open}_X$ is equivalent to $\mathcal{LA}_{\text{Open}_X}$, which follows from the sheaf-theoretic description in \ref{etale pullbacks} of the Lie algebroid structure on the pullback of a Lie algebroid along an \'etale map.
\end{proof}

\subsubsection{Lie algebroids vs vector bundles}   \label{Lie algebroids vs vector bundles}

If $X$ is a manifold then we have a pair of functors:
\begin{align*}
\mathcal{LA}_X & \to \mathrm{Vect}_X \\
\mathrm{Vect}_X & \to \mathcal{LA}_X 
\end{align*}
The first is the functor which forgets the Lie algebroid structure (the Lie bracket and anchor map) and maps a Lie algebroid to its underlying vector bundle, the second equips a vector bundle with the `zero' Lie algebroid structure, i.e. the zero Lie bracket and zero anchor map. However, in general Lie algebroid pullbacks are not isomorphic to vector bundle pullbacks and therefore these functors do not give us morphisms of weak presheaves / fibered categories over $\mathbf{Man}_\text{sub}$. One result of this will be that a Lie algebroid over a stack is not a vector bundle in the usual sense, see \ref{Lie algebroids over stacks are not vector bundles}.

From \ref{restriction to etale maps} we have that the functors $\mathcal{LA}_X  \to \mathrm{Vect}_X$  and $\mathrm{Vect}_X  \to \mathcal{LA}_X$ above do give morphisms of stacks over $\mathbf{Man}_\text{\'et}$. This will result in the fact that over \'etale stacks, Lie algebroids do have a well defined rank as vector bundles, see \ref{Lie algebroids over etale stacks}.

\subsection{Descent}

\begin{Theorem} \label{LA is a stack}
$\mathcal{LA}_{\text{\emph{Open}}_X}$ is a stack over $\text{\emph{Open}}_X$ with respect to the open cover topology.
\end{Theorem}

\begin{proof}
We first show that $\mathcal{LA}_{\text{\emph{Open}}_X}$ is a prestack. Let $U \in \text{Open}_X$ and $A,B \in \mathcal{LA}_U$. Then we must show that the presheaf of sets
\begin{align*}
  \mathrm{Hom}_U : \text{Open}_U & \to \mathbf{Set} \\
V & \mapsto \mathrm{Hom}_{\mathcal{LA}_V} (A |_V , B |_V)
\end{align*} 
is a sheaf. Let $\{U_i\}_{i \in I}$ be an open cover of $U$ and $\{\rho_i : A|_{U_i} \to B|_{U_i} \}_{i \in I}$ a family of morphisms that agree on intersections. As morphisms of Lie algebroids are in particular morphisms of vector bundles the $\rho_i$'s glue into a unique vector bundle morphism $\rho: A \to B$ such that $\rho|_{U_i}=\rho_i$ for all $i \in I$. We must show that $\rho$ is in fact a morphism of Lie algebroids. For each $i\in I$ we have that $\rho_i$ is a morphism of Lie algebroids and therefore
$$b|_{U_i} = (\rho_i \circ a|_{U_i})=(\rho|_{U_i} \circ a|_{U_i})=(\rho \circ a)|_{U_i}$$
and so
$$b = \rho \circ a$$
Associated to $\rho$ we have the morphism of sheaves of vector spaces
$$\mathcal A \equiv \Gamma ( - , A)   \to   \Gamma ( - , B) \equiv \mathcal B$$
which can be obtained by gluing the morphisms of sheaves $\mathcal A |_{U_i} \to \mathcal B|_{U_i}$ associated to the vector bundle maps $\rho_i$. Since each $\rho_i$ is a morphism of Lie algebroids we have that each sheaf morphism $\mathcal A |_{U_i} \to \mathcal B|_{U_i}$ is a morphism of sheaves of Lie algebras, and therefore so is $\mathcal A \to \mathcal B$. In particular, the map 
$$\Gamma(A) \to \Gamma(B)$$
induced by $\rho$ is a morphism of Lie algebras. Therefore $\rho:A \to B$ is a morphism of Lie algebroids.

It remains to show that $\mathcal{LA}_{\text{\emph{Open}}_X}$ is in fact a stack - that is, that we can glue objects. Let $U \in \text{Open}_X$ and $\{U_i \}_{i \in I}$ be an open cover of $U$, and let $U_{ij} = U_i \cap U_j$ and $U_{ijk}=U_i \cap U_j \cap U_k$. Consider a descent datum $(\{A_i\}_{i \in I}, \{ \theta_{ij}\}_{i,j \in I})$:
\begin{align*}
A_i  & \in \mathcal{LA}_{U_i} \\
\theta_{ij} & \in \mathrm{Isom}_{\mathcal{LA}_{U_{ij}}} ( A_j \vert_{U_{ij}} ,  A_i \vert _{U_{ij}} )
\end{align*}
such that
$$(\theta _ {ij} \vert_{U_{ijk}}) \circ (\theta _ {jk}\vert_{U_{ijk}}) = \theta _{ik}\vert_{U_{ijk}}$$
for all $i,j,k \in I$. Denote the vector bundle projections and anchors of the $A_i$'s by $\pi:A_i \to U_i$ and $a_i:A_i \to TU_i$ respectively. Since $ \{ \theta_{ij}\}_{i,j \in I}$  is in particular a family of vector bundle isomorphisms satisfying the cocycle conditions we have the vector bundle 
$$A \equiv (\bigsqcup_{i \in I} A_i) / \sim$$ 
over $U$, where if we represent an element of $\bigsqcup_{i \in I} A_i$ by a pair $(i,\xi)$ for $i \in I$ and $\xi \in A_i$, then the equivalence relation is given by $(i,\xi) \sim (j,\theta_{ij}(\xi))$. Denote the equivalence class of an element $(i,\xi)$ by $[ (i,\xi) ]$. The vector bundle projection $\pi: A \to U$ is the map $\pi [(i,\xi)] = \pi_i(\xi)$. For each $i \in I$ we have a vector bundle isomorphism 
\begin{align*}
\psi_i:A_i  &  \to A \vert_{U_i} \\
\psi_i (\xi) & =  [(i,\xi)]
\end{align*}
and these isomorphisms satisfy 
$$\psi_i^{-1}\vert_{U_{ij}} \circ \psi_j\vert_{U_{ij}} = \theta_{ij}$$
For each $i \in I$ we can pushforward the Lie algebroid structure from $A_i$ to $A\vert_{U_i}$ using the vector bundle isomorphism $\psi_i$. Explicitly, we define 
$$\widetilde a_i \equiv a_i \circ (\psi_i)^{-1} : A\vert_{U_i} \to TX\vert_{U_i}$$ 
and for any two sections $\nu,\nu' \in \Gamma(U_i,A)$ we define 
$$\widetilde{[\nu,\nu']}_i \equiv \psi_i \left( [\psi_i^{-1}(\nu),\psi_i ^{-1}(\nu')] \right) $$
We need to show that the Lie algebroid structures defined this way agree on the intersections $U_{ij}$. Firstly, over $U_{ij}$ we have:
\begin{align*}
\widetilde a_i \vert_{U_{ij}} & = (a_i \circ \psi_i ^{-1})\vert_{U_{ij}} \\
& = a_i \vert_{U_{ij}} \circ \psi_i ^{-1} \vert_{U_{ij}} \\
& = a_i \vert_{U_{ij}} \circ \theta_{ij} \circ \psi_j ^{-1} \vert_{U_{ij}} \\
& = a_j \circ \psi_j ^{-1} \vert_{U_{ij}} \\
& = \widetilde a_j  \vert_{U_{ij}} 
\end{align*}
where between the third and fourth lines we have used the fact that $\theta_{ij}$ is a Lie algebroid morphism so commutes with the anchor maps. The anchor maps $\widetilde a_i$ therefore glue together to form a vector bundle map 
$$\widetilde a : A \to TU$$
Let $\nu,\nu' \in \Gamma(A)$, and put $\nu_i = \nu|_{U_i}$ and $\nu_{ij}=\nu|_{U_{ij}}$, similarly for $\nu'$, then:
\begin{align*}
\left( \widetilde{ [\nu_i,\nu_i'] }_i   \right) \vert_{U_{ij}} & =  \left(\psi_i [\psi_i^{-1}(\nu_i),\psi_i ^{-1}(\nu_i')]     \right)\vert_{U_{ij}} \\
& =  \psi_i\vert_{U_{ij}} \left[ (\psi_i^{-1}\vert_{U_{ij}})(\nu_{ij}),(\psi_i ^{-1}\vert_{U_{ij}})(\nu'_{ij}) \right]  \\
& =   (\psi_i\vert_{U_{ij}} \circ \theta_{ij} \circ \theta_{ji}) \left[ (\psi_i^{-1}\vert_{U_{ij}})(\nu_{ij}),(\psi_i ^{-1}\vert_{U_{ij}})(\nu'_{ij}) \right]      \\
& =   (\psi_i\vert_{U_{ij}} \circ \theta_{ij}) \left[ (\theta_{ji} \circ \psi_i^{-1}\vert_{U_{ij}})(\nu_{ij}), (\theta_{ji} \circ \psi_i ^{-1}\vert_{U_{ij}})(\nu'_{ij}) \right]     \\
& =   (\psi_j\vert_{U_{ij}}) \left[ (\psi_j^{-1}\vert_{U_{ij}})(\nu_{ij}), (\psi_j ^{-1}\vert_{U_{ij}})(\nu'_{ij}) \right]    \\
& = \left( \widetilde{ \left[ \nu_j,\nu'_j \right] }_j   \right)   \vert _ {U_{ij}} 
\end{align*}
where we have used the fact that $\theta_{ij}$ is a Lie algebroid morphism and so is a Lie algebra morphism on sections. There is therefore a Lie bracket on $\Gamma(A)$ defined by 
$$\left( \left [  \nu,\nu' \right]  \right) |_{U_i} = \widetilde{[\nu _i ,\nu' _i]}_i $$
and the Leibniz identity holds because it holds over each $U_i$. We therefore have a Lie algebroid $A$ over $U$ and by construction the maps $\psi_i : A_i \to A|_{U_i}$ are isomorphisms of Lie algebroids and satisfy $\psi_i^{-1} \circ \psi_j = \theta_{ij}$ over $U_{ij}$. 
\end{proof}

\begin{Corollary} \label{LA is a stack 2}
$\mathcal{LA}$ is a stack over $\mathbf{Man}_\text{\emph{sub}}$ with respect to the open cover and \'etale topologies.
\end{Corollary}

\begin{proof}
This follows from Theorem \ref{LA is a stack}, the fact that the open cover and \'etale topologies are equivalent over $\mathbf{Man}_\text{sub}$, and that with respect to the open cover topology it is sufficient to check descent for open covers of manifolds.
\end{proof}

\subsection{Descent along submersions}

Since the submersion and \'etale topologies are not equivalent over $\mathbf{Man}_\text{sub}$ it does not automatically follow from Theorem \ref{LA is a stack} that $\mathcal{LA}$ satisfies descent for submersions. We prove this directly by roughly the same method as the analogous statement (Lemma 1.5.4) in \cite{BeBe1}. We will need the following (fairly standard) results:

\begin{Proposition} \label{singleton coverings}
Let $\{\phi_i : X_i \to Y \}_{i \in I}$ be a cover in the submersion topology. Let $X=\coprod_i X_i $ and $\phi = \coprod_i \phi_i$, then $\phi:X \to Y$ is a cover of which $\{\phi_i : X_i \to Y \}_{i \in I}$ is a refinement. For any weak presheaf $\mathcal F$ the refinement induces a functor between the categories of descent data:
$$\mathrm{Des}(\{\phi_i : X_i \to Y \}_{i \in I} , \mathcal F) \to \mathrm{Des}( \phi: X \to Y , \mathcal F)$$
which is an equivalence of categories.
\end{Proposition}

\begin{proof}

Since $X=\coprod_i X_i$ there is a bijection between objects of $\mathcal F (X)$ and collections $\{E_i \in \mathcal{F}(X_i)\}_{i \in I}$. We have a commutative diagram: 
$$\xymatrix{
\left(\coprod_i X_i \right) \times_Y \left(\coprod_i X_i\right) \ar[d]_{\mathrm{pr}_1} \ar[r]^\cong & \coprod_{i,j} \left(X_i \times _Y X_j\right) \ar[d]^{\coprod_{i,j} \mathrm{pr}_1} \\
X \ar[r]_{\mathrm{id}} & \coprod_i X_i 
}$$
and similarly with $\mathrm{pr}_2$ in place of $\mathrm{pr}_1$. For any $E \in \mathcal F(X)$ with corresponding $\{E_i \in \mathcal{F}(X_i)\}_{i \in I}$ there is therefore a bijection between isomorphisms
$$\mathrm{pr}_1 ^ \ast E  \to \mathrm{pr}_2 ^\ast E$$
over $X\times_Y X$ and collections of isomorphisms
$$\{ (\mathrm{pr}_1 ^\ast E_i \to \mathrm{pr_2}^\ast E_j) \in \mathrm{Isom}(\mathcal F(X_i \times_Y X_j)) \}_{i,j \in I}$$
It is easy to see that the cocycle condition for an isomorphism $\mathrm{pr}_1 ^ \ast E  \to \mathrm{pr}_2 ^\ast E$ is equivalent to the cocycle condition for the corresponding collection of isomorphisms.

\end{proof}

\begin{Proposition}
Let $\mathcal F$ be a stack over $\mathbf{Man}_{\mathrm{sub}}$ with respect to the open cover topology. Then $\mathcal F$ is a stack with respect to the submersion topology if and only if it satisfies descent for surjective submersions that admit global sections.
\end{Proposition}

\begin{proof}
One direction is clear. For the converse it is sufficient to check that $\mathcal{F}$ satisfies descent for surjective submersions by the Proposition \ref{singleton coverings}. Let $\phi:X \to Y$ be a surjective submersion. Then there exists an open cover $\{ U_i \}_{i \in I}$ of $Y$ and local sections $\{f_i : U_i \to X\}_{i \in I}$ of $\phi$. Let $U= \coprod_{i \in I} U_i$, $f=\coprod_{i \in I} f_i$, and $g: U \to Y$ be the map induced by the inclusions $U_i \rightarrow Y$. Then we have a commutative diagram:
$$\xymatrix{
U \ar[dr]_g \ar[r] ^ f &  X \ar[d]^\phi \\
& Y
}$$
Note that $f$ is an immersion (since each $f_i$ is a section of a submersion) and $g$ is \'etale and surjective. We can form the fibre product:
$$\xymatrix{
U \times _Y X  \ar[d]_{\bar{\phi}} \ar[r]^{\bar{g}} & X \ar[d] ^\phi \\
U \ar[r]_g & Y 
}$$
where $\bar{\phi}$ and $\bar{g}$ are the pullbacks of $\phi$ and $g$, and therefore $\bar{\phi}$ is a surjective submersion and $\bar{g}$ is surjective and \'etale. As $f \circ\phi = g$ the maps
\begin{align*}
\mathrm{id} : U & \to U \\
f: U & \to X
\end{align*}
induce a map
\begin{align*}
\bar{f}: U & \to U \times_Y X \\
y & \mapsto (y,f(y))
\end{align*}
which is a global section of $\bar{\phi}$. As $\mathcal{F}$ is a stack for the open covers topology it satisfies descent for surjective \'etale maps. Therefore $\mathcal{F}$ satisfies descent along $\phi$ if and only if it satisfies descent along $\bar{\phi}$.

\end{proof}
We can now prove:

\begin{Theorem} \label{submersive descent}
$\mathcal{LA}$ is a stack over $\mathbf{Man}_\mathrm{sub}$ with respect to the submersion topology.
\end{Theorem}

\begin{proof}

By the previous proposition it is enough to check that $\mathcal{LA}$ satisfies descent along surjective submersions that admit global sections. Let $\phi:X \to Y$ be a surjective submersion and $f:Y \to X$ a section of $\phi$. We'll construct a quasi-inverse functor to the descent functor
$$\phi^! :  \mathcal{LA}(Y) \to \mathrm{Des}(\phi:X \to Y, \mathcal{LA})$$
Denote the two projections $X\times_Y X \rightrightarrows X$ by $s$ and $t$. Let $(A,\psi)$ be a descent datum. We will show that the pullback $f^! A$ exists. We have the following diagram of Lie algebroid morphisms:
$$\xymatrix{
\mathrm{ker} s_\ast \ar[d] \ar@{^{(}->}[r] & s^! A \ar[d] \ar[r] ^\psi & t^! A  \ar[d]  \ar[r]^{t^!} & A  \ar[d] \\
X\times_Y X  \ar[r]_{\mathrm{id}}  & X\times_Y X \ar[r]_{\mathrm{id}} & X\times_Y X \ar[r]_t & X
}$$
The composition $\rho$ of these morphisms is a morphism covering $t$ and therefore
$$
\xymatrixcolsep{3pc}
\xymatrix{
\mathrm{ker}s_\ast \ar@{^{(}->}[d] \ar[r] ^\rho & A \ar[d]^ a \\
T(X\times_Y X) \ar[r]_-{t_\ast} & TX 
}$$
commutes.
If we identify $T(X\times_Y X)$ with $TX\times_{TY} TX$ then we have
$$( \mathrm{ker}s_\ast )_{(x,x')} = \{ (0,v) | \phi_\ast (v) = 0 \}$$
and
$$t_\ast (0,v) = v$$
It follows that the image of $a:A \to TX$ contains $\mathrm{ker}\phi_\ast$. Since $f:Y \to X$ is a section of $\phi$ we have
$$f_\ast (T_y Y) \oplus (\mathrm{ker}\phi_\ast)_{f(y)} = T_{f(y)}X$$
for all $y \in Y$. It follows that $f_\ast$ is transverse to $a$ and therefore the Lie algebroid $f^! A$ exists. We therefore have a functor
\begin{align*}
f^! :  \mathrm{Des}(\phi:X \to Y, \mathcal{LA}) &  \to \mathcal{LA}(Y) \\
(A,\psi) & \mapsto f^! A \\
\left( \chi: (A,\psi) \to (A',\psi') \right) & \mapsto f^! \chi : f^! A \to f^! A' 
\end{align*}
and
$$f^! \phi^! \simeq (\phi \circ f)^! = (\mathrm{id}_Y)^! \simeq \mathrm{id}_{\mathcal{LA}(Y)}$$
It remains to show that $\phi^! f^! \simeq \mathrm{id}_{\mathrm{Des} (\phi:X \to Y, \mathcal{LA})}$. We have 
$$\phi \circ (f \circ \phi)=(\phi \circ f) \circ \phi = \phi$$ 
and so the maps
\begin{align*}
f \circ \phi : X & \to X \\
\mathrm{id} : X & \to X
\end{align*}
induce a map
$$(f \circ \phi,\mathrm{id}) : X \to X\times_Y X$$
Since
\begin{align*}
s \circ (f \circ \phi, \mathrm{id}) & = f \circ \phi \\
t \circ (f \circ \phi, \mathrm{id}) & = \mathrm{id}
\end{align*}
and for any descent datum $(A,\psi)$ the pullbacks $(f \circ \phi)^! A, s^!A$ and $t^! A$ all exist, we have that $(f \circ \phi,\mathrm{id})^! s^! A$ and $(f \circ \phi,\mathrm{id})^!t^! A$ exist. Composing the isomorphisms:
$$
\xymatrixcolsep{3.2pc}
\xymatrix{
\phi^!f^! A  \ar[r]^\cong &  (f\circ \phi)^!  \ar[r] ^-\cong &  (f \circ \phi,\mathrm{id})^! s^! A   \ar[r]^{(f\circ \phi,\mathrm{id})^! \psi} &  (f \circ \phi,\mathrm{id})^! t^! A  \ar[r]^-\cong  &  A
}$$
gives an isomorphism of Lie algebroids $\Sigma_{(A,\psi)}$:
$$\Sigma_{(A,\psi)} : \phi^!f^! A \to A$$
that is given explicitly by the formula
$$ ( v,\phi_\ast v, \xi) \mapsto  \tilde \psi \left(  ((f\circ \phi)_\ast (v), v) \xi  \right) $$
Moreover this is an isomorphism of descent datum
$$\Sigma_{(A,\psi)} : (\phi^! f^! A, \bar \psi) \to (A,\psi)$$
where $\bar\psi$ is the composition of the isomorphisms
$$s^! \phi^! f^! A \cong (\phi \circ s)^! f^! A = (\phi \circ t)^! f^! A \cong t^! \phi^! f^! A$$
That $\Sigma_{(A,\psi)}$ is a morphism of descent data is the condition that the diagram
$$\xymatrix{
s^! \phi^! f^! A  \ar[d]_{s^! \Sigma_{(A,\psi)}} \ar[r]^{\bar \psi} &   t^! \phi^! f^! A  \ar[d]^{s^! \Sigma_{(A,\psi)}} \\
s^! A \ar[r]_\psi &  t^! A 
}$$
commutes, which follows from a direct computation using the cocycle condition for $\psi$. Finally, given a morphism $\chi : (A,\psi) \to (A',\psi')$ of descent data a similar computation, using the fact that $\psi' \circ s^! \chi = t^!\chi \circ \psi$, shows that the diagram
$$\xymatrix{
\phi^!f^! A  \ar[d]_{\Sigma_{(A,\psi)}}  \ar[r] ^ {\phi^! f^! \rho}  &  \phi^! f^! A'  \ar[d]^{\Sigma_{(A',\psi')}}  \\
A  \ar[r]_{\chi} & A' 
}$$
commutes and therefore the isomorphisms $\Sigma_{(A,\psi)}$ give a natural isomorphism $\phi^!f^! \simeq \mathrm{id}_{\mathrm{Des} (\phi:X \to Y, \mathcal{LA})}$.
\end{proof}

\subsection{The stack $\mathrm{Rep}$ of representations} \label{The stack Rep}

It will be useful later to know that representations of Lie algebroids satisfy descent along submersions. We'll collect together all representations of Lie algebroids into a single weak presheaf over $\mathbf{Man}_\text{sub}$. The proof that it satisfies descent along submersions is similar to that of Theorem \ref{submersive descent}.

\subsubsection{The category $\mathrm{Rep}_X$}

Let $X$ be a manifold. Then we can consider the category $\mathrm{Rep}_X$, whose objects are triples 
$$(A,E,\nabla)$$
where $A$ is a Lie algebroid over $X$, $E$ is a vector bundle over $X$, and $\nabla$ is a representation of $A$ on $E$. A morphism in $\mathrm{Rep}_X$ is given by a pair of maps
$$(\phi,\psi) : (A,E,\nabla) \to (A',E',\nabla')$$
where 
$$\phi: A \to A'$$
is a morphism of Lie algebroids over $X$ and
$$\psi: E \to E'$$
is a morphism of vector bundles over $X$, such that 
$$\nabla'_{\phi(\xi)} \psi(e) = \psi \left(  \nabla_\xi e \right) $$
for all $\xi \in \Gamma(A), e \in \Gamma(E)$.

\subsubsection{The presheaf $\mathrm{Rep}$}  \label{Rep presheaf}

Let $f: X \to Y$ be a submersion and $(A,E,\nabla) \in \mathrm{Rep}_Y$. Then we have 
$$(f^! A, f^\ast E, f^! \nabla) \in \mathrm{Rep}_X$$
(see \ref{Pullbacks of representations}). Given a morphism 
$$(\phi,\psi) : (A,E,\nabla) \to (A',E',\nabla')$$
in $\mathrm{Rep}_Y$ we have the maps
\begin{align*}
f^! \phi : f^! A & \to f^! A' \\
f^\ast \psi : f^\ast E & \to f^\ast E'
\end{align*}
and if $v \oplus (g \otimes \xi) \in \Gamma(f^! A)$ and $g'\otimes e \in \Gamma(f^\ast E)$ (see \ref{The pullback construction} for the description of sections of $f^! A$) then we have:
\begin{align*}
(f^\ast \psi) \left(  (f^! \nabla)_{v \oplus (g \otimes \xi)} (g' \otimes e) \right)  &  = 
(f^\ast \psi)   \left(  v(g') \otimes e  +  gg'\otimes \nabla_\xi e  \right)  \\
& = v(g') \otimes \psi(e)  +  gg' \otimes \psi(\nabla_\xi e)   \\
& = v(g') \otimes \psi(e)  +  gg' \otimes \nabla'_{\phi(\xi)} \psi(e)  \\
& = (f^! \nabla') _ { v \oplus (g\otimes \phi(\xi))} ( g' \otimes \psi(e) )  \\
& = (f^! \nabla')_{(f^! \phi)(v\oplus (g\otimes \xi))} ( (f^\ast \psi)(g'\otimes e) )
\end{align*}
and therefore 
$$(f^! A, f^\ast E, f^! \nabla) \in \mathrm{Rep}_X$$
Since $f^!$ and $f^\ast$ are functors we have a functor 
$$\mathrm{Rep}_Y \to \mathrm{Rep}_X$$
and it follows from the fact that $\mathcal{LA}$ and $\mathrm{Vect}$ are weak presheaves that in a natural way we have a weak presheaf:
$$\mathrm{Rep} : \mathbf{Man}_\mathrm{sub} \to \mathbf{Cat}$$
There are obvious morphims of presheaves:
$$\xymatrix{
& \mathrm{Rep} \ar[dr] \ar[dl] &  \\
\mathrm{Vect} & & \mathcal{LA}
}$$
that map an object $(A,E,\nabla)$ to either $E$ or $A$.

\subsubsection{Descent for representations}

\begin{Proposition}
$\mathrm{Rep}$ is a stack over $\mathbf{Man}_\mathrm{sub}$ in the open cover and \'etale topologies.
\end{Proposition}

\begin{proof}
The main part of the proof is essentially the same as that of Theorem \ref{LA is a stack} and Corollary \ref{LA is a stack 2}. The analogues of Propositions \ref{equivalence over etale category} and \ref{equivalence over open category} hold for the presheaf $\mathrm{Rep}$, so that when restricted to the category $\mathrm{Open}_X$ for a manifold $X$, $\mathrm{Rep}$ is equivalent to the strict presheaf obtained by replacing the pullback functors by restrictions. 

Let $X$ be a manifold, $U \in \mathrm{Open}_X$, and 
$$(A,E,\nabla),(A',E',\nabla') \in \mathrm{Rep}_U$$ 
Let $\{U_i\}_{i \in I}$ be an open cover of $U$ and 
$$\{ (\phi_i,\psi_i) : (A|_{U_i},E|_{U_i},\nabla|_{U_i}), \to (A'|_{U_i},E'|_{U_i},\nabla'|_{U_i}) \}_{i \in I}$$ 
a family of morphisms that agree on intersections. Then since $\mathcal{LA}$ and $\mathrm{Vect}$ are stacks the $\phi_i$'s and $\psi_i$'s glue uniquely into a pair of morphisms:
\begin{align*}
\phi : A & \to A' \\
\psi : E & \to E'
\end{align*}
It follows from the fact that connections are local, i.e. that for $\xi \in \Gamma(A), e \in \Gamma(E)$ we have:
$$ \left( \nabla _ {\xi} e  \right) |_{U_i}  =   \left(\nabla |_{U_i}\right) _ {\xi |_{U_i}} (e |_{U_i})    $$
that
$$ \nabla' _ {\phi(\xi)} \psi(e)  =  \psi \left(  \nabla_\xi e    \right) $$
and therefore
$$(\phi,\psi) \in \mathrm{Hom}_\mathrm{Rep} \left(  (A,E,\nabla) , (A',E',\nabla') \right)$$
Now consider a descent datum
$$ \left( \{ (A_i , E_i , \nabla_i)  \}_{i \in I} , \{ (\phi_{ij},\psi_{ij}) \} _{i,j \in I}  \right)$$
Then, in particular, we have that 
$$ \left( \{A_i\}_{i \in I} ,  \{\phi_{ij} \}_{i,j \in I}  \right)$$
is a descent datum for $\mathcal{LA}$, and
$$ \left( \{E_i\}_{i \in I} ,  \{\psi_{ij} \}_{i,j \in I}  \right)$$
is a descent datum for $\mathrm{Vect}$. We therefore have objects
\begin{align*} 
A  &  \in \mathcal{LA}_U \\
E & \in \mathrm{Vect}_U
\end{align*}
Using the canonical identifications
\begin{align*}
\Gamma(U_i,A) & \cong \Gamma(A_i) \\
\Gamma(U_i,E) & \cong \Gamma(E_i) 
\end{align*}
we can define a representation of $A$ on $E$ by  
$$ \left( \nabla_\xi e  \right) |_{U_i}  = (\nabla_i)_{\xi | _{U_i}} ( e | _{U_i}) $$ 
for $\xi \in \Gamma(A)$ and $e \in \Gamma(E)$. This is well defined because over $U_{ij}$ we have
$$ \psi_i   \left(   (\nabla_i) _ {\xi_i} e_i  \right)  =  (\nabla_j) _ {\phi_i(\xi_i)} ( \psi_i(e_i))$$
for $\xi_i \in \Gamma(A_i)$ and $e_i \in \Gamma(E_i)$.
\end{proof}

\begin{Proposition}
$\mathrm{Rep}$ is a stack over $\mathbf{Man}_\mathrm{sub}$ with respect to the submersion topology.
\end{Proposition}

\begin{proof}
The proof is similar to that of Theorem \ref{submersive descent}. We can restrict to the case where $g: X \to Y$ is a surjective submersion and $f: Y \to X$ a section of $g$. Denote the two projections $X\times_Y X \rightrightarrows Y$ by $s,t$. We need to construct a quasi-inverse to the descent functor
$$(g^! , g^\ast) :  \mathrm{Rep}_Y \to \mathrm{Des}(g:X \to Y, \mathrm{Rep})$$
Let $\left( (A,E,\nabla) , (\phi,\psi) \right)$ be a descent datum:
\begin{align*}
(A,E,\nabla) & \in \mathrm{Rep}_X \\ 
(\phi,\psi): (s^! A, s^\ast E, s^! \nabla) & \to (t^! A, t^\ast E, t^! \nabla) 
\end{align*}
Then, $(A,\phi)$ is a descent datum for $\mathcal{LA}$ and $(E,\psi)$ is a descent datum for $\mathrm{Vect}$. From the proof of Theorem \ref{submersive descent} the Lie algebroid $f^! A$ exists and so we have
$$ (f^! A, f^\ast E, f^! \nabla) \in \mathrm{Rep}_Y$$
This gives a functor
$$\mathrm{Des}(g:X \to Y, \mathrm{Rep}) \to \mathrm{Rep}_Y$$
The same proof as that of Theorem \ref{submersive descent} then shows that this is quasi-inverse to the descent functor.
\end{proof}

\newpage

\section{Lie algebroids over differentiable stacks}   \label{Lie algebroids over differentiable stacks}

\subsection{Restricting stacks to $\mathbf{Man}_\text{sub}$ and $\mathbf{Man}_\text{\'et}$}

In order to use $\mathcal {LA}$ as a classifying stack and define Lie algebroids over a stack $\mathfrak X$ as maps from $\mathfrak X$ to $\mathcal {LA}$ we need to restrict $\mathfrak X$ to $\mathbf{Man}_\text{sub}$. We have to do this in such a way that if $\mathfrak X$ is equivalent to a manifold $X$, then we recover the category of Lie algebroids over $X$.

\subsubsection{Preserving representability}

If $X$ is a manifold then we can restrict the presheaf $\underline X$ from $\mathbf{Man}$ to $\mathbf{Man}_\text{sub}$, but this does not preserve representability, i.e.  
$$ \underline X |_{\mathbf{Man}_\text{sub}}  \ncong \mathrm{Hom}_{ \mathbf{Man}_\text{sub}} ( - , X )  $$
Therefore we do not necessarily have that
$$ \mathrm{Hom}_{\mathrm{St}(\mathbf{Man}_\text{sub})} ( \underline X | _{\mathbf{Man}_\text{sub}} , \mathcal {LA} ) \simeq \mathcal{LA}_X$$
We could attempt to correct this by defining a functor
$$ - _\text{sub} : \mathrm{PSh}(\mathbf{Man}) \to \mathrm{PSh}(\mathbf{Man}_\text{sub}) $$ 
that maps represented presheaves to represented presheaves:
$$ \mathrm{Hom} _\mathbf{Man} ( - , X )  \mapsto \mathrm{Hom}_{ \mathbf{Man}_\text{sub}} ( - , X )  $$
and then extending to $\mathrm{PSh}(\mathbf{Man})$ by declaring that $ - _\text{sub}$ commutes with homotopy colimits (using the fact that every weak presheaf of groupoids is a colimit of representables). We will instead take a more concrete approach, which will amount to replacing a stack by the analogue of its smooth site in algebraic geometry. 

\subsubsection{The restriction functors  $-_\text{sub}$ and $-_\text{\'et}$ }

Let $\mathfrak X$ be a weak presheaf of groupoids over $\mathbf{Man}$. For any manifold $U$, we define $\mathfrak X _\text{sub} (U)$ to be the groupoid consisting of representable submersions from $U$ to $\mathfrak X$. By the Yoneda lemma we can consider  $\mathfrak X _\text{sub} (U)$ to be a full subgroupoid of $\mathfrak X (U)$. If $f: U \to V$ is a submersion of manifolds and $\phi: V \to \mathfrak X$ a representable submersion, then $(\phi \circ f) : U \to \mathfrak X$ is a representable submersion. It follows that we have a weak presheaf:
$$\mathfrak X  _\text{sub} : \mathbf{Man}_\text{sub} \to \mathbf{Gpd}$$
If $\phi : \mathfrak X \to \mathfrak Y$ is a representable submersion of weak presheaves over $\mathbf{Man}$, $U$ is a manifold, and $f : U \to \mathfrak X$ a representable submersion, then $(\phi \circ f): U \to \mathfrak Y$ is a representable submersion. It follows that we have a functor:
$$  - _\text{sub} :  \mathrm{PSh}(\mathbf{Man})_\mathrm{sub} \to \mathrm{PSh}(\mathbf{Man}_\text{sub})$$ 
where $\mathrm{PSh}(\mathbf{Man})_\mathrm{sub}$ is the subcategory of $\mathrm{PSh}(\mathbf{Man})$ consisting of weak presheaves and representable submersions.

Since morphisms of representable presheaves over $\mathbf{Man}$ are representable if and only if they correspond via the Yoneda lemma to submersions (this is proved in \cite{Me1}), we have that for any manifold $X$
$$\underline X  _\text{sub}  =  \mathrm{Hom}_{\mathbf{Man}_\text{sub}} ( - , X )$$
as desired. We define the functor
$$  - _\text{\'et} :  \mathrm{PSh}(\mathbf{Man})_\text{\'et} \to \mathrm{PSh}(\mathbf{Man}_\text{\'et}) $$
in exactly the same way as $-_\text{sub}$, except we replace representable submersions by representable \'etale morphisms.

\subsubsection{The submersion and \'etale sites} \label{The submersion and etale sites}

If $\mathfrak X$ is a weak presheaf over $\mathbf{Man}$ then the fibered category
$$\xymatrix{
 \int _{\mathbf{Man}_\mathrm{sub}} \mathfrak X  _ \text{sub}  \ar[d] \\
\mathbf{Man}_\text{sub}
}$$
corresponding to $\mathfrak X  _\mathrm{sub}$ is equivalent to the category whose objects $(U,u)$ are representable submersions:
$$u: U \to \mathfrak X$$ 
and whose morphisms $(f,\alpha):(U,u) \to (V,v)$ are 2-commutative triangles of representable submersions:
$$\xymatrix{
U \ar[d]_f \ar[r]^u  & \mathfrak X  \\
V \ar[ur]_v & \\
}$$
with $\alpha: u \to v \circ f$ an isomorphism (we use the notation of, for example, \cite{EySa1}). We'll equip this category with the open cover topology induced from $\mathbf{Man}_\text{sub}$ and call the resulting site the submersion site of $\mathfrak X$. This is the $\mathrm C^\infty$ analogue of the smooth site of an algebraic stack. Similarly we have the \'etale site of $\mathfrak X$ whose underlying category is the category fibred over $\mathbf{Man}_\text{\'et}$ corresponding to $\mathfrak X _ \text{\'et}$.

We'll usually just write $\mathfrak X_\text{sub}$ instead of  $\int _{\mathbf{Man}_\mathrm{sub}} \mathfrak X  _ \text{sub}$, but whether we are considering the weak presheaf or the corresponding fibered category will hopefully be clear.

\subsection{Lie algebroids over stacks}

We will define Lie algebroids over stacks essentially by declaring that they are classified by the stack $\mathcal{LA}$. We'll then show that a Lie algebroid over a stack $\mathfrak X$ is determined by a Lie algebroid over $U$ for every submersion $U \to \mathfrak X$ from a manifold $U$, together with isomorphisms corresponding to each 2-commutative triangle of such submersions. Using the fact that $\mathcal{LA}$ satisfies descent for submersions, we then show how the category of Lie algebroids over a differentiable stack can be described in terms of a Lie groupoid representing it.

\subsubsection{Definition via the classifying stack}

\begin{Definition} \label{Definition of Lie algebroids over stacks}
If $\mathfrak X$ is a differentiable stack, or more generally a weak presheaf of groupoids over $\mathbf{Man}$, then we define the category $\mathcal{LA}_\mathfrak X$, of Lie algebroids over $\mathfrak X$, as the hom category:
$$\mathcal{LA}_\mathfrak X \equiv \mathrm{Hom}_{\mathrm{PSh}(\mathbf{Man}_\text{sub})} ( \mathfrak X  _ \text{sub} , \mathcal{LA} )$$
\end{Definition} 
By this definition a Lie algebroid over a stack $\mathfrak X$ is a morphism $\mathbf A: \mathfrak X_\text{sub} \to \mathcal{LA}$ of stacks over $\mathbf{Man}_\text{sub}$, and a morphism between two Lie algebroids $\mathbf A$ and $\mathbf A'$ over $\mathfrak X$ is a 2-morphism $\psi: \mathbf A \to \mathbf A'$.
We have for any manifold $X$:
$$\mathcal{LA}_{\underline X} = \mathrm{Hom}_{\mathrm{PSh}(\mathbf{Man}_\text{sub})} (\underline X  _ \text{sub} , \mathcal{LA} )  \simeq  \mathcal{LA}_X $$

\subsubsection{Pullbacks}  \label{pullbacks along representable submersions}

We can pull back Lie algebroids along representable submersions: if $\phi: \mathfrak X \to \mathfrak Y$ is a representable submersion of weak presheaves over $\mathbf{Man}$ then applying the functor $- _\text{sub}$ gives a morphism 
$$\phi _\text{sub} : \mathfrak X _\text{sub} \to \mathfrak Y _\text{sub}$$
and therefore a functor
\begin{align*}
\left( \phi  _ \text{sub} \right) ^! : 
\mathrm{Hom}_{\mathrm{PSh}(\mathbf{Man}_\text{sub})} ( \mathfrak Y  _ \text{sub} , \mathcal{LA} )  & \rightarrow \mathrm{Hom}_{\mathrm{PSh}(\mathbf{Man}_\text{sub})} ( \mathfrak X  _ \text{sub} , \mathcal{LA} )  \\
\mathbf A  & \mapsto \mathbf A \circ \phi_\text{sub} 
\end{align*}
and so by the definition of $\mathcal{LA}_\mathfrak X$ and $\mathcal{LA}_\mathfrak Y$ we have a functor
$$\left( \phi _\text{sub} \right)^! : \mathcal{LA}_\mathfrak Y \to \mathcal{LA}_\mathfrak X$$
Under the equivalence $\mathcal{LA}_{\underline X} \simeq \mathcal{LA}_X$ these pullback functors are compatible with the standard pullback functors associated to submersions between manifolds.

\subsubsection{Base changing morphisms}

Using the pullback functors above we can define morphisms between Lie algebroids over different stacks, at least in the case that the underlying morphism of stacks is a submersion. Let $\mathbf A$ be a Lie algebroid over $\mathfrak X$ and $\mathbf B$ be a Lie algebroid over $\mathfrak Y$. Then a morphism from the pair $(\mathfrak X, \mathbf A)$ to the pair $(\mathfrak Y, \mathbf B)$ is a pair of morphisms $(\phi,\psi)$, where $\phi : \mathfrak X \to \mathfrak Y$ is a representable submersion and $\psi: \mathbf A \to \phi^\ast \mathbf B$ is a morphism in the category $\mathcal{LA}_\mathfrak X$. Diagrammatically we have:
\[
\xymatrix@C+2pc{
\mathfrak X \rtwocell^{\mathbf A}_{\mathbf B \circ \phi }{\;\;\;  \psi} &  \mathcal{LA} 
}
\]
If $X$ and $Y$ are manifolds then it follows from the universal property of pullbacks of Lie algebroids that under the equivalences $\mathcal{LA}_{\underline X} \simeq \mathcal{LA}_X$ and $\mathcal{LA}_{\underline Y} \simeq \mathcal{LA}_Y$ this definition agrees with the standard notion of base changing morphisms of Lie algebroids.

\subsubsection{The data determining a morphism $\mathfrak X_\text{sub} \to \mathcal{LA}$} \label{Lie algebroids via test spaces}

If we work with the corresponding fibered categories, then we have that a Lie algebroid over a stack $\mathfrak X$ is a morphism of categories fibered over $\mathbf{Man}_\text{sub}$:
$$\xymatrix{
\mathfrak X _\text{sub}  \ar[dr]  \ar[r]  &  \mathcal{LA}  \ar[d] \\
& \mathbf{Man}_\text{sub}
}$$
The data $\mathbf A$ of such a morphism consists of a Lie algebroid $A_{U,u}$ over $U$ for each object $(U,u)$ of $\mathfrak X_\text{sub}$, together with a morphism of Lie algebroids
$$\xymatrix{
A_{U,u} \ar[d]  \ar[r] ^ {\mathbf A_{f,\alpha}} &  A_{V,v} \ar[d] \\
U \ar[r] _f & V 
}$$
for each morphism $(f,\alpha):(U,u) \to (V,v)$ in $\mathfrak X_\text{sub}$, or equivalently, a morphism 
$$\xymatrix{
A_{U,u} \ar[dr] \ar[r] ^ { \bar{\mathbf A} _{f,\alpha}} & f^! A_{V,v} \ar[d] \\
& V 
}$$
The morphisms $\mathbf A_{f,\alpha}$ and $\bar{\mathbf A} _{f,\alpha}$ are related by:
$$\mathbf A_{f,\alpha} = f^! \circ \bar{\mathbf A} _{f,\alpha}$$ 
where $f^!$ is the pullback morphism $f^! A_{V,v} \to A_{V,v}$. Given a composable pair of morphisms in $\mathfrak X _\text{sub}$:
$$\xymatrix{
(U,u) \ar[r] ^ {(f,\alpha)} & (V,v) \ar[r] ^ {(g,\beta)} & (W,w)
}$$
the functoriality condition
$$\mathbf A_{gf, \beta \alpha} = \mathbf A_{g,\beta} \circ \mathbf A_{f,\alpha}$$
is then equivalent to the condition that the following diagram commutes:
$$\xymatrix{
A_{U,u}  \ar[d]_{\bar {\mathbf A} _{gf, \beta \alpha}} \ar[r]^{\bar {\mathbf A}_{f,\alpha}} & f^! A_{V,v} \ar[d]^{g^! (\bar {\mathbf A} _{v,\beta})} \\
(gf)^! A_{W,w}  \ar[r]_{c_{a,b}} &  f^! g^! A_{W,w} 
}$$
Morphisms between Lie algebroids over $\mathfrak X$ are by definition natural transformations that are vertical with respect to the projections to $\mathbf{Man}_\text{sub}$. In terms of the description above, such a morphism $\phi$ between two collections of data $\mathbf A$ and $\mathbf B$ consists of a morphism
$$\xymatrix{
A_{U,u} \ar[dr]  \ar[r] ^ {\phi_{U,u}} & B_{U,u} \ar[d] \\
& U
}$$
for each $(U,u)$. The naturality condition is equivalent to the condition that the diagram
$$\xymatrix{
A_{U,u} \ar[d] _ {\phi_{U,u}} \ar[r]^{\bar {\mathbf A}_{f,\alpha}} & f^! A_{V,v} \ar[d] ^{\phi_{U,u}} \\
B_{U,u}  \ar[r] _{\bar {\mathbf B} _ {f,\alpha}} & f^! B_{V,v} 
}$$
commutes for all $(f,\alpha) : (U,u) \to (V,v)$.

\subsubsection{Lie algebroids in terms of atlases}  \label{Lie algebroids in terms of atlases} 

Let $\mathfrak X$ be a differentiable stack and $X \to \mathfrak X$ be an atlas. We then have the Lie groupoid $X \times_\mathfrak X  X \rightrightarrows X$ and the 2-commutative diagram
$$\xymatrix{
X \times_\mathfrak X X \ar[d]_{\mathrm{pr}_1}  \ar[r]^>>>>>{\mathrm{pr}_2} &  X \ar[d] \\
X \ar[r] & \mathfrak X 
}$$
Let $\mathfrak Y$ be a stack over $\mathbf{Man}$. Then we can define a category:
$$\mathrm{Des} ( X \to \mathfrak X , \mathfrak Y )$$
The objects are pairs $(\phi,\alpha)$ where
$$ \phi : X \to \mathfrak Y$$ 
is a morphism of stacks and $\alpha$ is a 2-morphism:
\[
\xymatrix@C+2pc{
X \times _\mathfrak X X \rtwocell^{\phi \circ \mathrm{pr}_1}_{\phi \circ \mathrm{pr}_2}{\;\;\; \; \; \alpha} &  \mathfrak Y
}
\]
that satisfies the cocycle condition over $X \times _\mathfrak X \times X \times _\mathfrak X X$. A morphism $\beta : (\phi,\alpha) \to (\phi',\alpha')$ is a 2-morphism
\[
\xymatrix@C+2pc{
X \rtwocell^{\phi}_{\phi'}{\;\;\;  \beta} &  \mathfrak Y
}
\]
such that the diagram
$$\xymatrix{
\phi \circ \mathrm{pr}_1  \ar[d]_\alpha   \ar[r]^{\mathrm{pr}_1 ^\ast (\beta)} &   \phi' \circ \mathrm{pr}_1  \ar[d]^{\alpha'}  \\
\phi \circ \mathrm{pr}_2  \ar[r]_{\mathrm{pr}_2 ^\ast (\beta)}  &   \phi' \circ \mathrm{pr}_2   
}$$
commutes. The morphism $X \to \mathfrak X$ induces a functor 
$$\mathrm{Hom} _ {\mathrm{St}(\mathbf{Man})} ( \mathfrak X , \mathfrak Y )  \to 
\mathrm{Des} ( X \to \mathfrak X , \mathfrak Y )$$
which is an equivalence of categories \cite{He1}. By the Yoneda lemma, $\mathrm{Des}( X \to \mathfrak X , \mathfrak Y )$ is equivalent to the category whose objects are pairs $(\xi,\alpha)$, where $\xi \in \mathfrak Y (X)$ and $\alpha$ is an isomorphism $\alpha : \mathrm{pr}_1 ^\ast \xi \to \mathrm{pr}_2^* \xi$ that satisfies the cocycle condition. More generally, for any Lie groupoid $G \rightrightarrows Y$, we can define the category $\mathrm{Des} ( G \rightrightarrows Y , \mathfrak Y )$ as above. 

In our situation we have that $\mathcal{LA}$ is only a stack over $\mathbf{Man}_\text{sub}$. However, the key point of the proof of the above equivalence of categories is that if $U \to \mathfrak X$ is a morphism from a manifold $U$ to $\mathfrak X$, then an object $(\phi,\alpha)$ in $\mathrm{Des} ( X \to \mathfrak X , \mathfrak Y )$ induces a descent datum for $\mathfrak Y$ and the submersion $U\times_\mathfrak X X \to U$, and therefore, an object in $\mathfrak Y(U)$. Since $\mathcal {LA}$ satisfies descent for submersions (Theorem \ref{submersive descent}) the same proof works.

We define:
$$\mathcal{LA}_{G \rightrightarrows Y} = \mathrm{Des} ( G \rightrightarrows Y , \mathcal{LA} )$$
and applying the above to the case of Lie algebroids, and using the Yoneda lemma, we have:
\begin{Proposition} \label{Proposition on Lie algebroids via atlases}
If $\mathfrak X$ is a differentiable stack then any atlas $X \to \mathfrak X$ induces an equivalence of categories: 
$$\mathcal{LA}_\mathfrak X \simeq \mathcal{LA}_{X \times_\mathfrak X X \rightrightarrows X}$$
In particular, if $G \rightrightarrows X$ and $G' \rightrightarrows X'$ are Morita equivalent Lie groupoids then there is an equivalence of categories:
$$ \mathcal{LA}_{G \rightrightarrows X} \simeq \mathcal{LA}_{G' \rightrightarrows X'}$$
\end{Proposition}

This gives us a simple way to work with Lie algebroids over a differentiable stack $\mathfrak X$: choose a Lie groupoid $G\rightrightarrows X$ presenting $\mathfrak X$, then a Lie algebroid over $\mathfrak X$ is given by pair $(A,\alpha)$, where $A$ is a Lie algebroid over $X$, and $\alpha$ is an isomorphism $\alpha: s^! A \to t^! A$ that satisfies the cocycle condition. We'll use this in Section \ref{Examples} to describe the categories of Lie algebroids over some examples of differentiable stacks.

In general, for any Lie groupoid $G \rightrightarrows X$, we'll refer to pairs $(A,\alpha)$ as `Lie algebroids over $G\rightrightarrows X$'. We study these in Section \ref{Lie algebroids over Lie groupoids} and show how they are related to $\mathcal{LA}$-groupoids.

\subsubsection{Lie algebroids over stacks are not vector bundles}  \label{Lie algebroids over stacks are not vector bundles}

As discussed in \ref{Lie algebroids vs vector bundles} we do not have a natural morphism of stacks $\mathcal{LA} \to \mathrm{Vect}$, and so a Lie algebroid $\mathbf A :\mathfrak X_\text{sub} \to \mathcal{LA}$ over a stack $\mathfrak X$ does not determine a vector bundle over $\mathfrak X$ in any natural way. This is familiar from the case of the tangent bundle of a differentiable stack, see e.g. \cite{BeGiNoXu1} or \cite{He1}.

In fact, a Lie algebroid $\mathbf A$ over a differentiable stack $\mathfrak X$ does not even have a well defined rank: if $x:X \to \mathfrak X$ is an atlas, and $A_{X,x}$ has rank $r$ as a vector bundle over $X$, then for any submersion $f:Y \to X$, we have that $(Y \to X \to \mathfrak X)$ is an object of the site $\mathfrak X_\text{sub}$, and if $Y \to X$ has relative dimension $s$, then $A_{Y,x \circ f}$ has rank $r+s$. 

However, we'll show in \ref{Total stacks of Lie algebroids over stacks} that one can construct a stack $\mathfrak A$ together with a morphism $\mathfrak A \to \mathfrak X$, that represents the `total stack' of the Lie algebroid $\mathbf A$. As in the case of tangent stacks, the morphism $\mathfrak A \to \mathfrak X$ is not representable in general.

\subsection{Representations}

\subsubsection{Defining representations}

\begin{Definition} \label{definition of representations}
Let $\mathbf A:\mathfrak X  _\text{sub} \to \mathcal{LA}$ be a Lie algebroid over a stack $\mathfrak X$. Then the category $\mathrm{Rep}_\mathbf A$, of representations of $\mathbf A$, is the full subcategory of $\mathrm{Hom}_{\mathrm{PSh}(\mathbf{Man}_\text{sub})} (\mathfrak X _ \text{sub} , \mathrm{Rep} )$ consisting of morphisms $\mathfrak X_\text{sub} \to \mathrm{Rep}$ such that the diagram
$$\xymatrix{
&   \mathrm{Rep}  \ar[d]  \\
\mathfrak X _\text{sub} \ar[ur]   \ar[r] _ {\mathbf A} & \mathcal{LA}
}$$
commutes, where $\mathrm{Rep} \to \mathcal{LA}$ is the morphism described in \ref{Rep presheaf} that maps an object $(A,E,\nabla)$ to the Lie algebroid $A$.
\end{Definition}

Via the morphism $\mathrm{Rep} \to \mathrm{Vect}$, every representation $\mathfrak X_\text{sub} \to \mathrm{Rep}$ has an underlying vector bundle in the sense that it determines a morphism $\mathfrak X_\text{sub} \to \mathrm{Vect}$. We can therefore define a representation of a Lie algebroid $\mathbf A$ over $\mathfrak X$ on a fixed vector bundle $\mathbf E$:

\begin{Definition}
Given a Lie algebroid $\mathbf A: \mathfrak X_\text{sub} \to \mathcal{LA}$ and a vector bundle $\mathbf E: \mathfrak X_\text{sub} \to \mathrm{Vect}$ over a stack $\mathfrak X$, a representation of $\mathbf A$ on $\mathbf E$ is a morphism $\mathfrak X_\text{sub} \to \mathrm{Rep}$  such that the diagram
$$\xymatrix{
&  \mathcal{LA} \\
\mathfrak X _\text{sub} \ar[ur]^{\mathbf A}  \ar[dr]_{\mathbf E}  \ar[r] & \mathrm{Rep} \ar[u] \ar[d] \\
& \mathrm{Vect}
}$$
commutes.
\end{Definition}

\subsubsection{The data determining a representation}  \label{The data determining a representation}

If we think in terms of fibered categories, as in \ref{Lie algebroids via test spaces}, then we can describe representations as follows. Fix a stack $\mathfrak X$. Then a Lie algebroid $\mathbf A$ over $\mathfrak X$ is given by fixing a Lie algebroid $A_{U,u}$ over $U$ for each representable submersion $u: U \to \mathfrak X$, together with a morphism $\bar \chi _{f,\alpha}:A_{U,u} \to f^! A_{V,v}$ for each 2-commutative triangle $(f,\alpha) : (U,u) \to (V,v)$. Similarly, a vector bundle $\mathbf E$ over $\mathfrak X$ is given by fixing a vector bundle $E_{U,u}$ over $U$ for each representable submersion $u: U \to \mathfrak X$, together with a morphism $\bar \chi' _{f,\alpha}:E_{U,u} \to f^\ast E_{V,v}$ for each 2-commutative triangle $(f,\alpha) : (U,u) \to (V,v)$. The morphisms $\bar \chi _{f,\alpha}$ and $\bar \chi' _{f,\alpha}$ must satisfy the required compatibility condition. A representation of $\mathbf A$ on $\mathbf E$ is then determined by a representation
$$\nabla_{U,u} : \Gamma (A_{U,u}) \times \Gamma (E_{U,u}) \to \Gamma (E_{U,u})$$
for each $(U,u)$, and these representations must be compatible in the sense that the diagram  
$$\xymatrix{
\Gamma(A_{U,u}) \times \Gamma(E_{U,u})  \ar[d]_ {\bar \chi_{U,u} \times \bar \chi'_{U,u}}  \ar[r]^>>>>>>>{\nabla_{U,u}} &  \Gamma(E_{U,u})  \ar[d]^{\bar \chi' _{U,u}}   \\
\Gamma(f^! A_{V,v}) \times \Gamma(f^\ast E_{V,v})   \ar[r]_>>>>>{f^! \nabla_{V,v}} &  \Gamma(f^\ast E_{V,v})  
}$$
commutes for each $(f,\alpha) : (U,u) \to (V,v)$.

\subsubsection{Representations in terms of atlases}

As in \ref{Lie algebroids in terms of atlases} we can work with representations in terms of atlases: the functor 
$$\mathrm{Hom}_{\mathrm{PSh}(\mathbf{Man}_\text{sub})}(\mathfrak X_\text{sub},\mathrm{Rep}) \to \mathrm {Des}(X \to \mathfrak X, \mathrm{Rep})$$
is an equivalence for any differentiable stack $\mathfrak X$ and atlas $X \to \mathfrak X$ because the weak presheaf $\mathrm{Rep}$ satisfies descent for submersions. If $X \to \mathfrak X$ is an atlas of a differentiable stack and $\mathbf A : \mathfrak X_\text{sub} \to \mathcal{LA}$ is a Lie algebroid over $\mathfrak X$ then we have a Lie algebroid $A_X$ over $X$ and an isomorphism $\psi: s^! A_X \to t^! A_X$, where $s,t$ are the projections $X \times _\mathfrak X X \rightrightarrows X$. A representation $\mathfrak X_\text{sub} \to \mathcal{LA}$ of $\mathbf A$ determines a vector bundle $E_X$ over $X$, an isomorphism $\psi': s^\ast E_X \to t^\ast E_X$, and a representation
$$\nabla_X : \Gamma(A_X) \times \Gamma(E_X) \to \Gamma(E_X)$$
which satisfies the condition that the diagram
$$\xymatrix{
\Gamma(s^! A_X) \times \Gamma(s^\ast E_X)  \ar[d]_ {\psi \times \psi'}  \ar[r]^>>>>>{s^! \nabla_X} &  \Gamma(s^\ast E_X)  \ar[d]^{\psi'}   \\
\Gamma(t^! A_X) \times \Gamma(t^\ast E_X)   \ar[r]_>>>>>{t^! \nabla_X} &  \Gamma(t^\ast E_X)  
}$$
commutes. A morphism $\rho: \mathbf E \to \mathbf E'$ between representations $\mathbf E$ and $\mathbf E$ then determines a morphism of vector bundles $\rho_X : E_X \to E'_X$ which commutes with the actions of $A_X$ and $X\times_\mathfrak X X \rightrightarrows X$, i.e. that the diagrams
$$\xymatrix{
\Gamma(A_X) \times \Gamma(E_X) \ar[d]_{\mathrm{id}\times\rho} \ar[r]^>>>>>{\nabla_X} &  \Gamma(E_X)  \ar[d]^\rho \\
\Gamma(A_X) \times \Gamma(E'_X)  \ar[r]_>>>>>{\nabla'_X} &  \Gamma(E'_X)  
}$$
and
$$\xymatrix{
s^\ast E_X  \ar[d] \ar[r]^{s^\ast \rho}  &  s^\ast E'_X  \ar[d] \\
t^\ast E_X \ar[r] _ {t^\ast \rho} &  t^\ast E'_X 
}$$
commute. 

For an arbitrary Lie groupoid $G\rightrightarrows X$ and a Lie algebroid $(A,\psi)$ over $G\rightrightarrows X$, we then define the category $\mathrm{Rep}_{(A,\psi)}$ to be the category whose objects are tuples $((E,\psi'),\nabla)$, where $E$ is a vector bundle over $X$, $\psi : s^\ast E \to t^\ast E$ is an isomorphism satisfying the cocycle condition, and $\nabla$ is a representation of $A$ on $E$ that satisfies the above conditions. 

The above discussion shows that if $\mathbf A$ is a Lie algebroid over $\mathfrak X$ and $X \to \mathfrak X$ is an atlas, then we have a Lie algebroid $(A_X,\psi)$ over $X \times _\mathfrak X X \rightrightarrows X$ and an equivalence of categories:
$$\mathrm{Rep}_\mathbf A \simeq \mathrm{Rep}_{(A_X,\psi)}$$

\subsection{Cohomology}

\subsubsection{Sheaves over stacks}

We will use the sheaf cohomology for differentiable stacks developed in \cite{BeXu1} and \cite{BuShSp1}, which we summarise below (\ref{Sheaf cohomology for differentiable stacks}). Recall from \ref{The submersion and etale sites} that for any stack $\mathfrak X$ over $\mathbf{Man}$ we have a site called the submersion site of $\mathfrak X$, which we'll denote by $\mathfrak X_\text{sub}$. For a stack $\mathfrak X$ we therefore have the category $\mathrm{Sh}(\mathfrak X_\text{sub})$ of sheaves (of sets) over $\mathfrak X_\text{sub}$, and the category $\mathrm{Sh}_{\mathbb R }(\mathfrak X_\text{sub})$ of sheaves of $\mathbb R$-vector spaces over $\mathfrak X_\text{sub}$. 

If $F$ is a sheaf over $\mathfrak X_\text{sub}$ and $(U,u)$ is an object of $\mathfrak X_\text{sub}$ then there is an induced sheaf (in the usual sense of sheaves over topological spaces) $F_{U,u}$ over $U$: if $i:V \hookrightarrow U$ is an open subset of $U$ then we define
$$F_{U,u} (V) = F( V , u \circ i )$$
$F_{U,u}$ is called the `small sheaf' over $U$ associated to $F$ (and the map $u: U \to \mathfrak X)$. A sheaf $F$ over $\mathfrak X_\mathrm{sub}$ can then be defined by a collection of small sheaves $F_{U,u}$ together with morphisms 
$$F_{V,v} \to f_\ast F_{U,u}$$ 
for each morphism $(f,\alpha) : (U,u) \to (V,v)$ in $\mathfrak X_\text{sub}$ (these morphisms must satisfy a certain compatibility condition).

$\mathfrak X_\text{sub}$ has a structure sheaf $\mathrm C ^\infty _\mathfrak X$ defined by
\begin{align*}
(U,u)  &   \mapsto \mathrm C ^\infty (U)  \\
\left( (f,\alpha) : (U,u) \to (V,v) \right)  &  \mapsto \left( f^\ast : \mathrm C^\infty (V) \to \mathrm C^\infty (U) \right)
\end{align*}
This makes $(\mathfrak X_\mathrm{sub}, \mathrm C ^\infty _\mathfrak X)$ into a ringed site, and we therefore have the category $\mathrm C ^\infty _\mathfrak X \text{-mod}$ of sheaves of $\mathrm C ^\infty _\mathfrak X$-modules. We can take tensor products of $\mathrm C ^\infty _\mathfrak X$-modules `objectwise', this makes $\mathrm C ^\infty _\mathfrak X \text{-mod}$ into a monoidal category. Given a vector bundle $\mathbf E$ over $\mathfrak X$, considered as a functor $\mathbf E : \mathfrak X_\text{sub} \to \mathrm {Vect}$, we have a sheaf $\mathcal E$ of $\mathrm C ^\infty _\mathfrak X$-modules which associates to an object $(U,u)$ the $\mathrm C ^\infty (U)$-module $\Gamma(E_{U,u})$.

\subsubsection{Sheaf cohomology for differentiable stacks} \label{Sheaf cohomology for differentiable stacks}

We summarise the relevant results from \cite{BeXu1} and \cite{BuShSp1}. The category $\mathrm{Sh}_{\mathbb R }(\mathfrak X_\text{sub})$ is abelian and has enough injectives. We'll denote by $\mathrm{Ch} ^ + (\mathrm{Sh}_{\mathbb R }(\mathfrak X_\text{sub}))$  the category of bounded below chain complexes in $\mathrm{Sh}_{\mathbb R }(\mathfrak X_\text{sub})$, and by $\mathrm D ^ +  (\mathrm{Sh}_{\mathbb R }(\mathfrak X_\text{sub}))$ its derived category (see \cite{Wei1} for the relevant concepts from homological algebra). We define the global sections functor $\Gamma(\mathfrak X_\text{sub} ,-)$ by
\begin{align*}
\Gamma(\mathfrak X_\text{sub} , -) : \mathrm{Sh}_{\mathbb R }(\mathfrak X_\text{sub}) & \to \mathbb R \text{-mod}  \\
\mathrm F & \mapsto \mathrm{Hom}_{\mathrm{Sh}_{\mathbb R }(\mathfrak X_\text{sub})} ( \mathbb R _ \mathfrak X ,  F ) 
\end{align*}
where $\mathbb R_\mathfrak X$ is the sheaf that associates to an object $(U,u)$ the $\mathbb R$-vector space of locally constant functions $U \to \mathbb R$. If $X\to \mathfrak X$ is an atlas then
$$\Gamma(\mathfrak X_\text{sub} ,F) = \mathrm{Ker}\left(F(X) \rightrightarrows F(X\times_\mathfrak X X)\right)$$
where the right hand side is the kernel of the alternating sum of the two maps $F(X) \to F(X\times_\mathfrak X X)$ given by the two projections $X\times _\mathfrak X X \rightrightarrows X$. The global sections functor is left exact and so has right derived functors
$$ \mathrm H^i ( \mathfrak X_\text{sub} , - ) = \mathrm R^i \Gamma :  \mathrm{Sh}_{\mathbb R }(\mathfrak X_\text{sub})  \to \mathbb R \text{-mod}$$
If $F$ is a sheaf of $\mathbb R$-vector spaces over $\mathfrak X_\text{sub}$ then $\mathrm H^i ( \mathfrak X_\text{sub} , F )$ is the $i$'th sheaf cohomology group of $F$. More generally, there is a total right derived functor
$$\mathrm R \Gamma :  \mathrm D ^ +  (\mathrm{Sh}_{\mathbb R }(\mathfrak X_\text{sub}))  \to \mathrm D ^ +  ( \mathbb R \text{-mod} )$$
If $F^\bullet$ is a bounded below complex of sheaves of $\mathbb R$-vector spaces over $\mathfrak X _\text{sub}$ then we define
$$ \mathbb H ^ i (\mathfrak X_\text{sub}, F^\bullet )  = \mathrm H ^ i ( \mathrm R \Gamma (F^\bullet) )$$
and call it the $i$'th hypercohomolgy group of $F^\bullet$.

If $\mathfrak X$ is a differentiable stack and $X \to \mathfrak X$ is an atlas then we have the simplicial manifold $X_\bullet$:
$$\xymatrix{
\dots \ar@<1.5ex>[r]\ar@<.5ex>[r]\ar@<-.5ex>[r]\ar@<-1.5ex>[r] & X_2 \ar@<1ex>[r]\ar[r]\ar@<-1ex>[r] & X_1 \ar@<0.5ex>[r] \ar@<-.5ex>[r]  & X 
 }$$
where $X_n$ is the {n+1}-fold fibre product of $X$ over $\mathfrak X$ ($X_\bullet$ is the nerve of the Lie groupoid $X \times _\mathfrak X X \rightrightarrows X$). If $F$ is a sheaf of $\mathbb R$-vector spaces over $\mathfrak X_\text{sub}$ then (by choosing for each $n$ one of the maps $X_n \to \mathfrak X$) we get a diagram
$$\xymatrix{
\dots &  \ar@<1.5ex>[l]\ar@<.5ex>[l]\ar@<-.5ex>[l]\ar@<-1.5ex>[l]  F(X_2) & \ar@<1ex>[l]\ar[l]\ar@<-1ex>[l]  F(X_1) & \ar@<0.5ex>[l] \ar@<-.5ex>[l]   F(X) 
 }$$
which is a semi-cosimplicial $\mathbb R$-vector space, associated to which is its (un-normalized) cochain complex $\check {\mathrm C} ( X_\bullet , F)$ (see \cite{Wei1}). The n'th term in $\check{\mathrm C} ( X_\bullet , F)$ is $F(X_n)$, and the differentials are the alternating sums of the maps $F(X_n) \to F(X_{n+1})$. The cohomology groups $\Check{\mathrm H} ^ i  (X_\bullet, F)$ of $\check {\mathrm C} ( X_\bullet , F)$ are called the \v{C}ech cohomology groups of $F$ with respect to the atlas $X \to \mathfrak X$.

More generally, if $F^\bullet$ is a bounded below complex then we get a double complex $\check {\mathrm C} (X_\bullet , F^\bullet)$, with vertical differentials given by the simplicial structure of $X_\bullet$ and horizontal differentials coming from the complex $F^\bullet$. The cohomology groups $\Check { \mathbb H} ^ i  (X_\bullet , F^\bullet)$ of the total complex of $\check {\mathrm C} (X_\bullet , F^\bullet)$ are called the \v{C}ech cohomology groups of $F^\bullet$ with respect to the atlas $X \to \mathfrak X$. 

The main result we will use is that if, for all $n,m$, the small sheaf $(F^m)_{X_n}$ over $X_n$ induced by $F^m$ is acyclic (as a sheaf over the manifold $X_n$), then there is an isomorphism of cohomology groups:
$$\Check { \mathbb H} ^ i  (X_\bullet , F^\bullet) \cong \mathbb H ^i (\mathfrak X_\text{sub} , F^\bullet)$$

\subsubsection{The de Rham complex of a Lie algebroid over a stack}

Let $\mathbf A$ be a Lie algebroid over a stack $\mathfrak X$. We will construct a complex of sheaves over the submersion site of $\mathfrak X$. For $k \geq 0$ we set:
$$ \Omega ^ k _\mathbf A : (U,u) \mapsto \Omega^k (A_{U,u})$$
and 
$$\Omega ^ k _\mathbf A : \left(  (f,\alpha) : (U,u) \to (V,v) \right)  \mapsto (\mathbf A_{f,\alpha})^\ast : \Omega^k (A_{V,v}) \to \Omega^k (A_{U,u})$$
If
$$\xymatrix{
(U,u) \ar[r] ^ {(f,\alpha)} & (V,v) \ar[r] ^ {(g,\beta)} & (W,w)
}$$
is a composable pair of morphisms in $\mathfrak X_\text{sub}$ then we have 
\begin{align*}
(\mathbf A_{f,\alpha}) ^\ast \circ (\mathbf A_{g,\beta}) ^\ast   & = \left( \mathbf A_{g,\beta} \circ \mathbf A_{f,\alpha} \right) ^\ast \\
& = \left( \mathbf A _{gf,\beta \alpha}  \right) ^\ast 
\end{align*}
and therefore we have a presheaf of $\mathbb R$-vector spaces:
$$\Omega ^k _\mathbf A : \mathfrak X_\text{sub} \to \mathbb R \text{-mod}$$
In terms of `small sheaves' we have the following; 
$$ \Omega ^ k _\mathbf A  : (U,u) \mapsto \Omega^k _{A_{U,u}}$$
and 
$$\Omega ^ k _{\mathbf A} : \left(  (f,\alpha) : (U,u) \to (V,v) \right)  \mapsto (\mathbf A_{f,\alpha})^\ast : \Omega^k _{A_{V,v}} \to f_\ast \left( \Omega^k _{A_{U,u}} \right) $$
This shows that $\Omega ^k _\mathbf A$ is a sheaf over $\mathfrak X_\text{sub}$ in the open cover topology.

For each $(U,u)$ we have the differential
$$ d _ {A_{U,u}} : \Omega ^k (A_{U,u}) \to \Omega ^{k+1} ( A_{U,u}) $$
and for each $(f,\alpha):(U,u) \to (V,v)$ these commute with the pullback maps:
$$ d_ {A_{U,u}} \circ (\mathbf A _{f,\alpha} ) ^\ast  =  (\mathbf A _{f,\alpha} ) ^\ast \circ d_{A_{V,v}} $$
We therefore have morphisms of sheaves
$$\mathbf d _\mathbf A : \Omega^k _\mathbf A \to \Omega^{k+1} _\mathbf A$$
and a complex of sheaves over $\mathfrak X_\text{sub}$, which we'll denote by $( \Omega ^\bullet _\mathbf A , \mathbf d_\mathbf A)$ or just $\Omega^\bullet _\mathbf A$.

Note that if $\mathfrak X$ is a differentiable stack then the de Rham complex of $\mathbf A$ will always be concentrated in infinitely many degrees: as described in \ref{Lie algebroids over stacks are not vector bundles},
if $x:X \to \mathfrak X$ is an atlas, and $A_{X,x}$ has rank $r$ as a vector bundle over $X$, then for any submersion $f:Y \to X$, we have that $(Y \to X \to \mathfrak X)$ is an object of the site $\mathfrak X_\text{sub}$, and if $Y \to X$ has relative dimension $s$, then $A_{Y,x \circ f}$ has rank $r+s$. Therefore $\Omega^k  (A_{Y,x \circ f})$ is non-zero for all $k \leq r+s$, and there exist submersions $Y \to X$ of arbitrarily high relative dimension $s$. However, $\Omega^\bullet _\mathbf A$ is bounded below, so is an object in $\mathrm{Ch} ^ + (\mathrm{Sh}_{\mathbb R }(\mathfrak X_\text{sub}))$ and $\mathrm D ^ +  (\mathrm{Sh}_{\mathbb R }(\mathfrak X_\text{sub}))$.

\begin{Definition}
$( \Omega ^\bullet _\mathbf A , \mathbf d_\mathbf A)$ is the de Rham complex of $\mathbf A$. It is an object in $\mathrm D ^ +  (\mathrm{Sh}_{\mathbb R }(\mathfrak X_\text{sub}))$.
\end{Definition}

The construction above is functorial for morphisms of Lie algebroids over stacks, so for any stack $\mathfrak X$ there is a contravariant functor
$$\Omega ^\bullet  : \mathcal{LA}_\mathfrak X \to \mathrm D ^ +  (\mathrm{Sh}_{\mathbb R }(\mathfrak X_\text{sub}))$$

\subsubsection{The de Rham complex with coefficients a representation}

Let $\mathbf A$ be a Lie algebroid over $\mathfrak X$ and $(\mathbf E,  \nabla)$ be a representation of $\mathfrak X$. Then we define the sheaf $\Omega ^ k _{(\mathbf A,\mathbf E)}$ as the tensor product in the category of sheaves of $\mathrm C ^\infty _\mathfrak X$-modules:
$$\Omega ^ k _{(\mathbf A,\mathbf E)} \equiv \Omega ^ k _\mathbf A  \otimes 
 _ { \mathrm C ^\infty _\mathfrak X}  \mathbf {\mathcal E}$$
where $\mathcal E$ is the sheaf of sections of $\mathbf E$. Explicitly, we have for each object $(U,u)$ of $\mathfrak X_\text{sub}$ :
\begin{align*}
\Omega ^ k _{(\mathbf A,\mathbf E)} : (U,u) & \mapsto  \Omega^k (A_{U,u}) \otimes _{\mathrm C ^\infty (U)} \Gamma (E_{U,u}) \\
& = \Gamma ( {\bigwedge} ^k A_{U,u} ^\ast \otimes E_{U,u} )
\end{align*}
and for each morphism $(f,\alpha):(U,u) \to (V,v)$ we have the map (see  \ref{Pullbacks of representations}): 
$$ (\mathbf A_{f,\alpha}) ^ \ast \otimes (\mathbf E_{f,\alpha}) ^\ast :  \Omega^k (A_{V,v}) \otimes _{\mathrm C ^\infty (U)} \Gamma (E_{V,v}) \to 
\Omega^k (A_{U,u}) \otimes _{\mathrm C ^\infty (U)} \Gamma (E_{U,u})$$
The compatibility condition for the representations $\nabla_{U,u}$ (see \ref{The data determining a representation}) imply that the differentials
$$ \mathrm d_{ (  A_{U,u},\nabla_{U,u}  )  } :  \Omega^k (A_{U,u}) \otimes _{\mathrm C ^\infty (U)} \Gamma (E_{U,u}) 
\to \Omega^{k+1} (A_{U,u}) \otimes _{\mathrm C ^\infty (U)} \Gamma (E_{U,u}) $$
give a morphism of sheaves of vector spaces over $\mathfrak X_\text{sub}$:
$$\mathbf  d _ { ( \mathbf A , \nabla ) }  : \Omega ^ k _ {(\mathbf A, \mathbf E)}  \to \Omega ^ {k+1} _ {(\mathbf A, \mathbf E)} $$
We therefore have a bounded below complex which we'll denote by $(\Omega ^ \bullet _{(\mathbf A, \mathbf E)} , \mathbf d _ {(\mathbf A, \mathbf \nabla)} )$ or just $\Omega ^ \bullet _{(\mathbf A, \mathbf E)}$.

\begin{Definition}
$(\Omega ^ \bullet _{(\mathbf A, \mathbf E)} , \mathbf d _ {(\mathbf A, \mathbf \nabla)} )$ is the de Rham complex of $\mathbf A$ with values in the representation $(\mathbf E, \nabla)$. It is an object in $\mathrm D ^ +  (\mathrm{Sh}_{\mathbb R }\mathfrak X_\text{sub})$.
\end{Definition}

\subsubsection{Cohomology}

We can now define Lie algebroid cohomology for Lie algebroids over stacks:

\begin{Definition} \label{Lie algebroid cohomology over stacks}
If $\mathbf A$ is a Lie algebroid over a stack $\mathfrak X$ then we define the Lie algebroid cohomology $\mathrm H^\bullet (\mathbf A)$ of $\mathbf A$ as the hypercohomology 
$$\mathrm H^\bullet (\mathbf A)  = \mathbb H^\bullet (\mathfrak X_\text{sub} ,  \Omega ^\bullet _\mathbf A )$$
of the de Rham complex of $\mathbf A$. 

If $(\mathbf E, \nabla)$ is a representation of $\mathbf A$ then we define the Lie algebroid cohomology $\mathrm H ^\bullet (\mathbf A, \mathbf E, \nabla)$ with coefficients in $(\mathbf E, \nabla)$ as the hypercohomology $$\mathrm H ^\bullet (\mathbf A, \mathbf E, \nabla) = \mathbb H ^\bullet (\mathfrak X_\text{sub} ,  \Omega ^ \bullet _ {(\mathbf A, \mathbf E)}  )$$
of the de Rham complex of $\mathbf A$ with values in the representation $(\mathbf E, \nabla)$.
\end{Definition}

Since the de Rham functor $\Omega^\bullet$ and the hypercohomology functor $\mathrm R \Gamma $ are both functors, Lie algebroid cohomology is a contravariant functor
$$\mathrm H^\bullet : \mathcal{LA}_\mathfrak X \to \mathbb R \text{-mod}$$

Under some reasonable assumptions on a stack $\mathfrak X$ we can compute Lie algebroid cohomology groups in terms of the \v{C}ech cohomology described in \ref{Sheaf cohomology for differentiable stacks}:

\begin{Proposition}
Let $\mathbf A$ be a Lie algebroid over a differentiable stack $\mathfrak X$, and $X \to \mathfrak X$ be an atlas. If the manifolds $X_n$ are all Hausdorff then there are isomorphims of cohomology groups:
$$ \mathrm H^\bullet (\mathbf A)   \cong \Check{\mathbb H}^\bullet (X_\bullet ,  \Omega ^\bullet _\mathbf A )$$
Under the same assumptions, if $(\mathbf E,\nabla)$ is a representation of $\mathbf A$ then there are isomorphisms:
$$\mathrm H ^\bullet (\mathbf A, \mathbf E, \nabla)  \cong \Check{\mathbb H} ^\bullet (X_\bullet ,  \Omega ^ \bullet _ {(\mathbf A, \mathbf E)}  )  $$
\end{Proposition}

\begin{proof}
For each $n$ and $k$ the small sheaf $(\Omega^k _\mathbf A)_{X_n}$ over the manifold $X_n$ induced by $\Omega^k _\mathbf A$ is the sheaf of sections of the vector bundle ${\bigwedge}^k A_{X_n} ^\ast$. 
The result then follows from the results of \cite{BeXu1} and \cite{BuShSp1} described in \ref{Sheaf cohomology for differentiable stacks} and the fact that the sheaf of sections of a vector bundle over a Hausdorff manifold is acyclic. The case of cohomology with coefficients is the same.
\end{proof}

\subsection{Lie algebroids over \'etale stacks}  \label{Lie algebroids over etale stacks}

\subsubsection{The rank of a Lie algebroid over an \'etale stack} \label{Use of etale site}

If $\mathfrak X$ is an \'etale differentiable stack then it is sufficient to work over the \'etale site of $\mathfrak X$: if $X \to \mathfrak X$ is an \'etale atlas then we have equivalences of categories:
$$
\mathrm{Hom} _ {\mathrm{PSh}(\mathbf{Man}_\text{\'et})} ( \mathfrak X_\text{\'et} , \mathcal {LA} )  \simeq 
\mathrm{Des} ( X \to \mathfrak X , \mathcal{LA} ) 
\simeq
\mathrm{Hom} _ {\mathrm{PSh}(\mathbf{Man}_\text{sub})} ( \mathfrak X_\text{sub} , \mathcal {LA} )
$$
Over $\mathbf{Man}_\text{\'et}$ we have a natural morphism of stacks $\mathcal{LA} \to \mathrm{Vect}$  (see
\ref{Lie algebroids vs vector bundles}), and therefore a Lie algebroid over an \'etale stack $\mathfrak X$ determines a morphism $\mathfrak X_\text{\'et} \to \mathrm{Vect}$ in a natural way.

In particular, unlike the case of Lie algebroids over general stacks, a Lie algebroid $\mathbf A$ over an \'etale stack $\mathfrak X$ has a well defined rank - the rank of the vector bundle $A_{U,u}$ over $U$ for any \'etale map $u: U \to \mathfrak X$.

\subsubsection{Cohomology of Lie algebroids over \'etale stacks}  \label{Cohomology of Lie algebroids over etale stacks}

In the case of orbifolds / Deligne-Mumford differentiable stacks, the \v{C}ech complex associated to a Lie algebroid simplifies. Let $x:X \to \mathfrak X$ be an \'etale atlas of such a stack, so that $X_1 \rightrightarrows X$ is a proper \'etale Lie groupoid (recall that $X_1 = X\times _\mathfrak X X$). This implies that the simplicial maps
$$\xymatrix{
\dots \ar@<1.5ex>[r]\ar@<.5ex>[r]\ar@<-.5ex>[r]\ar@<-1.5ex>[r] & X_2 \ar@<1ex>[r]\ar[r]\ar@<-1ex>[r] & X_1 \ar@<0.5ex>[r] \ar@<-.5ex>[r]  & X 
 }$$
are all \'etale. Choose, for each $n$, one of the maps $X_n \to X$, and label it $x_n$. If $\mathbf A$ is a Lie algebroid over $\mathfrak X$ then since each map $x_n$ is \'etale we have:
$$A_{X_n , x \circ x_n} \cong (x_n)^! A_{X,x} \cong (x_n)^\ast A_{X,x}$$
Moreover, for each $n$ and each $k \geq 0$, the corresponding term in the \v{C}ech complex $\check {\mathrm C} (X_\bullet , \Omega_\mathbf A^\bullet)$ is
$$ \Omega ^ k (A_{X_n,(x\circ x_n)})  =  \Gamma \left( {\bigwedge }  ^k A^\ast_{X_n, x \circ x_n} \right) \cong   \Gamma \left(   (x_n)^\ast  {\bigwedge}^k A^\ast_{X,x}   \right)  $$
The same argument as given in \cite{TuXu1}, Lemma 3.1, now applies: under the isomorphisms above, the vertical differentials in the double complex $\check {\mathrm C} (X_\bullet , \Omega_\mathbf A^\bullet)$ correspond to pulling back sections of the vector bundles  $(x_n)^\ast  {\bigwedge}^k A^\ast_{X,x}$ along the maps $X_{n+1}  \to X_n$, and so the cohomology of the $k$'th column is isomorphic to the groupoid cohomology of $X_1 \rightrightarrows X$ with coefficients in the vector bundle ${\bigwedge}^k A_{X,x}^\ast$ (see \cite{Cr1} for the general definition of groupoid cohomology). It then follows from the fact, proved in \cite{Cr1}, that groupoid cohomology of proper Lie groupoids is always zero in positive degrees, and in degree zero is equal to the space of invariant sections of the coefficient vector bundle, and the proof given in \cite{TuXu1}, that there is an isomorphism of cohomology groups:
$$ \Check{\mathbb H}^\bullet (X_\bullet ,  \Omega ^\bullet _\mathbf A )  \cong  \mathrm H ^\bullet \left( \Omega^\bullet(A_{X,x})^{X_1 \rightrightarrows X} \right)$$ 
where the right hand side is the cohomology of the subcomplex of $\Omega^\bullet (A_{X,x})$ given by $X_1 \rightrightarrows X$ invariant sections of the ${\bigwedge}^k A_{X,x}^\ast$'s. Recalling the description of the global sections functor given in \ref{Sheaf cohomology for differentiable stacks} we then have:

\begin{Proposition} \label{Deligne-Mumford cohomology}
Let $\mathbf A$ be a Lie algebroid over a Deligne-Mumford differentiable stack $\mathfrak X$. If there exists an \'etale atlas $X \to \mathfrak X$ such that the manifolds $X_n$ are all Hausdorff then there are isomorphisms:
$$ \mathrm H ^\bullet ( \mathbf A)  \cong \mathrm H^\bullet \left( \Gamma \left( \mathfrak X_\mathrm{sub} ,\Omega^\bullet _\mathbf A \right)  \right)$$
where the right hand side is the cohomology of the complex of global sections of the de Rham complex of $\mathbf A$.
\end{Proposition}

Applying this to the case of manifolds we get:

\begin{Proposition} \label{Manifold cohomology}
If $X$ is a Hausdorff manifold, then the canonical equivalence of categories $\mathcal{LA}_{\underline X} \simeq \mathcal{LA}_X$
induces isomorphisms between the stack-theoretic and standard Lie algebroid cohomology groups.
\end{Proposition} 

\begin{proof}
The stack $\underline X$ is Deligne-Mumford and $X \to \underline X$ is an atlas, the Lie groupoid $X \times _{\underline X} X \rightrightarrows X$ is just the trivial Lie groupoid $X \rightrightarrows X$, and the manifolds $X_n \cong X$ are Hausdorff. If $\mathbf A$ is a Lie algebroid over $\underline X$ with induced Lie algebroid $A$ over $X$, then from Proposition \ref{Deligne-Mumford cohomology} we have
$$ \mathrm H^\bullet (\mathbf A )  
\cong  \mathrm H ^\bullet \left( \Omega^\bullet(A)^{X \rightrightarrows X} \right) 
= \mathrm H ^\bullet (A)$$
\end{proof}

\subsubsection{Lie algebroids over \'etale Lie groupoids}

If we fix an \'etale atlas $X \to \mathfrak X$ of an \'etale stack $\mathfrak X$ then we'll show that we get a description of Lie algebroids over $\mathfrak X$ that almost coincides with the definition given in \cite{Ro1} in terms of \'etale groupoids. 

Recall that if $E$ is a vector bundle with a left action of an \'etale groupoid then there is also a naturally defined action on the sheaf of sections of $E$. The tangent bundle $TX$ of the base of an \'etale groupoid $G\rightrightarrows X$ has a natural left action of $G\rightrightarrows X$ because each morphism $g \in G$ defines a germ of a diffeomorphism from $s(g)$ to $t(g)$, see for example \cite{MoMr1} and \cite{MoMr2} for this and more generally for the theory of sheaves over \'etale groupoids.

\begin{Proposition}  \label{G-sheaves of Lie algebras...}
Let $G\rightrightarrows X$ be an \'etale Lie groupoid, $A$ a Lie algebroid over $X$, and $\psi:s^\ast A \to t^\ast A$ an isomorphism of vector bundles that satisfies the cocycle condition. Then the following are equivalent:
\begin{itemize}
\item $\psi: s^\ast A \to t^\ast A$ is an isomorphism of Lie algebroids, where $s^\ast A$ and $t^\ast A$ are equipped with Lie algebroid structures via the isomorphisms of vector bundles $s^\ast A \cong s^! A$ and $t^\ast A \cong t^! A$.
\item The anchor map $a:A \to TX$ is equivariant with respect to the actions of $G\rightrightarrows X$, and the sheaf $\mathcal A$ of sections of $A$ is a $G \rightrightarrows X$-sheaf of Lie algebras.
\end{itemize}
\end{Proposition}

\begin{proof}
Let $W \subset G$ be a bisection, by which we mean an open subset of $G$ such that the restrictions of $s$ and $t$ to $W$ are diffeomorphisms onto their images. Associated to $W$ is a diffeomorphism
$$l_W = t \circ {(s|_{s(W)})}^{-1}  : s(W) \to t(W)$$
The action of $G\rightrightarrows X$ on $A$ gives an isomorphism $L_W$ of vector bundles covering $l_W$:
$$\xymatrix{
A | _{s(W)} \ar[d] \ar[r] ^{L_W} & A |_{t(W)}  \ar[d] \\
s(W)  \ar[r] _{l_W} & t(W) 
}$$
where $L_W$ is defined as the composition of the diffeomorphisms
$$\xymatrix{
 A |_{s(W)}  \ar[r] &  (s^\ast A) |_W  \ar[r] ^\psi &  (t^\ast A)|_W \ar[r] & A|_{t(W)}
}$$
It then follows that $\psi:s^\ast A \to t^\ast A$ is an isomorphism of Lie algebroids over $G$ if and only if $L_W$ is an isomorphism of Lie algebroids for every bisection $W$. 

The action of $G\rightrightarrows X$ on $TX$ is given by 
\begin{align*} 
g : T_{s(g)} X & \to T_{t(g)} X  \\
v & \mapsto (l_W)_\ast (v) 
\end{align*}
where $g \in G$ and $W$ is a bisection containing $g$, and the action of $G\rightrightarrows X$ on $\mathcal A$ is given by
\begin{align*} 
g : \mathcal A_{s(g)}  & \to \mathcal A_{t(g)}   \\
[\xi]_{s(g)} & \mapsto [L_W(\xi)]_{t(g)}
\end{align*}
where $[\xi]_{s(g)}$ is an element of the stalk $\mathcal A_{s(g)}$ represented by a local section $\xi \in \Gamma(s(w),A)$, and $L_W(\xi) \in \Gamma(t(W) , A)$ is the section of $A$ over $t(W)$ corresponding to $\xi$ under the bijection $\Gamma(s(W),A) \cong \Gamma(t(W),A)$ induced by $L_W$.

The map $L_W : A|_{s(W)} \to A|_{t(W)}$ corresponding to a bisection $W$ is an isomorphism of Lie algebroids if and only if 
$$ a \circ L_W = (l_W)_\ast \circ a $$
and 
$$ [ L_W (\xi) , L_W (\xi') ] = L_W [\xi,\xi'] $$
for all $\xi,\xi' \in \Gamma(s(W),A)$. It is then clear that $L_W$ is an isomorphism of Lie algebroids for all bisections $W$ if and only if the anchor map $a:A \to TX$ is equivariant and for each $g \in G$, the map $\mathcal A_{s(g)} \to \mathcal A_{t(g)}$ is an isomorphism of Lie algebras, which is exactly the condition that $\mathcal A$ is a $G\rightrightarrows X$-sheaf of Lie algebras.
\end{proof}

\subsubsection{Comparison with the work of Roggiero Ayala}

In \cite{Ro1} a Lie algebroid over an \'etale groupoid $G\rightrightarrows X$ is defined to be a Lie algebroid $A$ over $X$ together with a left action of $G\rightrightarrows X$ such that the anchor map is equivariant and the space of $G\rightrightarrows X$ invariant sections is closed under the Lie bracket. The condition that $\mathcal A$ is a $G\rightrightarrows X$-sheaf of Lie algebras implies that the space of invariant sections is closed under the Lie bracket, but is in general a stronger condition. For example, if $G$ is a discrete group, considered as a Lie groupoid $G\rightrightarrows \ast$ over a point, and $\mathfrak g$ is a Lie algebra with a linear action of $G$, then the weaker condition is that $[\mathfrak g ^G, \mathfrak g ^G] \subset \mathfrak g ^G$, which is vacuous if $G$ acts irreducibly on $\mathfrak g$, whereas the stronger condition is that $G$ acts by Lie algebra automorphisms.

\newpage

\newpage

\section{Lie algebroids and $\mathcal{LA}$-groupoids}  \label{Lie algebroids over Lie groupoids}

In this section we explain the relationship between Lie algebroids over Lie groupoids, as defined in \ref{Lie algebroids in terms of atlases}, and $\mathcal{LA}$-groupoids, which are groupoid objects in the category of Lie algebroids. The main result is Theorem \ref{inverse functor}, which gives an equivalence between the category of Lie algebroids over Lie groupoids, and a certain subcategory $\text{!-}\mathbf{Gpd}\mathcal{LA}$ of the category $\mathbf{Gpd}\mathcal{LA}$ of $\mathcal{LA}$-groupoids.  The definition of this subcategory, and the details of the equivalence, are similar to the theory of `vacant' $\mathcal{LA}$-groupoids developed in \cite{Ma1}.

In particular, the results of this section show that a Lie algebroid $\mathbf A$ over a differentiable stack $\mathfrak X$ induces a $\mathcal{LA}$-groupoid with base $X\times_\mathfrak X X \rightrightarrows X$ for every atlas $X \to \mathfrak X$ of $\mathfrak X$. This clarifies the relationship between Lie algebroids over stacks and $\mathcal{LA}$-groupoids, as suggested in \cite{Meh1}.

Most of the material in this section is independent of the results of sections \ref{The stack of Lie algebroids} and \ref{Lie algebroids over differentiable stacks} and this section is more or less self-contained.

\subsection{Actions of Lie groupoids on Lie algebroids}  \label{Actions of Lie groupoids on Lie algebroids}

We first formulate the notion of action in terms of an isomorphism between the two pullbacks of a Lie algebroid along the source and target maps. Since Lie algebroid pullbacks are not vector bundle pullbacks this does not give an action of the groupoid on the underlying vector bundle in the usual sense. However, it turns out that there is an induced action of the tangent groupoid, and it is possible to formulate the definition purely in terms of this action.

\subsubsection{Actions in terms of $!$-cocycles}

\begin{Definition}
A Lie algebroid over a Lie groupoid $G\rightrightarrows X$ is a Lie algebroid $A$ over $X$ together with an isomorphism: 
$$\psi: s^! A \to t^! A$$ 
of Lie algebroids over $G$ that satisfies the cocycle condition: 
$$m^! \psi = \mathrm{pr}_1^! \psi \circ \mathrm{pr}_2 ^! \psi$$ 
over $G_2 = G_s \! \times _t G$, where $\mathrm{pr}_1,\mathrm{pr}_2$ are the projections
\begin{align*}
\mathrm{pr}_1,\mathrm{pr}_2 : G _s\! \times _t G & \to G \\
\mathrm{pr}_1 (g,h) & = g \\
\mathrm{pr}_2 (g,h) & = h
\end{align*}

 We'll denote a Lie algebroid over a Lie groupoid by a triple $(G\rightrightarrows X, A, \psi)$ or just $(A,\psi)$ or $A$ when $G\rightrightarrows X$ is clear.
\end{Definition}

\begin{Remark}
We have suppressed the natural isomorphisms in the cocycle condition. Written out in full, the cocycle condition is:
$$(c_{t,\mathrm{pr}_1} \circ c_{t,m}^{-1}) \circ m^! \psi \circ (c_{s,m} \circ c_{s,\mathrm{pr}_2}^{-1}) = \mathrm{pr}_1^! \psi \circ (c_{s,\mathrm{pr}_1} \circ c_{t,\mathrm{pr}_2}^{-1}) \circ \mathrm{pr}_2^! \psi$$
as maps from $\mathrm{pr}_2^! s^! A$ to $\mathrm{pr}_1^! t^! A$, where the maps of the form $c_{t,\mathrm{pr}_1}$ are the isomorphisms determined by the universal property of pullbacks of Lie algebroids, see \ref{LA as a weak presheaf of categories}.
\end{Remark}

\begin{Definition} \label{!-morphisms}
If $(A,\psi)$ is a Lie algebroid over $G\rightrightarrows X$ and $(A',\psi')$ is a Lie algebroid over $G'\rightrightarrows X'$, then a morphism from $(G\rightrightarrows X,A,\psi)$ to $(G'\rightrightarrows X',A',\psi')$ is a triple of smooth maps $(\phi,f,\rho)$ where:
$$\xymatrix{
G \ar@<-.5ex>[d] \ar@<.5ex>[d] \ar[r]^\phi  & G'  \ar@<-.5ex>[d] \ar@<.5ex>[d] \\
X \ar[r]_f  & X'
}$$
is a morphism of Lie groupoids, 
$$\xymatrix{
A \ar[d] \ar[r]^\rho  & A'  \ar[d] \\
X \ar[r]_f & X' 
}$$
is a morphism of Lie algebroids, and the diagram:
$$\xymatrix{
s^! A \ar[d]_\psi \ar[r]^{(\phi_\ast,\rho)} & s^! A' \ar[d]^{\psi'}  \\
t^! A \ar[r] _{(\phi_\ast, \rho)} & t^! A' 
}$$
commutes. We'll denote such a morphism by:
$$(\phi,f,\rho) : (G\rightrightarrows X, A, \psi) \to (G'\rightrightarrows X', A', \psi')$$
\end{Definition}

\begin{Remark} \label{square morphism}
The horizontal maps in the last diagram are the Lie algebroid morphisms constructed as follows. Applying the tangent functor to the groupoid morphism $(\phi,f)$ we get a morphism $(\phi_\ast,f_\ast)$ between tangent groupoids, which implies that the following diagram commutes:
$$\xymatrix{
& & (s')^! A \ar[d]  \ar[r] & A' \ar[d] \\
s^! A \ar[d] \ar[r]  & A \ar[d] \ar[urr]^<<<<<<<<<<<<<{\rho} & TG' \ar[r]^{(s')_\ast}  & TX' \\
TG \ar[r]_{s_\ast} \ar[urr]^>>>>>>>>>{\phi_\ast} & TX  \ar[urr]_{f_\ast}\\
}$$
where the vertical maps are the anchor maps. Viewing $(s')^!A'$ as a fibre product in the category of Lie algebroids, it follows that there is a Lie algebroid morphism from $s^!A$ to $(s')^!A'$ given by the formula: 
\begin{align*}
(\phi_\ast, \rho): s^!A & \to (s')^! A' \\
(v,\xi) & \mapsto (\phi_\ast (v), \rho(\xi))
\end{align*}
Similarly, there is a morphism:
\begin{align*}
(\phi_\ast, \rho): t^!A & \to (t')^! A' \\
(v,\xi) & \mapsto (\phi_\ast (v), \rho(\xi))
\end{align*}
\end{Remark}

\begin{Proposition}
Given two morphisms:
$$(\phi,f,\rho) : (G\rightrightarrows X, A, \psi) \to (G'\rightrightarrows X', A', \psi')$$
and:
$$(\phi',f',\rho') : (G'\rightrightarrows X', A', \psi') \to (G''\rightrightarrows X'', A'', \psi'')$$
as in Definition \ref{!-morphisms}, their composition, defined as:
$$(\phi' \circ \phi,f'\circ f,\rho' \circ \rho) : (G\rightrightarrows X, A, \psi) \to (G''\rightrightarrows X'', A'', \psi'')$$
is a morphism.
\end{Proposition}

\begin{proof}
It is immediate that:
$$\xymatrix{
G \ar@<-.5ex>[d] \ar@<.5ex>[d] \ar[r]^{\phi'\circ\phi}  & G''  \ar@<-.5ex>[d] \ar@<.5ex>[d] \\
X \ar[r]_{f'\circ f}  & X''
}$$
is a morphism of Lie groupoids and:
$$\xymatrix{
A \ar[d] \ar[r]^{\rho'\circ \rho}  & A''  \ar[d] \\
X \ar[r]_{f' \circ f} & X'' 
}$$
is a morphism of Lie algebroids. Since $(\phi,f,\rho)$ and $(\phi',f',\rho')$ are morphisms the following diagram commutes:
$$\xymatrix{
s^! A \ar[d]_\psi \ar[r]^{(\phi_\ast,\rho)} & (s')^! A' \ar[d]_{\psi'} \ar[r] ^{(\phi'_\ast,\rho')} & (s'')^! A''  \ar[d]^{\psi''}  \\
t^! A \ar[r]_{(\phi_\ast,\rho)} & (t')^! A' \ar[r]_{(\phi'_\ast,\rho')} & (t'')^! A'' 
}$$
It then follows from the fact that $(\phi' \circ \phi)_\ast = \phi'_\ast \circ \phi_\ast$ that the diagram:
$$\xymatrix{
s^! A \ar[d]_{((\phi'\circ\phi)_\ast,(\rho'\circ\rho))} \ar[r]^{\psi}  & t^! A \ar[d] ^{((\phi'\circ\phi)_\ast,(\rho'\circ\rho))}\\
(s'')^! A' \ar[r]_{\psi''} & (t'')^! A''
}$$
commutes as required.
\end{proof}

\begin{Definition}
Lie algebroids over Lie groupoids and morphisms defined as in Definition \ref{!-morphisms} form a category, which we'll denote $\text{LieGpd}\ltimes \mathcal{LA}$. 
\end{Definition}

\begin{Definition}
For a fixed Lie algebroid $G\rightrightarrows X$, we define the subcategory $\mathcal{LA}_G$ of $\text{LieGpd}\ltimes \mathcal{LA}$ to consist of Lie algebroids over $G\rightrightarrows X$ and morphisms of the form $(\mathrm{id}_G, \mathrm{id}_X, \rho)$.
\end{Definition}

\begin{Remark}
There is an obvious forgetful functor from $\text{LieGpd}\ltimes \mathcal{LA}$ to the category of Lie groupoids, and the categories $\mathcal{LA}_G$ are exactly the fibres of this functor.
\end{Remark}

\subsubsection{Actions in terms of tangent groupoids}

We will show that Lie algebroids over a Lie groupoid $G\rightrightarrows X$ can be described in terms of actions of the tangent groupoid $TG\rightrightarrows TX$. Recall that if $G \rightrightarrows X$ is a Lie groupoid with structure maps $s,t,m,u,i$ then the tangent groupoid $TG\rightrightarrows TX$ consists of the manifolds $TG$ and $TX$, and has structure maps $s_\ast,t_\ast,m_\ast,u_\ast,i_\ast$. The tangent groupoid construction gives a functor
$$\mathrm{LieGpd} \to \mathrm{LieGpd}$$

\begin{Definition} \label{compatible actions}
Let $G \rightrightarrows X$ be a Lie groupoid and $A$ a Lie algebroid over $X$, with anchor map $a$. Then a \emph{compatible action} of $TG \rightrightarrows TX$ on $A$ is an action of the groupoid $TG \rightrightarrows TX$ on $A$ along $a:A \to TX$, such that the map: 
$$TG_{s_\ast} \! \!  \times _a A \to A$$ 
defining the action is a Lie algebroid morphism covering the target map $t: G \to X$, where we consider $TG_{s_\ast} \! \!  \times_a A$ as a Lie algebroid over $G$ by identifying it with the Lie algebroid pullback along the source map:
$$TG_{s_\ast} \! \! \times _a A = s^! A$$
\end{Definition}

\begin{Proposition} \label{tangent actions}
Let $G\rightrightarrows X$ be a Lie groupoid and $A$ be a Lie algebroid over $X$. Then there is a bijection between isomorphisms $s^! A \to t^! A$ satisfying the cocycle condition and compatible actions $TG_{s_\ast} \times _a A \to A$ of $TG$ on $A$.
\end{Proposition}

\begin{proof}
Let $A$ be a Lie algebroid over $X$. Then by the universal property of pullbacks there is a bijection between morphisms $\psi:s^! A \to t^! A$ over $G$, and morphisms $\tilde \psi: s^! A \to A$ covering $t$. The correspondence is given by: 
$$\psi ( v, \xi ) = ( v, \tilde \psi (v,\xi) )$$
for $(v,\xi) \in s^! A$. We can view $A$ as a space over $TX$ via the anchor map $a:A \to TX$, then we have:
\begin{align*}
s^! A & = (s_\ast)^\ast A  \\
t^! A & = (t_\ast) ^ \ast A
\end{align*}
If $\psi: s^! A \to t^! A$ is an isomorphism of Lie algebroids over $G$ then it commutes with the anchor maps, so we can consider it as an isomorphism of spaces over $TG$:
$$\xymatrix{
(s_\ast)^\ast A  \ar[dr] \ar[r] ^\psi  &  (t_\ast)^\ast A  \ar[d] \\
&   TG
}$$
We can then write the !-cocycle condition for $\psi$
$$m^! \psi = \mathrm{pr}_1^! \psi \circ \mathrm{pr}_2 ^! \psi$$
as
$$(m_\ast)^\ast \psi = ((\mathrm{pr}_1)_\ast)^\ast \psi  \circ  ((\mathrm{pr}_2)_\ast)^\ast \psi$$
The canonical identification of $(TG)_2$ with $T(G_2)$ identifies the map
$$(\mathrm{pr}_1)_\ast : T(G_2)  \to TG$$
with the projection
$$\mathrm{pr}_1 : (TG)_2 \to TG$$
and similarly for $(\mathrm{pr}_2)_\ast$. Therefore the !-cocycle condition for $\psi$ is exactly the standard cocycle condition for $\psi$ to determine an action of $TG\rightrightarrows TX$ on $A$ along $a:A \to TX$.

Alternatively, if we set $\psi(v,\xi) = (v, \tilde{\psi}(v,\xi))$ as above, and use the identification $TG_{s_\ast} \times _{t_\ast} TG$, then if we write out the terms appearing in the cocycle condition explicitly we get:
\begin{align*}
\mathrm{pr}_2^! \psi : \mathrm{pr}_2 ^! s^! A & \to \mathrm{pr}_2 ^! t^! A  \\
( (v,w) , (w,\xi) ) & \mapsto ( (v,w) ,  \psi (w,\xi) ) \\
& = ( (v,w) , (w,\tilde \psi(w,\xi)) ) \\
\\
\mathrm{pr}_1^! \psi : \mathrm{pr}_1 ^! s^! A & \to \mathrm{pr}_1 ^! t^! A  \\
( (v,w) , (v,\xi) ) & \mapsto ( (v,w) ,  \psi (v,\xi) ) \\
& = ( (v,w) , (v,\tilde \psi(v,\xi)) ) \\
\\
m^! \psi : m ^! s^! A & \to m ^! t^! A  \\
( (v,w) , (m_\ast(v,w),\xi) ) & \mapsto ( (v,w) ,  \psi (m_\ast(v,w),\xi) ) \\
& = ( (v,w) , (m_\ast(v,w),\tilde \psi(m_\ast(v,w),\xi)) )
\end{align*}
Including the natural isomorphisms we then have:
\begin{align*}
\left(  (c_{t,\mathrm{pr}_1} \circ c_{t,m}^{-1}) \circ m^! \psi \circ (c_{s,m} \circ c_{s,\mathrm{pr}_2}^{-1}) \right)   ( (v,w) , (w,\xi) )  & =          
\left(  (c_{t,\mathrm{pr}_1} \circ c_{t,m}^{-1}) \circ m^! \psi  \right)  
((v,w),(m_\ast(v,w),(w,\xi)))     \\
& = \left(  c_{t,\mathrm{pr}_1} \circ c_{t,m}^{-1} \right) 
( (v,w) , (m_\ast(v,w),\tilde \psi(m_\ast(v,w),\xi)) ) \\ 
& = ( (v,w) , (v,\tilde \psi(m_\ast(v,w),\xi)) ) \\
\\
\left(  \mathrm{pr}_1^! \psi \circ (c_{s,\mathrm{pr}_1} \circ c_{t,\mathrm{pr}_2}^{-1}) \circ \mathrm{pr}_2^! \psi  \right)  
( (v,w) , (w,\xi) )  & = \left(  \mathrm{pr}_1^! \psi \circ (c_{s,\mathrm{pr}_1} \circ c_{t,\mathrm{pr}_2}^{-1})   \right) 
( (v,w) , (w, \tilde \psi (w, \xi) ) ) \\
& = (\mathrm{pr}_1 ^! \psi) ( (v,w) , (v, \tilde \psi (w,\xi) ) ) \\
& = ( (v,w) , (v, \tilde \psi (v, \tilde \psi (w,\xi)) ) )
\end{align*}
Therefore, the cocycle condition for $\psi$ is equivalent to the condition: 
\begin{align*}
\tilde \psi ( m_\ast (v,w) , \xi)   & = \tilde \psi (v, \tilde \psi (w, \xi) ) 
\end{align*}
which is exactly the condition that the map $\tilde \psi : TG_{s_\ast} \times _a A \to A$ is an action of $TG \rightrightarrows TX$ on $A$.
\end{proof}

\begin{Proposition}
Let $(A,\psi)$ and $(A',\psi')$ be Lie algebroids over Lie groupoids $G\rightrightarrows X$ and $G'\rightrightarrows X'$ respectively, and let $\tilde \psi$ and $\tilde \psi'$ denote the corresponding actions of $TG$ and $TG'$ as determined by Proposition \ref{tangent actions}. Let $(\phi,f,\rho)$ be a triple of smooth maps such that:
$$\xymatrix{
G \ar@<-.5ex>[d] \ar@<.5ex>[d] \ar[r]^\phi  & G'  \ar@<-.5ex>[d] \ar@<.5ex>[d] \\
X \ar[r]_f  & X'
}$$
is a morphism of Lie groupoids and
$$\xymatrix{
A \ar[d] \ar[r]^\rho  & A'  \ar[d] \\
X \ar[r]_f & X' 
}$$
is a morphism of Lie algebroids. Then the diagram:
$$
\xymatrix{
s^! A  \ar[d]_{\psi}   \ar[r]^-{(\phi_\ast,\rho)} & (s')^! A' \ar[d]^{\psi'}\\
t^! A \ar[r]_{(\phi_\ast,\rho)} & (t')^! A'
}$$
commutes, and therefore $(\phi,f,\rho)$ is a morphism in the sense of Definition \ref{!-morphisms}, if and only if $\rho$ is equivariant with respect to the actions of $TG$ and $TG'$ on $A$ and $A'$ respectively and the groupoid morphism:
$$\xymatrix{
TG \ar@<-.5ex>[d] \ar@<.5ex>[d] \ar[r]^{\phi_\ast}  & TG'  \ar@<-.5ex>[d] \ar@<.5ex>[d] \\
TX \ar[r]_{f_\ast}  & TX'
}$$
\end{Proposition}

\begin{proof}
This follows immediately if we consider $A$ and $A'$ as spaces over $TX$ and $TX'$ respectively, so that the diagram
$$\xymatrix{
s^! A  \ar[d]_{\psi}   \ar[r]^-{(\phi_\ast,\rho)} & (s')^! A' \ar[d]^{\psi'}\\
t^! A \ar[r]_{(\phi_\ast,\rho)} & (t')^! A'
}$$
becomes
$$\xymatrix{
(s_\ast)^\ast A  \ar[d]_{\psi}   \ar[r]^-{(\phi_\ast,\rho)} & (s_\ast)^\ast A' \ar[d]^{\psi'}\\
(t_\ast)^\ast A \ar[r]_{(\phi_\ast,\rho)} & (t_\ast)^\ast A'
}$$
the commutativity of which is exactly the statement that $\rho$ is equivariant.

Alternatively, let $(v,\xi) \in s^! A$, then we compute:
\begin{align*}
\left( (\phi_\ast,\rho) \circ \psi  \right) (v,\xi) & = (\phi_\ast,\rho) (v,\tilde \psi (v,\xi) ) \\
& = (\phi_\ast (v) , (\rho \circ \tilde\psi) (v,\xi) ) \\
\\
\left(\psi' \circ (\phi_\ast,\rho) \right) (v,\xi) & = \psi' ( \phi_\ast(v) , \rho(\xi) ) \\
& = ( \phi_\ast(v) , \tilde \psi' (\phi_\ast(v), \rho(\xi) ) )
\end{align*}
Therefore we have the required commutativity if and only if
$$(\rho \circ \tilde\psi) (v,\xi)  = \tilde \psi' (\phi_\ast(v), \rho(\xi) )$$
which is exactly the condition that $\rho$ is equivariant.
\end{proof}

\subsection{$\mathcal{LA}$-groupoids}  \label{Lie algebroid groupoids section}

The previous proposition shows that given a morphism $(\phi,f,\rho)$ of Lie algebroids over Lie groupoids:
$$(\phi,f,\rho):(G\rightrightarrows X,A,\psi) \to (G'\rightrightarrows X',A',\psi')$$
there is a groupoid morphism $((\phi_\ast,\rho),\rho)$:
$$\xymatrix{
TG\ltimes A  \ar@<-.5ex>[d] \ar@<+.5ex>[d]  \ar[r] ^{(\phi_\ast,\rho)} & TG' \ltimes A'  \ar@<-.5ex>[d] \ar@<+.5ex>[d]  \\
A  \ar[r] _ \rho &  A' 
}$$
where the groupoids $TG \ltimes A$ and $TG'\ltimes A'$ are the action groupoids associated to the actions of $TG$ and $TG'$ on $A$ and $A'$ respectively. Since $s^! A = TG \ltimes A$ as manifolds, the manifold $s^!A$ has both the structure of a Lie groupoid over $A$ and of a Lie algebroid over $G$, and $(\phi_\ast,\rho)$ is a Lie algebroid morphism covering $\phi$  as described in Remark \ref{square morphism}.
We therefore have a commutative diagram:
$$\xymatrix{
& & (s')^! A' \ar[d]  \ar@<-.5ex>[r] \ar@<+.5ex>[r] & A' \ar[d] \\
s^! A \ar[d] \ar[urr]^{(\phi_\ast,\rho)} \ar@<-.5ex>[r] \ar@<+.5ex>[r]  & A \ar[d] \ar[urr]^<<<<<<<<<<<<<{\rho} & G' \ar@<-.5ex>[r] \ar@<+.5ex>[r]  & X' \\
G \ar@<-.5ex>[r] \ar@<+.5ex>[r] \ar[urr]^>>>>>>>>>>>{\phi} & X  \ar[urr]_{f}\\
}$$
each face of which is either a morphism of Lie groupoids or Lie algebroids.  We'll show that the squares $(s^!A, G, A, X)$ and $((s')^!A', G', A', X')$ form Lie algebroid groupoids, i.e. groupoid objects in the category of Lie algebroids. Furthermore, there is a fully-faithful embedding of the category $\text{LieGpd} \ltimes \mathcal{LA}$ of Lie algebroids over groupoids, into the category $\mathbf{Gpd} \mathcal{LA}$ of Lie algebroid groupoids. This embedding gives rise to an equivalence between $\text{LieGpd}\ltimes\mathcal{LA}$ and a certain subcategory $\text{!-}\mathbf{Gpd}\mathcal{LA}$ of $\mathbf{Gpd}\mathcal{LA}$ which we will define.

\subsubsection{$\mathcal{LA}$-groupoids}  \label{subsubsection of LA-groupoids}

For the general theory of $\mathcal{LA}$-groupoids see \cite{Ma1}.

\begin{Definition}
A Lie algebroid groupoid, or $\mathcal{LA}$-groupoid for short, is a groupoid object in the category of Lie algebroids. More explicitly, an $\mathcal{LA}$-groupoid is a square:
$$\xymatrix{
\Omega \ar[d] \ar@<-.5ex>[r]\ar@<.5ex>[r] &  A \ar[d]  \\
G \ar@<-.5ex>[r]\ar@<.5ex>[r] &  X 
}$$
where $\Omega \rightrightarrows A$ and $G\rightrightarrows X$ are Lie groupoids, $\Omega$ and $A$ are Lie algebroids over $G$ and $X$ respectively, and the groupoid structure maps of $\Omega$ are Lie algebroid morphisms covering the corresponding structure maps of $G$. If we denote the structure maps of $\Omega$ by $(\tilde s,\tilde t, \tilde m, \tilde u, \tilde i)$ and those of $G$ by $(s,t,m,u,i)$, then the domain of the multiplication map $\tilde m$ is the fibre product Lie algebroid:
$$
\xymatrixcolsep{2.5pc}
\xymatrix{
\Omega _{\tilde s} \times _{\tilde t} \Omega \ar[d]_{\mathrm{pr}_1}  \ar[r]^-{\mathrm{pr}_2} &   \Omega \ar[d]^{\tilde t} \\
\Omega \ar[r] _{\tilde s}  & A 
}$$ 
which is a Lie algebroid over $G_s \times _t G$. This fibre product exists because $\tilde s,\tilde t, s$ and $t$ are all surjective submersions, so that $(\tilde s,s)$ and $(\tilde t,t)$ are Lie algebroid fibrations. We'll usually denote an $\mathcal{LA}$-groupoid by $(\Omega,G,A,X)$, with the structure maps understood. (Note that we do not require that the `double source map' is a submersion, as defined in \cite{Ma1}).
\end{Definition}

\begin{Remark}
If $(\Omega,G,A,X)$ is an $\mathcal{LA}$-groupoid then one can show that the vector bundle projections $\Omega \to G$ and $A \to X$ form a groupoid morphism from $\Omega \rightrightarrows A$ to $G\rightrightarrows X$, and the anchor maps $\Omega \to TG$ and $A \to TX$ form a groupoid morphism from $\Omega \rightrightarrows A$ to $TG\rightrightarrows TX$ (see \cite{Ma1}).
\end{Remark}

\begin{Definition}
A morphism of $\mathcal{LA}$-groupoids from $(\Omega,G,A,X)$ to $(\Omega',G',A',X')$ is a morphism of groupoid objects in $\mathcal{LA}$, or more explicitly a quadruple of smooth maps $(\Phi,\phi,\rho,f)$:
$$\xymatrix{
& & \Omega' \ar[d]  \ar@<.5ex>[r]\ar@<-.5ex>[r] & A' \ar[d] \\
\Omega \ar[d] \ar[urr]^{\Phi} \ar@<.5ex>[r]\ar@<-.5ex>[r]  & A \ar[d] \ar[urr]^<<<<<<<<<<{\rho} & G' \ar@<.5ex>[r]\ar@<-.5ex>[r]  & X' \\
G \ar@<.5ex>[r]\ar@<-.5ex>[r] \ar[urr]^>>>>>>>>>{\phi} & X  \ar[urr]_{f}\\
}$$
such that $(\Phi,\rho)$ and $(\phi,f)$ are Lie groupoid morphisms, and $(\Phi,\phi)$ and $(\rho,f)$ are Lie algebroid morphisms. We'll denote such a morphism by:
$$(\Phi,\phi,\rho,f):(\Omega,G,A,X) \to (\Omega',G',A',X')$$
These morphisms can be composed in the obvious way, and we denote the resulting category by $\mathbf{Gpd}\mathcal{LA}$. 
\end{Definition}

Given any Lie groupoid $G\rightrightarrows X$, the tangent groupoid $TG\rightrightarrows TX$ forms an $\mathcal{LA}$-groupoid $(TG,G,TX,X)$ and for any $\mathcal{LA}$-groupoid $(\Omega,G,A,X)$ the anchor maps form a morphism of $\mathcal{LA}$-groupoids:
$$\xymatrix{
& & TG \ar[d]  \ar@<.5ex>[r]\ar@<-.5ex>[r] & TX \ar[d] \\
\Omega \ar[d] \ar[urr]^{\tilde a} \ar@<.5ex>[r]\ar@<-.5ex>[r]  & A \ar[d] \ar[urr]^<<<<<<<<<<<{a} & G \ar@<.5ex>[r]\ar@<-.5ex>[r]  & X \\
G \ar@<.5ex>[r]\ar@<-.5ex>[r] \ar[urr]^>>>>>>>>>>{\mathrm{id}} & X  \ar[urr]_{\mathrm{id}}\\
}$$

\subsubsection{$\mathcal{LA}$-groupoids from actions}

We will show how to construct an $\mathcal{LA}$-groupoid from a Lie  algebroid $(A,\psi)$ over a Lie groupoid $G\rightrightarrows X$. The key point, mentioned above, is that as manifolds:
$$s^! A = TG _{s_\ast} \times _a A $$
which is the total space of the action groupoid: 
$$TG \ltimes A \rightrightarrows TX$$
determined by the action of $TG \rightrightarrows TX$ on $A$. This means that $s^! A$ has both the structure of a Lie algebroid over $G$, and of a Lie groupoid over $A$.

\begin{Theorem} \label{s! LA groupoids}
Let $G\rightrightarrows X$ be a Lie groupoid and
$A$ be a Lie algebroid over $X$, with anchor $a:A \to TX$ and vector bundle projection $\pi:A \to X$. Denote the anchor and bundle projection of $s^! A$ by $\tilde a:s^! \to TG$ and $\tilde \pi:s^! A \to G$ respectively. 
Let $\tilde \psi : s^! A \to A$ be a compatible action of $TG\rightrightarrows TX$ on $A$. Then the square: 
$$\xymatrix{  
s^! A \ar[d] _{\tilde \pi}  \ar@<-.5ex>[r]\ar@<+.5ex>[r]^{\tilde s, \tilde t}  &  A \ar[d] ^\pi \\
G \ar@<.5ex>[r]\ar@<-.5ex>[r] _{s,t} & X
}$$
is an $\mathcal{LA}$ groupoid, where $s^! A \rightrightarrows A$ is the action groupoid for the action of $TG$ on $A$.
\end{Theorem}

\begin{proof}
We need to prove that the groupoid structure maps of $s^! A \rightrightarrows A$ are Lie algebroid morphisms covering the corresponding structure maps of $G \rightrightarrows X$. Denote the structure maps of the action groupoid $TG \ltimes A \rightrightarrows A$ by $\tilde s, \tilde t, \tilde m, \tilde u$ and $\tilde i$. As  manifolds we have: 
$$TG \ltimes A = TG_{s_\ast} \times _a A = s^! A$$
The structure maps of $TG\ltimes A$ are:
\begin{align*}
\tilde s (v,\xi) & = \xi \\
\tilde t (v,\xi) & = \tilde \psi (v,\xi) \\
\tilde m \left( (v,\tilde \psi(v',\xi)) , (v', \xi) \right)  & = \left( m_\ast(v,v') , \xi\right) \\
\tilde u (\xi)  & = (u_\ast (a(\xi)) , \xi)  \\
\tilde i (v,\xi)  & = \left(  i_\ast (v) , \tilde \psi(v,\xi)  \right) 
\end{align*}
Firstly, $\tilde s$ is the projection $s^! A \to A$ so it is a morphism covering $s$. Secondly, $\tilde t$ is the map defining the action of $TG$ on $A$, which is a morphism covering $t$ by the definition of compatible actions (Definition \ref{compatible actions}). In order to show that the multiplication maps $(\tilde m ,m)$ constitute a Lie algebroid morphism we'll need to describe the Lie algebroid structure on $(TG\ltimes A)_2 = (s^! A)_{\tilde s} \times _{\tilde t} (s^! A)$. We will do this using the universal property of fibre products of Lie algebroids. We define two maps:
\begin{align*}
\chi_1 : m^! s^! A & \to s^! A \\
( (v,v') , (m_\ast(v,v'),\xi) )  & \mapsto (v, \tilde \psi (v',\xi) )  \\
\\
\chi_2 : m^! s^! A & \to s^! A  \\
( (v,v') , (m_\ast(v,v'),\xi) )  & \mapsto (v',\xi) 
\end{align*}
Both $\chi_1$ and $\chi_2$ are Lie algebroid morphisms because they can be written as compositions of certain morphisms:
\begin{align*}
\chi_1 & = \psi^{-1} \circ \mathrm{pr}_{t^! A} \circ c_{t,\mathrm{pr}_1} 
\circ c_{t,m}^{-1} \circ m^! \psi  \\
\chi_2 & = \mathrm{pr}_{s^! A} \circ c_{s,\mathrm{pr}_1} \circ c_{s,m}^{-1} 
\end{align*}
The maps $\chi_1$ and $\chi_2$ cover $\mathrm{pr}_1 : G_2 \to G$ and $\mathrm{pr}_2 : G_2 \to G$ respectively. Since $\tilde s \circ \chi_1 = \tilde t \circ \chi_2$ we have a commutative diagram:
$$\xymatrix{
m^! s^! A  \ar[d]_{\chi_1} \ar[r] ^ {\chi_2}  &  s^! A \ar[d]^{\tilde t}  \\
s^! A \ar[r]_{\tilde s} & A 
}$$
The universal property of fibre products then determines a canonical base preserving morphism of Lie algebroids $(\chi_1,\chi_2)$:
\begin{align*}
(\chi_1,\chi_2) : m^! s^! A & \to (s^! A)_{\tilde s} \times _{\tilde t} (s^! A) \\
( (v,v') , (m_\ast(v,v') ,\xi ) )  & \mapsto \left( \chi_1 ( (v,v') , (m_\ast(v,v') ,\xi ) ), \chi_2 ( (v,v') , (m_\ast(v,v') ,\xi ) ) \right)  \\
& = \left( (v,\tilde \psi(v',\xi)) , (v', \xi) \right) 
\end{align*}
The fibres over $(g,h) \in G_2$ of the vector bundles $m^!s^!A$ and $(s^! A)_{\tilde s} \times _{\tilde t} (s^! A)$ are both canonically isomorphic to the vector space $(T_g G)_{s_\ast} \times_{t_\ast} (T_hG)_{s_\ast} \times_a A$. Moreover, $(\chi_1,\chi_2)$ respects these isomorphisms which implies that it is a fibrewise isomorphism and therefore an isomorphism of Lie algebroids. 
Composing $\tilde m$ with $(\chi_1,\chi_2)$ gives:
\begin{align*}
(\tilde m \circ (\chi_1,\chi_2) ) : m^! s^! A & \to s^! A \\
( (v,v') , (m_\ast(v,v') ,\xi ) ) & \mapsto \left( m_\ast(v,v') , \xi\right) 
\end{align*}
which is the pullback morphism $m^!s^! A \to s^! A$. As $(\chi_1,\chi_2)$ is an isomorphism this implies that $\tilde m$ is a Lie algebroid morphism covering $m$.

It remains to show that the unit and inverse maps constitute Lie algebroid morphisms. Consider the identity morphism $A \to A$ covering the identity map on $X$. We can factor $\mathrm{id}_X$ as $\mathrm{id}_X = s \circ u$. Therefore, by the universal property of the pullback $s^! A$, $\mathrm{id}_A$ factors as $\mathrm{id}_A = \mathrm{pr}_A \circ \sigma$ for some unique morphism $\sigma:A \to s^! A$ covering $u$. The morphism $\sigma$ is given by:
\begin{align*}
\sigma : A & \to s^! A \\
\xi & \mapsto ( u_\ast (a(\xi)) , \xi)  \\
& = \tilde u (\xi)
\end{align*}
and therefore $\tilde u$ is a Lie algebroid morphism covering $u$. Similarly, the target map $\tilde t: s^! A \to A$ is a morphism covering $t$, and $t$ can be factored as $t=s \circ i$. Therefore $\tilde t$ factors into $\mathrm{pr}_{A} \circ \sigma' = \tilde s \circ \sigma'$ for some unique morphism $\sigma': s^! A \to s^! A$ covering $i$. The morphism $\sigma'$ is given by:
\begin{align*}
\sigma': s^! A & \to s^! A \\
(v,\xi) & \mapsto ( i_\ast (v,\xi), \tilde t (v,\xi) ) \\
& = ( i_\ast (v) , \tilde \psi (v,\xi))  \\
& = \tilde i (v,\xi)
\end{align*}
and therefore $\tilde i$ is a Lie algebroid morphism covering $i$. 
\end{proof}

\subsubsection{Functoriality}

The construction of an $\mathcal{LA}$-groupoid from a Lie algebroid over a Lie groupoid is functorial with respect to the morphisms of Definition \ref{!-morphisms}:

\begin{Proposition} \label{LieGpd-LA to GpdLA functor}
There is a fully-faithful functor:
$$\mathcal F _1 : \mathrm{LieGpd} \ltimes \mathcal{LA} \to \mathbf{Gpd}\mathcal{LA}$$
which maps an object $(G\rightrightarrows X,A,\psi)$ in $\mathrm{LieGpd}\ltimes\mathcal{LA}$ to the $\mathcal{LA}$-groupoid $( s^! A, G , A, X)$:
$$\xymatrix{
s^! A  \ar[d]  \ar@<-.5ex>[r] \ar@<+.5ex>[r] &  A \ar[d]  \\
G \ar@<-.5ex>[r] \ar@<+.5ex>[r] &  X  
}$$
and maps a morphism $(\phi,f,\rho):(G\rightrightarrows X,A,\psi) \to (G'\rightrightarrows X',A',\psi')$ in $\mathrm{LieGpd}\ltimes \mathcal{LA}$ to the morphism of $\mathcal{LA}$-groupoids $( (\phi_\ast,\rho) , \phi , \rho , f)$:
$$\xymatrix{
& & (s')^! A' \ar[d]  \ar@<-.5ex>[r] \ar@<+.5ex>[r] & A' \ar[d] \\
s^! A \ar[d] \ar[urr]^{(\phi_\ast,\rho)} \ar@<-.5ex>[r] \ar@<+.5ex>[r]  & A \ar[d] \ar[urr]^<<<<<<<<<<<<{\rho} & G' \ar@<-.5ex>[r] \ar@<+.5ex>[r]  & X' \\
G \ar@<-.5ex>[r] \ar@<+.5ex>[r] \ar[urr]^>>>>>>>>>>>>>{\phi} & X  \ar[urr]_{f}\\
}$$
\end{Proposition}

\begin{proof}
As discussed at the beginning of \ref{Lie algebroid groupoids section}, $( (\phi_\ast,\rho),\phi,\rho,f)$ is a morphism of $\mathcal{LA}$-groupoids, and the assignment:
$$(\phi,f,\rho) \mapsto ( (\phi_\ast,\rho),\phi,\rho,f)$$
is clearly functorial and faithful, so we need to show that it is also full. Let $(G \rightrightarrows X, A, \psi)$ and \mbox{$(G'\rightrightarrows X', A', \psi')$} be objects in $\mathrm{LieGpd}\ltimes\mathcal{LA}$, and let $(\Phi,\phi,\rho,f)$ be a morphism between the associated $\mathcal{LA}$ groupoids:
$$\xymatrix{
& & (s')^! A' \ar[d]  \ar@<-.5ex>[r] \ar@<+.5ex>[r] & A' \ar[d] \\
s^! A \ar[d] \ar[urr]^{\Phi} \ar@<-.5ex>[r] \ar@<+.5ex>[r]  & A \ar[d] \ar[urr]^<<<<<<<<<<<<{\rho} & G' \ar@<-.5ex>[r] \ar@<+.5ex>[r]  & X' \\
G \ar@<-.5ex>[r] \ar@<+.5ex>[r] \ar[urr]^>>>>>>>>>>>{\phi} & X  \ar[urr]_{f}\\
}$$
Let $(v,\xi) \in s^! A$. Since $(\Phi,\rho)$ is a morphism of groupoids it commutes with the source maps and we have:
\begin{align*}
\rho(\xi)  & =  (\rho \circ s) (v,\xi) \\
& = (s' \circ \Phi) (v,\xi) 
\end{align*}
and therefore $\Phi(v,\xi) = (v',\rho(\xi))$ for some $v' \in TX'$. But $(\Phi,\phi)$ is a morphism of Lie algebroids so it commutes with the anchor maps and we have:
\begin{align*}
\phi_\ast (v)  &  =   (\phi_\ast \circ \tilde a) (v,\xi)  \\
& = (\tilde a' \circ \Phi)  (v,\xi)
\end{align*}
and therefore $\Phi(v,\xi)=(\phi_\ast (v) , \xi')$ for some $\xi' \in A'$. Combining these, we have $\Phi = (\phi_\ast,\rho)$, which shows that 
$(\Phi,\phi,\rho,f) = ((\phi_\ast,\rho),\phi,\rho,f)$ is in the image of the functor.
\end{proof}

\subsubsection{!-vacant $\mathcal{LA}$-groupoids}  \label{!-vacant LA-groupoids}

The functor $\mathcal F_1$ of Proposition \ref{LieGpd-LA to GpdLA functor} embeds the category $\text{LieGpd}\ltimes\mathcal{LA}$ into $\mathbf{Gpd}\mathcal{LA}$. We will describe the $\mathcal{LA}$ groupoids which arise this way. 

\begin{Definition} \label{!-vacant definition}
An $\mathcal{LA}$ groupoid $(\Omega,G,A,X)$:
$$\xymatrix{  
\Omega \ar[d] _{\tilde \pi}  \ar@<-.5ex>[r]\ar@<.5ex>[r]^{\tilde s, \tilde t}  &  A \ar[d] ^\pi \\
G  \ar@<.5ex>[r]\ar@<-.5ex>[r] _{s,t} & X
}$$
is \emph{!-vacant} if the map:
\begin{align*}
(\tilde a, \tilde s):\Omega & \to s^! A \\
\omega & \mapsto (\tilde a(\omega), \tilde s(\omega) )
\end{align*}
is an isomorphism of Lie algebroids over $G$. The category $\text{!-}\mathbf{Gpd}\mathcal{LA}$ is the full subcategory of $\mathbf{Gpd}\mathcal{LA}$ consisting of !-vacant $\mathcal{LA}$-groupoids and morphisms between them.
\end{Definition}

\begin{Remark}
The condition that the map $(\tilde a, \tilde s):\Omega  \to s^! A$ is an isomorphism is exactly the condition that the Lie algebroid morphism $(\tilde s,s)$ is a pullback morphism in the category of Lie algebroids, or an `inductor' in the terminology of \cite{Ma2}. Compare this with the notion of a vacant $\mathcal{LA}$-groupoid, given in \cite{Ma1}, which is the condition that the map:
\begin{align*}
(\tilde \pi, \tilde s):\Omega & \to s^\ast A \\
\omega & \mapsto (\tilde \pi(\omega), \tilde s(\omega) )
\end{align*}
is an isomorphism, or equivalently, that $(\tilde s,s)$ is a pullback morphism in the category of vector bundles, or an `action morphism' of Lie algebroids.
\end{Remark}

By definition an $\mathcal{LA}$-groupoid of the form $(s^!A,G,A,X)$, where $A$ is a Lie algebroid over $G\rightrightarrows X$, is !-vacant. Therefore we can consider the functor $\mathcal F_1$ of Proposition \ref{LieGpd-LA to GpdLA functor} as a functor
$$\mathcal F_1 : \text{LieGpd}\ltimes\mathcal{LA} \to \text{!-}\mathbf{Gpd}\mathcal{LA}$$
We'll construct a quasi-inverse to this functor by associating to each !-vacant $\mathcal{LA}$-groupoid a certain action morphism of Lie groupoids.

\subsubsection{From !-vacant $\mathcal{LA}$-groupoids to actions}

\begin{Definition}
A morphism of Lie groupoids: 
$$\xymatrix{
G' \ar@<.5ex>[d]\ar@<-.5ex>[d]_{s', t'} \ar[r]^{\phi} & G \ar@<-.5ex>[d]\ar@<.5ex>[d]^{s,t} \\
X' \ar[r]_f & X 
}$$
is an \emph{action morphism} if the induced map:
\begin{align*}
(\phi,s'): G' & \to f^\ast G  \\
g' & \mapsto (\phi(g),s'(g))
\end{align*}
is a diffeomorphism.
\end{Definition}

\begin{Theorem}\label{action morphisms} \emph{(\cite{Ma2}).}
If $(\phi,f):G' \to G$ is an action morphism then the map:
\begin{align*}
t' \circ (\phi,s')^{-1}: G _ s \times _f X' & \to X'
\end{align*}
defines an action of $G$ on $X'$ along $f$. The map $(\phi,s')$ is a base preserving isomorphism of Lie groupoids from $G'$ to the action groupoid $G\ltimes X'$:
$$
\xymatrixcolsep{2.5pc}
\xymatrix{
G' \ar@<.5ex>[d]\ar@<-.5ex>[d]_{   } \ar[r]^-{(\phi,s')} & G \ltimes X' \ar@<.5ex>[d]\ar@<-.5ex>[d]^{  } \\
X' \ar[r]_{\mathrm{id}_{X'}} & X' 
}$$
Conversely, given an action of a Lie groupoid $G\rightrightarrows X$ on a manifold $X'$ along a map $f:X' \to X$, the map:
$$
\xymatrixcolsep{2.5pc}
\xymatrix{
G \ltimes X' \ar@<.5ex>[d]\ar@<-.5ex>[d]_{   } \ar[r]^-{\mathrm{pr}_G} & G \ar@<.5ex>[d]\ar@<-.5ex>[d]^{  } \\
X' \ar[r]_f & X 
}$$
is an action morphism and the induced action of $G$ on $X'$ is the original one.  
\end{Theorem}

We can give an alternate characterisation of !-vacant $\mathcal{LA}$-groupoids in terms of action morphisms. Given an $\mathcal{LA}$-groupoid $(\Omega,G,A,X)$ the anchor maps of $\Omega$ and $A$ form a morphism of Lie groupoids:
$$\xymatrix{
\Omega \ar@<.5ex>[d]\ar@<-.5ex>[d]_{\tilde s,\tilde t} \ar[r]^{\tilde a} & TG \ar@<-.5ex>[d]\ar@<.5ex>[d]^{s_\ast,t_\ast} \\
A \ar[r]_a & TX 
}$$
We immediately have:
\begin{Proposition}
An $\mathcal{LA}$-groupoid is !-vacant if and only if the Lie groupoid morphism $(\tilde a,a)$ is an action morphism.
\end{Proposition}

Given a !-vacant $\mathcal{LA}$-groupoid $(\Omega,G,A,X)$, Theorem \ref{action morphisms} then determines an action of $TG\rightrightarrows TX$ on $A$:

\begin{Proposition} \label{TG action from !vacancy}
Let $(\Omega,G,A,X)$ be a !-vacant $\mathcal{LA}$-groupoid. Then the map:
\begin{align*}
\tilde t \circ (\tilde a, \tilde s)^{-1}  : s^! A & \to A 
\end{align*}
defines an action of $TG \rightrightarrows TX$ on $A$ along $a$, and is a Lie algebroid morphism covering $t$.
\end{Proposition}

\begin{proof}
Since $(\Omega,G,A,X)$ is !-vacant the groupoid morphism $(\tilde a, a)$ is an action morphism, so by Theorem \ref{action morphisms} the map $\tilde t \circ (\tilde a, \tilde s)^{-1}$ defines an action of $TG$ on $A$ along $a$. The map $(\tilde a,\tilde s)^{-1}$ is an isomorphism of Lie algebroids over $G$, and $\tilde t$ is a Lie algebroid morphism covering $t$ because $(\Omega,G,A,X)$ is an $\mathcal{LA}$-groupoid.
\end{proof}

We can use this result to construct a functor from $\text{!-}\mathbf{Gpd}\mathcal{LA}$ to $\text{LieGpd}\ltimes \mathcal{LA}$.

\begin{Theorem} \label{inverse functor}
There is a functor 
$$\mathcal F_2 : \text{\emph{!-}}\mathbf{Gpd}\mathcal{LA} \to \text{\emph{LieGpd}}\ltimes \mathcal{LA}$$
which maps a !-vacant $\mathcal{LA}$-groupoid $(\Omega,G,A,X)$ to $(G\rightrightarrows X,A,\psi)$, where 
$$\psi: s^! A \to t^! A$$
is the isomorphism associated to the action of $TG\rightrightarrows TX$ on $A$
$$\tilde t \circ (\tilde a, \tilde s)^{-1}  : s^! A  \to A$$
determined by Proposition \ref{TG action from !vacancy}, and which maps a morphism
$$(\Phi,\phi,\rho,f) : (\Omega,G,A,X) \to (\Omega',G',A',X')$$
of !-vacant $\mathcal{LA}$-groupoids to the morphism
$$(\phi,f,\rho) : (G\rightrightarrows X,A,\psi) \to (G'\rightrightarrows X',A',\psi')$$
in the category $\text{\emph{LieGpd}}\ltimes\mathcal{LA}$. This functor is quasi-inverse to the functor $\mathcal F_1$ of Proposition \ref{LieGpd-LA to GpdLA functor}. In particular, there is an equivalence of  categories: 
$$\text{\emph{LieGpd}}\ltimes \mathcal{LA}\simeq\text{\emph{!-}}\mathbf{Gpd}\mathcal{LA}$$
\end{Theorem}

\begin{proof}
We first need to show that the assignment $(\Phi,\phi,\rho,f) \mapsto (\phi,f,\rho)$ does indeed map morphisms in $\text{!-}\mathbf{Gpd}\mathcal{LA}$ to morphisms in $\text{LieGpd}\ltimes\mathcal{LA}$. Let $( \Phi , \phi, \rho, f)$ be a morphism in $\text{!-}\mathbf{Gpd}\mathcal{LA}$ from $(\Omega, G, A, X)$ to $(\Omega',G',A',X')$: 
$$\xymatrix{
& & \Omega' \ar[d]  \ar@<.5ex>[r]\ar@<-.5ex>[r] & A' \ar[d] \\
\Omega \ar[d] \ar[urr]^{\Phi} \ar@<.5ex>[r]\ar@<-.5ex>[r]  & A \ar[d] \ar[urr]^<<<<<<<<<{\rho} & G' \ar@<.5ex>[r]\ar@<-.5ex>[r]  & X' \\
G \ar@<.5ex>[r]\ar@<-.5ex>[r] \ar[urr]^>>>>>>>>>{\phi} & X  \ar[urr]_{f}\\
}$$
Then in particular, $(\phi,f): G \to G'$ is a morphism of groupoids, and $\rho: A \to A'$ is a morphism of Lie algebroids covering $f$. Therefore, to show that
$$(\phi,f,\rho) : (G \rightrightarrows X,A,\psi) \to (G'\rightrightarrows X', A',\psi')$$
is a morphism in $\text{LieGpd}\ltimes\mathcal{LA}$ we need to show that $\rho$ is equivariant with respect to the actions of $TG\rightrightarrows TX$ and $TG'\rightrightarrows TX'$ on $A$ and $A'$ respectively. The actions are given by:
\begin{align*}
\tilde \psi & = \tilde t \circ (\tilde a, \tilde s)^{-1}  \\
\tilde \psi' & = \tilde t' \circ (\tilde a', \tilde s')^{-1} 
\end{align*}
As maps from $\Omega$ to $(s')^! A'$ we have:
\begin{align*}
(\tilde a', \tilde s') \circ \Phi & = ( \tilde a' \circ \Phi  ,  \tilde s' \circ \Phi  ) \\
& = (\phi_\ast \circ \tilde a  , \rho \circ \tilde s  ) \\
& = (\phi_\ast , \rho) \circ (\tilde a , \tilde s) 
\end{align*}
where we have used the facts that $\Phi$ is a Lie algebroid morphism covering $\phi$ and $(\Phi,\rho)$ is a Lie groupoid morphism. Since $(\tilde a,\tilde s)$ and $(\tilde a', \tilde s')$ are diffeomorphisms, we then have: 
\begin{align*}
\Phi \circ (\tilde a,\tilde s)^{-1} & = (\tilde a',\tilde s')^{-1} \circ (\phi_\ast,\rho) 
\end{align*}
and therefore the upper square of the following diagram commutes:
$$\xymatrix{
s'^! A' \ar[d]_{(\tilde a' , \tilde s')^{-1}} \ar[r]^{(\phi_\ast,\rho)} & s^!A \ar[d]^{(\tilde a, \tilde s)^{-1}} \\
\Omega'  \ar[d]_{\tilde t'} \ar[r]^{\Phi} & \Omega   \ar[d] ^{\tilde t} \\
A' \ar[r]_\rho & A 
}$$
The lower square commutes because $(\Phi,\rho)$ is a groupoid morphism, and therefore we have that
$$\xymatrix{
s'^!A'  \ar[d]_{\tilde \psi'} \ar[r]^{(\phi_\ast,\rho)} & s^!  A \ar[d] ^{\tilde \psi} \\
A' \ar[r]_\rho & A 
}$$
commutes, which is the statement that $\rho$ is equivariant as required. It is clear that $(\phi,f,\rho)$ depends functorially on $(\Phi,\phi,\rho,f)$. 

It remains to show that the functors $\mathcal F_1$ and $\mathcal F_2$ are quasi-inverse to each other. Let $(G\rightrightarrows X, A,\psi)$ be an object of $\text{LieGpd}\ltimes \mathcal{LA}$, then
$$(\mathcal F_2 \circ \mathcal F_1)(G\rightrightarrows X, A,\psi) = (G\rightrightarrows X, A,\psi')$$
where $\psi'$ is the action of $TG\rightrightarrows TX$ on $A$ determined by the !-vacant $\mathcal{LA}$ groupoid $(s^! A, G,A,X)$. However, by Theorem \ref{action morphisms} this action coincides with the action $\psi$. Explicitly, we have:
\begin{align*}
\psi' = \tilde t \circ (\tilde a, \tilde s)^{-1} & = \tilde t \circ \mathrm{id}_{s^! A} = \tilde t = \psi
\end{align*}
Therefore $(\mathcal F_2 \circ \mathcal F_1)$ is the identity on objects, and it is also the identity on morphisms:
$$(\mathcal F_2 \circ \mathcal F_1)(\phi,f,\rho)  = \mathcal F_2( (\phi_\ast,\rho),\phi,\rho,f) ) = (\phi,f,\rho)$$
and so 
$$(\mathcal F_2 \circ \mathcal F_1) = \mathrm{id}_{\text{LieGpd}\ltimes \mathcal{LA}}$$
Now let $(\Omega,G,A,X)$ be an object of $\text{!-}\mathbf{Gpd}\mathcal{LA}$, then
$$(\mathcal F_1 \circ \mathcal F_2) (\Omega,G,A,X) = (s^! A, G,A,X)$$
with $\mathcal{LA}$-groupoid structure determined by Theorem \ref{s! LA groupoids}. Since $(\Omega,G,A,X)$ is !-vacant we have an isomorphism of $\mathcal{LA}$-groupoids:
\begin{align*}
((\tilde a,\tilde s),\mathrm{id}_G, \mathrm{id}_A,\mathrm{id}_X) : (\Omega,G,A,X) & \to   (s^! A,G,A,X) 
\end{align*}
$$\xymatrix{
& & s^! A \ar[d]  \ar@<-.5ex>[r]\ar@<0.5ex>[r] & A \ar[d] \\
\Omega \ar[d] \ar[urr]^{(\tilde a, \tilde s)} \ar@<-.5ex>[r]\ar@<0.5ex>[r]  & A \ar[d] \ar[urr]^<<<<<<<{\mathrm{id}_A} & G \ar@<-.5ex>[r]\ar@<0.5ex>[r]  & X \\
G \ar@<-.5ex>[r]\ar@<0.5ex>[r] \ar[urr]^>>>>>>>>{\mathrm{id}_G} & X  \ar[urr]_{\mathrm{id}_X}\\
}$$
This is indeed a morphism, because 
$$\xymatrix{
\Omega   \ar[d] \ar[r]^{(\tilde a,\tilde s)} &  s^! A   \ar[d]  \\
G \ar[r] _{\mathrm{id}_G}  & G 
}$$
is the morphism of Lie algebroids induced by $\tilde s : \Omega \to A$, and 
$$\xymatrix{
\Omega   \ar@<-.5ex>[d]\ar@<.5ex>[d] \ar[r]^{(\tilde a,\tilde s)} &  s^! A   \ar@<-.5ex>[d]\ar@<.5ex>[d] \\
A \ar[r] _{\mathrm{id}_A}  & A 
}$$
a morphism of Lie groupoids by Theorem \ref{action morphisms}. Lastly we need to check that these isomorphisms are natural. Let $$(\Phi,\phi,\rho,f):(\Omega,G,A,X) \to (\Omega',G',A',X')$$ be a morphism in $\text{!-}\mathbf{Gpd}\mathcal{LA}$, then the naturality square is: 
$$\xymatrix{
(\Omega,G,A,X)  \ar[d]_{ ((\tilde a,\tilde s),\mathrm{id}_G, \mathrm{id}_A,\mathrm{id}_X)}  \ar[rr]^{(\Phi,\phi,\rho,f)} &&  (\Omega',G',A',X') \ar[d]^{((\tilde a',\tilde s'),\mathrm{id}_G, \mathrm{id}_A,\mathrm{id}_X)}  \\
(s^! A,G,A,X)  \ar[rr]_-{((\phi_\ast,\rho),\phi,\rho,f)} && ((s')^! A',G',A',X') 
}$$
This commutes because $(\tilde a',\tilde s') \circ \Phi = (\phi_\ast,\rho)\circ (\tilde a,\tilde s)$, as shown above. 
\end{proof}

\subsubsection{Lie algebroids over a fixed base groupoid}

In a natural way we can define Lie algebroids and $\mathcal{LA}$-groupoids over a fixed base groupoid, and consider base preserving morphims:

\begin{Definition}
Let $G\rightrightarrows X$ be a Lie groupoid. Then $(\text{LieGpd}\ltimes\mathcal{LA})_{G\rightrightarrows X}$ is the subcategory of $\text{LieGpd}\ltimes\mathcal{LA}$ consisting of objects of the form $(G\rightrightarrows X,A,\psi)$, for varying $A$ and $\psi$, and morphisms of the form $(\mathrm{id}_G,\mathrm{id}_X,\rho)$. The subcategory $(\text{!-}\mathbf{Gpd}\mathcal{LA})_{G\rightrightarrows X}$ of $\text{!-}\mathbf{Gpd}\mathcal{LA}$ consists of !-vacant $\mathcal{LA}$-groupoids of the form $(\Omega,G,A,X)$, and morphisms of the form $(\Phi,\mathrm{id}_G,\psi,\mathrm{id}_X)$.
\end{Definition}

\begin{Theorem} \label{LA groupoid theorem}
For any Lie groupoid $G\rightrightarrows X$ the functors $\mathcal F_1$ and $\mathcal F_2$, defined in Proposition \ref{LieGpd-LA to GpdLA functor} and Theorem \ref{inverse functor} respectively, restrict to an equivalence of categories:
$$(\text{\emph{LieGpd}}\ltimes\mathcal{LA})_{G\rightrightarrows X}  \simeq (\text{\emph{!-}}\mathbf{Gpd}\mathcal{LA})_{G\rightrightarrows X}$$
\end{Theorem}

\begin{proof}
The proof is essentially the same as that of Theorem \ref{inverse functor}.
\end{proof}

\subsubsection{Vacant vs !-vacant $\mathcal{LA}$-groupoids}

We will compare the notions of !-vacant and vacant. Recall from \ref{!-vacant LA-groupoids} that an $\mathcal{LA}$ groupoid $(\Omega,G,A,X)$:
$$\xymatrix{  
\Omega \ar[d] _{\tilde \pi}  \ar@<-.5ex>[r]\ar@<.5ex>[r]^{\tilde s, \tilde t}  &  A \ar[d] ^\pi \\
G  \ar@<.5ex>[r]\ar@<-.5ex>[r] _{s,t} & X
}$$
is called !-vacant if the vector bundle map $(\tilde s,s)$ is a pullback morphism in the category of Lie algebroids, and vacant if $(\tilde s,s)$ is a pullback morphism in the category of vector bundles. We then have:

\begin{Proposition}  \label{Prop on vacancy and etaleness 1}
If an $\mathcal{LA}$-groupoid $(\Omega,G,A,X)$ is both !-vacant and vacant then $G\rightrightarrows X$ is an \'etale groupoid. Conversely, if $G\rightrightarrows X$ is \'etale then $(\Omega,G,A,X)$ is !-vacant if and only if it is vacant.
\end{Proposition}

\begin{proof}
This follows immediately from the fact that $s^! A \cong s^\ast A$ if and only if $s$ is an \'etale map.
\end{proof}

\begin{Proposition} \label{Prop on vacancy and etaleness 2}
Let $(\Omega, G, A, M)$ be an $\mathcal{LA}$ groupoid and the base groupoid $G \rightrightarrows M$ be \'etale. Then $(\Omega, G, A, M)$ is vacant if and only if the groupoid $\Omega \rightrightarrows A$ is also \'etale. If $G\rightrightarrows M$ is \'etale and proper, then $(\Omega, G, A, M)$ is vacant if and only if $\Omega \rightrightarrows A$ is also \'etale and proper. 
\end{Proposition}

\begin{proof}
Let $G\rightrightarrows M$ be \'etale. If $\Omega \rightrightarrows A$ is also \'etale, then in particular we have that in the vector bundle morphism: 
$$\xymatrix{ \Omega  \ar[d] _{} \ar[r] ^{\tilde s}  & A \ar[d] \\
                G \ar[r]_s & M}
$$
both $s$ and $\tilde s$ are \'etale, so $\tilde s$ must be a fibrewise isomorphism, which is equivalent to $(\tilde s,s)$ being a pullback morphism of vector bundles. Conversely, if $(\Omega, G, A, M)$ is vacant then the fibres of $\Omega$ and $A$ (as vector bundles over $G$ and $M$ respectively) must have the same dimension, but as $G\rightrightarrows M$ is \'etale $G$ and $M$ have the same dimension, and so $\Omega$ and $A$ have equal dimensions, which implies that $\Omega \rightrightarrows A$ is \'etale. 

Now assume that $G\rightrightarrows M$ is \'etale and proper. If $\Omega \rightrightarrows A$ is also \'etale and proper then by the above $(\Omega, G, A, M)$ is vacant. Conversely, if $(\Omega, G, A, M)$ is  vacant, then $\Omega \rightrightarrows A$ is \'etale, and by Theorem 4.11 in \cite{Ma1} there is an isomorphism of Lie groupoids $\Omega \rightrightarrows A \cong A \rtimes G$ for some action of $G$ on $A$, and action groupoids of actions of proper groupoids are proper. 
\end{proof}

\subsection{Applications to Lie algebroids over stacks}

We can apply the results of \ref{Actions of Lie groupoids on Lie algebroids} and \ref{Lie algebroid groupoids section} to the study of Lie algebroids over differentiable stacks.

\subsubsection{Lie algebroids over stacks as $\mathcal{LA}$-groupoids}

Combining  Proposition \ref{Proposition on Lie algebroids via atlases} in \ref{Lie algebroids in terms of atlases} with Theorem \ref{LA groupoid theorem} above we have:

\begin{Theorem} \label{Mehta theorem}
Let $\mathfrak X$ be a differentiable stack and $X \to \mathfrak X$ be an atlas. Then there are equivalences of categories
$$ \mathcal{LA}_\mathfrak X  \simeq  \mathcal{LA}_{X \times_\mathfrak X X \rightrightarrows X} \simeq 
(\text{\emph{!-}}\mathbf{Gpd}\mathcal{LA})_{G\rightrightarrows X} $$
between the categories of Lie algebroids over $\mathfrak X$, Lie algebroids over the Lie groupoid $X \times_\mathfrak X X \rightrightarrows X$, and \emph{!-vacant} $\mathcal{LA}$-groupoids with base $X \times_\mathfrak X X \rightrightarrows X$.
\end{Theorem}

It was suggested in sections 5.4 and 7 of \cite{Meh1} that vacant $\mathcal{LA}$-groupoids could be used to define Lie algebroids over orbifolds (see \ref{Some classes of differentiable stacks} for the terminology), though as discussed there one needs to prove that this is a Morita invariant notion. Theorem \ref{Mehta theorem} shows that !-vacant $\mathcal{LA}$-groupoids do arise naturally in the description of Lie algebroids over differentiable stacks, and Propositions \ref{Prop on vacancy and etaleness 1} and \ref{Prop on vacancy and etaleness 2} above show that in the case of \'etale stacks, and in particular orbifolds, the relevant $\mathcal{LA}$-groupoids are in fact vacant.

\subsubsection{Total stacks of Lie algebroids over stacks}  \label{Total stacks of Lie algebroids over stacks}

We can use the construction of $\mathcal{LA}$-groupoids of Theorem \ref{s! LA groupoids} to construct the `total stack' of a Lie algebroid over a differentiable stack, and a morphism of stacks from this total stack to the tangent stack of $\mathfrak X$. 

Let $\mathbf A: \mathfrak X_\text{sub} \to \mathcal{LA}$ be a Lie algebroid over a differentiable stack $\mathfrak X$, and let $X \to \mathfrak X$ be an atlas. We'll write $X_1$ for $X \times _\mathfrak X X$. Then we have a Lie algebroid $A_X$ over $X$, together with an action of $X_1 \rightrightarrows X$, and Theorem \ref{s! LA groupoids} gives an $\mathcal{LA}$-groupoid:
$$\xymatrix{  
s^! A_X \ar[d] _{\tilde \pi}  \ar@<-.5ex>[r]\ar@<+.5ex>[r]^{\tilde s, \tilde t}  &  A \ar[d] ^\pi \\
X_1  \ar@<.5ex>[r] \ar@<-.5ex>[r] _{s,t} & X
}$$
In particular, we have a Lie groupoid $s^! A_X \rightrightarrows A$, and a morphism $(\tilde \pi,\pi)$ of Lie groupoids from $s^! A_X \rightrightarrows A$ to $X_1 \rightrightarrows X$.
Recall from \ref{subsubsection of LA-groupoids} that the anchor maps $\tilde a$ and $a$ of $s^! A$ and $A$ give a morphism of $\mathcal{LA}$ groupoids:
$$\xymatrix{
& & TX_1 \ar[d]  \ar@<.5ex>[r]\ar@<-.5ex>[r] & TX \ar[d] \\
s^! A_X \ar[d] \ar[urr]^{\tilde a} \ar@<.5ex>[r]\ar@<-.5ex>[r]  & A_X \ar[d] \ar[urr]^<<<<<<<<<<{a} & X_1 \ar@<.5ex>[r]\ar@<-.5ex>[r]  & X \\
X_1 \ar@<.5ex>[r]\ar@<-.5ex>[r] \ar[urr]^>>>>>>>>>{\mathrm{id}} & X  \ar[urr]_{\mathrm{id}}\\
}$$
If we pass from Lie groupoids to their classifying stacks we therefore have a (strictly) commutative diagram of morphisms of differentiable stacks:
$$\xymatrix{
[s^! A_X \rightrightarrows X]  \ar[dr]  \ar[r] ^{\mathbf a} &  [TX_1 \rightrightarrows TX]  \ar[d]  \\
&  [X_1 \rightrightarrows X]
}$$
where $\mathbf a$ is the morphism of stacks induced by the anchor maps. Let $\mathfrak A = [s^! A \rightrightarrows A]$, identify $[X_1 \rightrightarrows X]$ with $\mathfrak X$ (they are equivalent stacks), and recall from \cite{BeGiNoXu1} that $[TX_1 \rightrightarrows TX]$ is the tangent stack $T\mathfrak X$ of $\mathfrak X$, then we have a diagram:
$$\xymatrix{
\mathfrak A \ar[dr] \ar[r] ^ {\mathbf a} &  T\mathfrak X \ar[d]  \\
& \mathfrak X
}$$
The morphism $\mathfrak A \to \mathfrak X$ is the `stacky analogue' of the vector bundle projection of a Lie algebroid over a manifold, and $\mathbf a : \mathfrak A \to T\mathfrak X$ is the analogue of the anchor map of a Lie algebroid. In general neither of these maps will be representable. In particular, as expected from the discussion in section \ref{Lie algebroids over stacks are not vector bundles} and the case of the tangent stack, $\mathfrak A \to \mathfrak X$ is not a vector bundle in the usual sense (for example see Definition 3.1 and Proposition 3.2 of \cite{BeGiNoXu1}).

Propositions \ref{Prop on vacancy and etaleness 1} and \ref{Prop on vacancy and etaleness 2} above show that if $\mathfrak X$ is an \'etale stack and $X \to \mathfrak X$ is an \'etale atlas then $\mathfrak A$ will be an \'etale stack, and similarly if $\mathfrak X$ is an orbifold then so will $\mathfrak A$ be.

\newpage

\section{Examples and applications}  \label{Examples}

In this section we describe some examples of Lie algebroids over stacks.

\subsection{Tangent bundle Lie algebroids}

We'll define the tangent bundle Lie algebroid of an arbitrary stack and show that it has several properties analogous to those of the tangent bundle of a manifold. We then sketch a description of its representations and cohomology.

\subsubsection{The definition of $\mathbf T \mathfrak X$ for a stack $\mathfrak X$}

For any stack $\mathfrak X$ over $\mathbf{Man}$ there is a Lie algebroid $\mathbf T \mathfrak X : \mathfrak X_\text{sub} \to \mathcal{LA}$ defined as follows. For each object $(U,u)$ of $\mathfrak X_\text{sub}$ we set $$\mathbf T \mathfrak X _{U,u} =  TU$$
and for each morphism $(f,\alpha):(U,u) \to (V,v)$ we associate the canonical isomorphism
$$TU \to f^! TV$$
given by the formula
$$ v \mapsto (v, f_\ast v)$$
for $v \in TU$. 

\subsubsection{Some properties of $\mathbf T \mathfrak X$} \label{Some properties of TX}

For any manifold $U$, $TU$ is a final object in the category $\mathcal{LA}_U$ of Lie algebroids over $U$. It follows from the description of morphisms between Lie algebroids given in \ref{Lie algebroids via test spaces} that $\mathbf T \mathfrak X$ is a final object in the category $\mathcal{LA}_\mathfrak X$ of Lie algebroids over $\mathfrak X$: if $\mathbf A$ is a Lie algebroid over $\mathfrak X$ then the anchor maps 
$$a_{U,u} : A_{U,u} \to TU$$
form the unique morphism
$$ \mathbf A \to \mathbf T\mathfrak X$$

Recall from \ref{pullbacks along representable submersions} that we can pull back Lie algebroids along representable submersions. If $\phi: \mathfrak X \to \mathfrak Y$ is a representable submersion of stacks over $\mathbf{Man}$, then we have 
$$\left( \phi^! \mathbf T \mathfrak Y \right) _{U,u} = \mathbf T \mathfrak Y_{U, \phi \circ u} = TU $$
and therefore
$$\phi^! \mathbf T \mathfrak Y = \mathbf T \mathfrak X$$

\subsubsection{Representations of $\mathbf T \mathfrak X$, local systems, and representations of $\pi_1 (\mathfrak X)$}  \label{local systems...}

We'll sketch the relationship between representations of $\mathbf T \mathfrak X$, local systems, and representations of $\pi_1 (\mathfrak X)$. By a local system over a manifold we'll mean a locally constant sheaf of $\mathbb R$-vector spaces. We'll denote the category of local systems over $X$ by $\mathrm{Loc}_X$.

First, recall the relationship between local systems and flat vector bundles: if $X$ is a manifold and $(E,\nabla)$ is a vector bundle $E$ over $X$ together with a flat connection $\nabla$, then the sheaf of flat sections of $E$, denoted $\mathcal E^\nabla$, is a local system on $X$. Conversely, if $\mathcal F$ is a local system on $X$, then the sheaf $\mathrm C _X ^\infty \otimes _{\mathbb R} \mathcal F$ is a locally free sheaf of $\mathrm C_X ^\infty$-modules, so is the sheaf of sections of some vector bundle $F$, and there is a flat connection $\nabla$ on $\mathrm C _X ^\infty \otimes _{\mathbb R} \mathcal F$ defined by
$$ \nabla _v (f \otimes \xi) = v(f) \otimes \xi$$
for $v$ a local section of $TX$. These constructions give an equivalence of categories:
$$ \mathrm{Flat}_X \simeq \mathrm{Loc}_X$$
where $\mathrm{Flat}_X$ is the category of flat vector bundles over $X$.

Let $\phi:X \to Y$ be a submersion. Then flat vector bundles and local systems can be pulled back along $\phi$, giving functors
\begin{align*}
\phi^\ast: \mathrm{Flat}_Y & \to \mathrm{Flat}_X \\
\phi^\ast : \mathrm{Loc}_Y & \to \mathrm{Loc}_X
\end{align*}
and if $(E,\nabla)$ is a flat vector bundle of $Y$ then there is a canonical isomorphism of local systems
$$\phi^\ast (\mathcal E ^\nabla)  \cong \left( \phi^\ast(E,\nabla) \right) ^{\phi^\ast \nabla}$$
given by pulling back flat sections of $E$. The result is that there are weak presheaves $\mathrm{Flat}$ and $\mathrm{Loc}$ over $\mathbf{Man}_\text{sub}$, and a equivalence
$$ \mathrm{Flat} \to \mathrm{Loc}$$
Both $\mathrm{Flat}$ and $\mathrm{Loc}$ satisfy descent for submersions. 

Recall that for any manifold $X$, representations of the Lie algebroid $TX$ are exactly flat vector bundles over $X$, so in a natural way $\mathrm{Flat}$ is a substack of $\mathrm{Rep}$ (the stack of representations of Lie algebroids, see \ref{The stack Rep}). Let $\mathfrak X$ be a stack over $\mathbf{Man}$. It follows from the way we defined representations (Definition \ref{definition of representations}) that a representation $\mathbf T \mathfrak X \to \mathrm{Rep}$ takes values in the substack $\mathrm{Flat}$, and we have:
$$\mathrm{Rep}_{\mathbf T \mathfrak X} = 
\mathrm{Hom}_{\mathrm{PSh}(\mathbf{Man}_\text{sub})} (\mathfrak X _ \text{sub} , \mathrm{Flat} )
\simeq 
\mathrm{Hom}_{\mathrm{PSh}(\mathbf{Man}_\text{sub})} (\mathfrak X _ \text{sub} , \mathrm{Loc} )
= \mathrm{Loc}_\mathfrak X
$$
where the last equality we take as the definition of the category $\mathrm{Loc}_\mathfrak X$ of local systems on $\mathfrak X$. 

Now assume that $\mathfrak X$ is a differentiable stack and $X \to \mathfrak X$ is an atlas. Since $\mathrm{Loc}$ satisfies descent for submersions there is an equivalence of categories:
$$\mathrm{Loc}_\mathfrak X \simeq \mathrm{Loc}_{X \times _\mathfrak X X \rightrightarrows X}$$
where $\mathrm{Loc}_{X \times _\mathfrak X X \rightrightarrows X}$ is the category of local systems over $X$ together with a left action of the groupoid $X \times _\mathfrak X X \rightrightarrows X$. Recall that a sheaf over $X$ is a local system if and only if its \'etale space is a covering space over $X$. Then the same proof as that of Theorem 3.18 in \cite{MoMr2} shows that if $\mathfrak X$ is connected (which is equivalent to the quotient space $X / X \times _\mathfrak X X$ being connected) then there is an equivalence of categories:
$$\mathrm{Loc}_{X \times _\mathfrak X X \rightrightarrows X} \simeq \mathbb R [ \pi_1(\mathfrak X)] \text{-mod}$$
where the right hand side is the category of finite dimensional real representations of the fundamental group $\pi_1(\mathfrak X)$ of $\mathfrak X$. (In \cite{MoMr2} the fundamental group of a connected differentiable stack $\mathfrak X$ is defined to be the isotropy group of the fundamental groupoid of a Lie groupoid representing $\mathfrak X$. See \cite{MoMr2} for details, and \cite{No1},\cite{No2} for the homotopy theory of topological and differentiable stacks).  Summarising the discussion above, we have:

\begin{Proposition}
If $\mathfrak X$ is a stack over $\mathbf{Man}$ then there is an equivalence of categories 
$$\mathrm{Rep}_{\mathbf T \mathfrak X} \simeq \mathrm{Loc}_\mathfrak X$$
If $\mathfrak X$ is a connected differentiable stack then there is also an equivalence
$$ \mathrm{Loc}_\mathfrak X
\simeq \mathbb R [ \pi_1(\mathfrak X)] \text{\emph{-mod}}
$$
\end{Proposition}

\subsubsection{Cohomology of $\mathbf T \mathfrak X$}

For any manifold $U$ the de Rham complex $\Omega^\bullet (TU)$ of the Lie algebroid $TU$ is the usual de Rham complex of differential forms over $U$. Therefore the de Rham complex 
$$\Omega^\bullet _{\mathbf T \mathfrak X} \in  \mathrm D ^ +  (\mathrm{Sh}_{\mathbb R }(\mathfrak X_\text{sub}))$$
of the the tangent Lie algebroid is the de Rham complex $\Omega_\mathfrak X$ of $\mathfrak X$ as defined in \cite{BeXu1} and \cite{BuShSp1} and 
$$\mathrm H^\bullet (\mathbf T \mathfrak X) = \mathrm H ^\bullet _\mathrm {dR} (\mathfrak X) \text{,}$$ 
the de Rham cohomology of $\mathfrak X$. In \cite{BeXu1} and \cite{BuShSp1} it is shown that the de Rham complex of a differentiable stack is a resolution of the constant sheaf $\mathbb R_\mathfrak X$, and one therefore has
$$ \mathrm H ^\bullet ( \mathbf T \mathfrak X ) = \mathrm H^\bullet (\mathfrak X_\text{sub} , \mathbb R_\mathfrak X)$$
The same argument can be used in the case of cohomology with coefficients. If $(\mathbf E, \nabla)$ is a representation of $\mathbf T \mathfrak X$ then from the discussion in \ref{local systems...} we get a map $\mathfrak X_\text{sub} \to \mathrm{Loc}$. This induces a sheaf $\mathcal E^\nabla$ over the site $\mathfrak X_\text{sub}$ with the property that for any object $(U,u)$ the induced small sheaf $\mathcal E ^\nabla _{U,u}$ over $U$ is locally constant, and the complex of sheaves $\Omega^\bullet _U \otimes _{\mathrm C_U ^\infty} \mathcal  E$ is a resolution of it. We then have
$$ \mathrm H ^\bullet \left( \mathbf T \mathfrak X , (\mathbf E, \nabla)  \right) = \mathrm H^\bullet (\mathfrak X_\text{sub} , \mathcal E ^\nabla)$$

\subsubsection{$\mathbf T \mathfrak X$ and the tangent stack $T \mathfrak X$}

If $X \to \mathfrak X$ is an atlas of a differentiable stack then the $\mathcal{LA}$ groupoid corresponding to $\mathbf T \mathfrak X$ under the construction of Theorem \ref{s! LA groupoids} in section \ref{Lie algebroids over Lie groupoids} is 
$$\xymatrix{  
s^! TX \ar[d]   \ar@<-.5ex>[r]\ar@<+.5ex>[r]  &  TX \ar[d] \\
X_1 \ar@<.5ex>[r]\ar@<-.5ex>[r]  & X
}$$
which is canonically isomorphic to 
$$\xymatrix{  
TX_1 \ar[d]   \ar@<-.5ex>[r]\ar@<+.5ex>[r]  &  TX \ar[d] \\
X_1 \ar@<.5ex>[r]\ar@<-.5ex>[r]  & X
}$$
The total stack (see \ref{Total stacks of Lie algebroids over stacks}) of $\mathbf T \mathfrak X$ is therefore equal to the tangent stack $T\mathfrak X$ of $\mathfrak X$.

\subsection{Poisson structures on \'etale stacks} \label{Poisson}

Throughout this and the following section (\ref{Poisson} and \ref{Etale examples}) all stacks will be assumed to be \'etale, and we'll consider Lie algebroids, vector bundles and other structures over an \'etale stack $\mathfrak X$ as morphisms out of the \'etale site $\mathfrak X_\text{\'et}$ of $\mathfrak X$ (see \ref{Use of etale site}).  In particular, the restriction of $\mathbf T\mathfrak X$ to $\mathfrak X_\text{\'et}$ is a vector bundle in this sense.  We'll use vector bundle pullbacks instead of Lie algebroid pullbacks where its convenient, see Proposition \ref{equivalence over etale category}.

\subsubsection{Poisson structures on manifolds}  \label{Poisson structures on manifolds}

We recall the relevant results concerning Poisson structures, see \cite{Va1} and \cite{LaPiVa1} for details. If $X$ is a manifold then a Poisson bracket on $X$ is a Lie bracket
$$\{ , \}  : \mathrm C^\infty (X)  \times \mathrm C ^\infty (X) \to \mathrm C^\infty (X)$$
that satisfies the Leibniz identity
$$ \{f,gh\} =  \{f,g\}h +  g\{f,h\}$$
for all smooth functions $f,g,h$ on $X$.
There is a bijection between Poisson brackets on $X$ and bivector fields
$$\Pi \in \Gamma  \left( {\bigwedge}^2 TX \right)$$
for which
$$ \left[ \Pi , \Pi \right] = 0 $$
where $[,]$ is the Schouten-Nijenhuis bracket on $\Gamma  \left( {\bigwedge}^\bullet TX \right)$. The Poisson bracket $\{,\}_\Pi$ corresponding to such a bivector field $\Pi$ is given by
$$ \{ f, g\}_\Pi = \Pi (\mathrm df ,\mathrm dg )$$
Either of these objects is called a Poisson structure on $X$, and $(X,\{,\})$ or $(X,\Pi)$ is then called a Poisson manifold. Poisson brackets are automatically local, so that one in fact has a Lie bracket
$$\{ , \}  : \mathrm C^\infty _X  \times \mathrm C ^\infty _X \to \mathrm C^\infty _X $$
on the sheaf $\mathrm C ^\infty _X$.  If $X$ and $Y$ are Poisson manifolds then a smooth map $\phi: X \to Y$ is called a Poisson map if and only if 
$$ \phi^\ast : \mathrm C ^\infty (Y) \to \mathrm C ^\infty (X) $$
is a morphism of Lie algebras.

Associated to a Poisson structure $\Pi$ on $X$ is a Lie algebroid structure $T_\Pi ^\ast X$ on $T^\ast X$. The anchor is given by the map
\begin{align*}
\Pi^\sharp : T^\ast _\Pi X  & \to TX \\
\alpha & \mapsto \Pi(\alpha , -)
\end{align*}
and the Lie bracket is given locally by
$$ [ \mathrm df , \mathrm dg ] =  \mathrm d \{f,g\} $$
If $\phi:X \to Y$ is an \'etale map between Poisson manifolds then $\phi$ is a Poisson map if and only if the associated bundle map
$$\xymatrix{
T^\ast _{\Pi_X} X  \ar[d] \ar[r]  ^ {\phi_\ast} &  T^\ast _{\Pi_Y} Y \ar[d] \\
X \ar[r] _\phi &  Y
}$$
is a morphism of Lie algebroids over $X$. The Lie algebroid cohomology $\mathrm H ^\bullet (T_\Pi ^\ast)$ is called the Poisson cohomology of $X$.

One class of examples of Poisson manifolds arises from Lie algebroids. If $A$ is a Lie algebroid over $X$, with vector bundle projection $\pi:A \to X$, then there is a Poisson structure on the total space of the vector bundle $A^\ast$ that is determined by: 
\begin{align*}
\Big \{ \widetilde \xi ,\widetilde \nu \Big \}   &   \equiv \widetilde{  [\xi,\nu] }    \\
\{ \pi ^\ast f , \pi ^\ast g  \}  &  \equiv  0 \\
\Big \{ \widetilde \xi, \pi ^\ast f \Big \}   &  \equiv \pi ^\ast (\xi (f))
\end{align*}
where $\xi,\nu \in \Gamma(A)$, $f,g \in \mathrm C^\infty (X)$, and for any $\xi \in \Gamma(A)$, $\widetilde \xi$ is the corresponding fibrewise linear function on $A^\ast$. (This is enough to determine the Poisson bracket on $A^\ast$ because the differentials of the functions of the form $\widetilde \xi$ and $\pi^\ast f$ span the cotangent spaces of $A^\ast$). This Poisson structure is fibrewise linear, in the sense that the space of fibrewise linear functions on $A^\ast$ is closed under the bracket, and this construction gives a bijection between such Poisson structures and Lie algebroid structures on $A$.

\subsubsection{Defining Poisson structures on \'etale stacks}

It follows from the fact that a Poisson bracket on $\mathrm C^\infty (X)$ induces a Poisson bracket on the sheaf $\mathrm C^\infty _X$ that for any manifold $X$ the assignment 
$$U \mapsto \mathrm{Poiss}_U \equiv \{ \text{Poisson brackets on } \mathrm C^\infty (U) \} $$
is a sheaf over $X$. If $\phi:X \to Y$ is an \'etale map, then a Poisson structure $\Pi$ on $Y$ induces a Poisson structure $\phi^\ast \Pi$ on $X$. There is therefore a presheaf
\begin{align*}
\mathrm{Poiss} : \mathbf{Man}_\text{\'et} &  \to \mathbf{Set} \\
X & \mapsto \mathrm{Poiss}_U
\end{align*}
which is a sheaf with respect to the open cover (and therefore the \'etale) topology. If $\phi : X \to Y$ is an \'etale map between Poisson manifolds then $\phi$ is a Poisson map if and only if $\phi^\ast \Pi_Y = \Pi_X$. If $\phi: X \to Y$ is \'etale and $\Pi$ is a Poisson structure on $Y$ then the natural map 
$$ T^\ast _{\phi^\ast \Pi} X \to \phi^\ast \left(  T^\ast _\Pi Y \right) $$
is an isomorphism of Lie algebroids. This shows that there is a morphism of sheaves over $\mathbf{Man}_\text{\'et}$:
\begin{align*}
\mathbf T^\ast : \mathrm{Poiss} & \to \mathcal{LA}  \\
\mathrm{Poiss}_X \ni \Pi & \mapsto T^\ast _\Pi X 
\end{align*}
We can now define:

\begin{Definition} \label{Definition of Poisson structures on stacks}
A Poisson structure $\mathbf \Pi$ on an \'etale stack $\mathfrak X$ is a morphism 
$$\mathbf \Pi : \mathfrak X_\text{\emph{\'et}} \to \mathrm{Poiss}$$
of stacks over $\mathbf{Man}_\text{\'et}$, and a pair $(\mathfrak X, \mathbf \Pi)$ is a Poisson \'etale stack. Associated to such a Poisson structure $\mathbf \Pi$ is a Lie algebroid $\mathbf T ^\ast _{\mathbf \Pi} \mathfrak X$ defined as the composition
$$\xymatrix{
\mathfrak X_\text{\emph{\'et}} \ar[r] ^{\mathbf \Pi} &  \mathrm{Poiss}  \ar[r]^{\mathbf T^\ast} &  \mathcal{LA}
}$$
We define the Poisson cohomology of $(\mathfrak X, \mathbf \Pi)$ to be the Lie algebroid cohomology of $\mathbf T ^\ast _{\mathbf \Pi} \mathfrak X$. 
\end{Definition}

Explicitly, a morphism $\mathfrak X_\text{\'et} \to \mathrm{Poiss}$ is determined by the choice of a Poisson structure $\Pi_{U,u}$ for every object $(U,u)$ of $\mathfrak X_{\text{\'et}}$, such that for every morphism  $(f,\alpha):(U,u) \to (V,v)$ in $\mathfrak X_\text{\'et}$ we have $f^\ast \Pi_{V,v} = \Pi_{U,u}$. The following shows that this is a reasonable generalisation of the notion of Poisson structures on manifolds:
\begin{Proposition}
If $X$ is a manifold then there is a bijection between Poisson structures on $X$ and Poisson structures on the \'etale stack $\underline X$. If $X$ is Hausdorff then this bijection induces isomorphisms between Poisson cohomology groups.
\end{Proposition}

\begin{proof}
The bijection is given by the Yoneda lemma:
$$\mathrm{Hom}_{\mathrm{PSh}(\mathbf{Man}_\text{\'et})} \left( \underline X , \mathrm{Poiss} \right) 
\cong \mathrm{Poiss}_X$$
The isomorphisms between cohomology groups follow from Proposition \ref{Manifold cohomology} applied to the Lie algebroid $\mathbf T ^\ast _\Pi \mathfrak X$.
\end{proof}

\subsubsection{Poisson structures in terms of atlases}

If $X \to \mathfrak X$ is an \'etale atlas of an \'etale stack $\mathfrak X$ then since $\mathrm{Poiss}$ is a sheaf with respect to the \'etale topology there is a bijection:
$$ \{ \text{Poisson structures on } \mathfrak X  \}  = 
\mathrm{Hom}_{\mathrm{PSh}(\mathbf{Man}_\text{\'et})}(\mathfrak X_\text{\'et} , \mathrm{Poiss} ) 
\cong \{ \Pi \in \mathrm{Poiss}_X  |  s^\ast \Pi = t^\ast \Pi \}
$$
where $s,t$ are the source and target maps of the Lie groupoid $X_1 \rightrightarrows X$. In general, we'll call such Poisson structures on the base of an \'etale groupoid `invariant'. These structures have been studied in  \cite{Ta1}. (In fact, if $G \rightrightarrows Y$ is an \'etale groupoid then there is a natural left action of $G$ on the bundle ${\bigwedge}^2 TY$, and a Poisson structure is invariant in the sense above exactly if it is a $G$-invariant section of ${\bigwedge}^2 TY$). It is easy see that a Poisson structure $\Pi$ on $X$ is invariant if and only if there exists a (necessarily unique) Poisson structure $\underline \Pi$ on $X_1$ such that the structure maps of $X_1 \rightrightarrows X$ are all Poisson maps.

\subsubsection{Associated $\mathcal{LA}$-groupoids and the total stack}

Now let $\mathbf \Pi$ be a Poisson structure on $\mathfrak X$, and fix an \'etale atlas $X \to \mathfrak X$, so we have a $X_1 \rightrightarrows X$-invariant Poisson structure $\Pi$ on $X$, and a Poisson structure 
$$\underline \Pi = s^\ast \Pi = t^\ast \Pi$$ 
on $X_1$. The $\mathcal{LA}$-groupoid associated to the Lie algebroid $\mathbf T^\ast _{\mathbf \Pi} \mathfrak  X$ is 
$$\xymatrix{  
s^\ast \left(T^\ast _\Pi X \right) \ar[d]   \ar@<-.5ex>[r]\ar@<+.5ex>[r]  &  T^\ast _\Pi X \ar[d] \\
X_1 \ar@<.5ex>[r]\ar@<-.5ex>[r] & X
}$$
which is isomorphic to 
$$\xymatrix{  
T^\ast _{\underline \Pi} X  \ar[d]   \ar@<-.5ex>[r]\ar@<+.5ex>[r]  &  T^\ast _\Pi X \ar[d] \\
X_1 \ar@<.5ex>[r]\ar@<-.5ex>[r] & X
}$$
It was suggested in \cite{Me1}, Example 5.14, that $\mathcal{LA}$-groupoids of this form might be used to describe Poisson structures over orbifolds. 
By definition, the classifying stack of the Lie groupoid $[T^\ast X_1 \rightrightarrows T^\ast X]$ is the cotangent stack $T^\ast \mathfrak X$ of $\mathfrak X$. Therefore the total stack (see \ref{Total stacks of Lie algebroids over stacks}) $T^\ast _{\mathbf \Pi} \mathfrak X$ of the Lie algebroid $\mathbf T^\ast _\mathbf \Pi \mathfrak X$ is $T^\ast \mathfrak X$ (as expected from the manifold case), and we have a diagram
$$\xymatrix{
T^\ast _{\mathbf \Pi} \mathfrak X \ar[dr]  \ar[r] & T\mathfrak X \ar[d]  \\
& \mathfrak X
}$$

\subsubsection{Poisson structures on duals of Lie algebroids}

We can construct some examples of Poisson structures as follows. Let $\mathbf A$ be a Lie algebroid over $\mathfrak X$, let $X \to \mathfrak X$ be an \'etale atlas, and $A_X$ be the induced Lie algebroid over $X$, so we have the $\mathcal{LA}$-groupoid
$$\xymatrix{  
A_{X_1}  \ar[d]   \ar@<-.5ex>[r]\ar@<+.5ex>[r]  &  A_X \ar[d] \\
X_1 \ar@<.5ex>[r]\ar@<-.5ex>[r] & X
}$$
where $A_{X_1}=s^\ast A_X$. Since $A_{X_1} \rightrightarrows A_X$ is an \'etale groupoid we can dualise its structure maps to give a groupoid object in the category of vector bundles:
$$\xymatrix{  
A_{X_1}^\ast  \ar[d]   \ar@<-.5ex>[r]\ar@<+.5ex>[r]  &  A_X^\ast \ar[d] \\
X_1 \ar@<.5ex>[r]\ar@<-.5ex>[r] & X
}$$
Recall from \ref{Poisson structures on manifolds} that the manifolds $A_X^\ast$ and $A_{X_1}^\ast$ carry natural fibrewise linear Poisson structures, and it follows from the fact that $(A_{X_1},X_1,A,X)$ is an $\mathcal{LA}$-groupoid that the structure maps of the \'etale groupoid $A_{X_1}^\ast \rightrightarrows A_X^\ast$ are Poisson maps. In particular, the Poisson structure on $A_X$ is invariant and therefore determines a Poisson structure on the \'etale stack $[A_{X_1}^\ast \rightrightarrows A_X^\ast]$.

\subsubsection{Symplectic structures}

Another class of examples of Poisson structures are those arising from symplectic forms. (See \cite{Can1} for the relevant aspects of symplectic geometry, and \cite{Va1} and \cite{LaPiVa1} for the relationship to Poisson geometry). Recall that a symplectic form on a manifold $X$ is a 2-form $\omega \in \Omega^2(X)$ which is closed and nondegenerate.  Associated to a symplectic form $\omega$ is a Poisson structure $\{,\}_\omega$ given by
$$\{ f, g \}_\omega = \omega (v_f, v_g) $$
where, for a function $f \in \mathrm C^\infty (X)$, $v_f$ is the vector field obtained by the composition of the maps
$$\xymatrix{
\mathrm C^\infty (X)  \ar[r] ^d  &  \Omega^1(X)  \ar[r] ^{\omega^\sharp}  & \Gamma(TX) 
}$$
The corresponding bivector $\Pi_\omega$ is the `inverse' $\omega^{-1}$ of $\omega$. Since $\omega$ is nondegenerate the anchor map of the Lie algebroid $T^\ast_{\omega^{-1}} X$ is an isomorphism 
$$T^\ast_{\omega^{-1}} X \to TX$$
It follows that the Poisson cohomology of $(X, \{,\}_{\omega})$ is canonically isomorphic to the de Rham cohomology of $X$.

Since the exterior derivative commutes with pullbacks, closed 2-forms form a presheaf over $\mathbf {Man}$:
\begin{align*}
\Omega^2 _{\mathrm {cl}} : \mathbf{Man} & \to \mathbf{Set} \\
X & \mapsto \Omega^2 _{\mathrm{cl}} (X)
\end{align*}
where $\Omega^2 _{\mathrm{cl}}(X)$ is the set of closed 2-forms on $X$. Nondegeneracy is in general not preserved by pullbacks, but is preserved under pullbacks along \'etale maps. It follows that there is a presheaf over $\mathbf {Man}_\text{\'et}$:
\begin{align*}
\mathrm{Symp}: \mathbf{Man}_\text{\'et} & \to \mathbf{Set} \\
X & \mapsto \mathrm{Symp}_X
\end{align*}
where $\mathrm{Symp}_X$ is the set of symplectic forms on $X$. Since being closed and nondegenerate is a local condition, $\mathrm{Symp}$ is a sheaf with respect to the open cover and \'etale topologies on $\mathbf{Man}_\text{\'et}$. The construction of a Poisson structue $\{,\}_\omega$ from a symplectic structure $\omega$ commutes with \'etale pullbacks of symplectic and Poisson structures, so there is a natural morphism of sheaves
$$\mathrm{Symp} \to \mathrm{Poiss}$$

\begin{Definition}
A symplectic form $\mathbf \omega$ on an \'etale stack $\mathfrak X$ is a morphism 
$$\mathbf \omega : \mathfrak X_\text{\'et} \to \mathrm{Symp}$$
of stacks over $\mathbf{Man}_\text{\'et}$. A pair $(\mathfrak X, \mathbf \omega)$ is called a symplectic \'etale stack. Associated to a symplectic form $\omega$ on $\mathfrak X$ is a Poisson structure $\mathbf \Pi_\omega$ defined as the composition
$$\xymatrix{
\mathfrak X_ \text{\'et} \ar[r]^\omega  &  \mathrm{Symp}  \ar[r] &  \mathrm{Poiss} 
}$$
\end{Definition}
Explicitly, a morphism $\mathfrak X_\text{\'et} \to \mathrm{Symp}$ is determined by the choice of a symplectic form $\omega_{U,u}$ for every object $(U,u)$ of $\mathfrak X_{\text{\'et}}$, such that for every morphism  $(f,\alpha):(U,u) \to (V,v)$ in $\mathfrak X_\text{\'et}$ we have $f^\ast \omega_{V,v} = \omega_{U,u}$. 

Let $X \to \mathfrak X$ be an \'etale atlas of an \'etale stack $\mathfrak X$. Then because $\mathrm{Symp}$ is a sheaf with respect to the \'etale topology there is a bijection
$$ \{ \text{Symplectic forms on }   \mathfrak X  \}  
\cong \{ \omega \in \mathrm{Symp}_X  |  s^\ast \omega = t^\ast \omega \}
$$
In particular, if $X$ is a manifold then there is a bijection between symplectic forms on the \'etale stack $\underline X$ and symplectic forms in the usual sense on $X$.

Consider the Lie algebroid $\mathbf T^\ast _{\mathbf \Pi_\omega} \mathfrak X$ over $\mathfrak X$ associated to a symplectic form $\omega$ on $\mathfrak X$ (recall Definition \ref{Definition of Poisson structures on stacks}).
Then for every object $(U,u)$ of $\mathfrak X_\text{\'et}$ we have that the anchor map
$$ \left( \Pi_{\omega_{U,u}} \right) ^\sharp : \left( \mathbf T^\ast _{\mathbf \Pi_\omega} \mathfrak X \right) _{U,u} \to TU $$
is an isomorphism because $\Pi_{\omega_{U,u}}$ is the Poisson structure associated to the symplectic form $\omega_{U,u}$. The anchor maps therefore give an isomorphism:
$$\mathbf T^\ast _{\mathbf \Pi_\omega} \mathfrak X \to \mathbf T \mathfrak X$$
It then follows from the functoriality of Lie algebroid cohomology (see the remark after Definition \ref{Lie algebroid cohomology over stacks}) that there is an isomorphism

$$\mathrm H ^\bullet \left( \mathbf T^\ast _{\mathbf \Pi_\omega} \mathfrak X \right)  
\cong \mathrm H ^\bullet _\text{dR} ( \mathfrak X )$$
Therefore, as in the case of manifolds, the Poisson cohomology of a Poisson \'etale stack $(\mathfrak X , \mathbf \Pi_\omega)$, where $\mathbf \Pi_\omega$ is the Poisson structure associated to a symplectic form $\omega$ on $\mathfrak X$, is isomorphic to the de Rham cohomology of $\mathfrak X$.

\subsection{Lie algebroid structures on line bundles over \'etale stacks} \label{Etale examples}

For any manifold $X$ there is a bijection between Lie algebroid structures on the trivial line bundle $\mathbb R \times X \to X$ and the set of vector fields $\Gamma(TX)$: the anchor map 
$$\xymatrix{
\mathbb R \times X \ar[dr]  \ar[r]^a  &  TX  \ar[d]  \\
&  X 
}$$
is determined by the vector field $v=a(1)$, where $1$ is the constant section $x \mapsto 1$ of $\mathbb R \times X$, and identifying sections of $\mathbb R \times X$ with smooth functions on $X$, the Lie bracket is then forced to be
$$[f,g] = fv(g) - gv(f)$$
and so the Lie algebroid structure is completely determined by the anchor map. If $a:\mathbb R \times X \to TX$ and $a' : \mathbb R \times X \to TX$ are two different Lie algebroid structures, then a simple calculation shows that a morphism of vector bundles $\phi: \mathbb R \times X \to \mathbb R \times X$ is a morphism of Lie algebroids if and only $a = a' \circ \phi$.

More generally let $L$ be a Lie algebroid over $X$ that has rank 1 as a vector bundle. Then the anchor map $a:L \to TX$ is determined by a section 
$$ \chi \in \Gamma \left( L^\ast \otimes  TX  \right) $$
If $\xi$ is a local trivialising section of $L$, then $\xi^\ast$ gives a local trivialisation of $L^\ast$, and locally $\chi = \xi^\ast \otimes v$ for some unique local section of $TX$. If $f\xi , g \xi$ are arbitrary local sections of $L$ then we have
\begin{align*}
\left[ f\xi , g\xi  \right]  &  = a(f\xi)(g) \xi -  a(g\xi)(f) \xi + fg[\xi,\xi] \\
& = \left( fv(g) - gv(f) \right) \xi
\end{align*}
and so the Lie bracket is completely determined by the section $\chi$ and we have: 
\begin{Proposition}
If $L$ is a real line bundle over a manifold $X$ then there is a bijection
$$ \{ \text{\emph{Lie algebroid structures on} } L \}  \cong \mathrm{Hom}_{\mathrm{Vect}_X}(L,TX) $$
\end{Proposition}
A similar calculation to that mentioned above shows that if $L,L'$ are two rank 1 Lie algebroids over $X$ with anchor maps $a:L \to TX$ and $a':L' \to TX'$, then a morphism of vector bundles $\phi:L \to L'$ is a morphism of Lie algebroids if and only if $a = a' \circ \phi$.

If $Y \to X$ is an \'etale map and $L$ is a rank $1$ Lie algebroid on $X$, with anchor $a:L \to TX$, then $\phi^\ast L$ is a rank 1 Lie algebroid over $Y$ with anchor given by the composition
$$\xymatrix{
\phi^\ast L  \ar[r]^{\phi^\ast} &  \phi^\ast TX   \ar[r] ^\cong & TY
}$$
(see \ref{etale pullbacks}).

Now let $\mathfrak X$ be an \'etale stack, and $\mathbf L : \mathfrak X_\text{\'et} \to \mathrm{Vect}$ be a line bundle over $\mathfrak X$. For each object $(U,u)$ of $\mathfrak X_\text{\'et}$ we have a line bundle $L_{U,u}$, and for each morphism $(f,\alpha):(U,u) \to (V,v)$ an isomorphism of line bundles 
$$L_{f,\alpha} : L_{U,u} \cong f^\ast L_{V,v}$$
A Lie algebroid structure on $\mathbf L$ is then determined by anchor maps
$$a_{U,u} :  L _{U,u} \to TU$$
that are compatible with the isomorphisms $L_{f,\alpha}$, which is exactly the condition that they determine a morphism of vector bundles
$$\mathbf L \to \mathbf T \mathfrak X$$
We therefore have:
\begin{Proposition}
Let $\mathbf L : \mathfrak X_\text{\emph{\'et}} \to \mathrm{Vect}$ be a line bundle over an \'etale stack $\mathfrak X$. Then there is a bijection
$$ \{ \text{\emph{Lie algebroid structures on} } \mathbf L \}  \cong  \mathrm{Hom}_{\mathrm{Vect}_\mathfrak X} (\mathbf L,\mathbf T\mathfrak X) $$
\end{Proposition}

One can prove an equivalent statement by choosing an atlas $X \to \mathfrak X$ and checking that a map $L_X \to TX$ defines a Lie algebroid structure satisfying the conditions of Proposition \ref{G-sheaves of Lie algebras...} if and only if the map is $X\times_\mathfrak X X \rightrightarrows X$-equivariant, in which case it corresponds to a morphism of vector bundles over $\mathfrak X$.

\subsection{Trivial transitive Lie algebroids}

Let $\ast$ be a one point space, considered as a presheaf over $\mathbf{Man}_\text{sub}$, then by the Yoneda lemma we have an equivalence
$$ \mathrm{Hom}_{\mathrm{PSh}(\mathbf {Man}_\text{sub})} (\ast , \mathcal{LA})
\simeq 
\mathcal{LA}_\ast
$$
and $\mathcal{LA}_\ast$ is just the category of finite dimensional $\mathbb R$-Lie algebras. For each such Lie algebra $\mathfrak g$ we therefore have a Lie algebroid $\underline {\mathfrak g}$ over $\ast$. As $\ast$ is a terminal object in $\mathrm{PSh}(\mathbf {Man}_\text{sub})$ there is, for any stack $\mathfrak X$ over $\mathbf{Man}$ a unique morphism $p:\mathfrak X _\text{sub} \to \ast$.
Therefore, for any finite dimensional $\mathbb R$-Lie algebra $\mathfrak g$ we have a morphism
$$\xymatrix{
\mathfrak X_\text{sub} \ar[r] ^ p &  \ast \ar[r] ^{\mathfrak g}  & \mathcal{LA} 
}$$
which defines a Lie algebroid $\underline {\mathfrak g} \oplus \mathbf T \mathfrak X$ over $\mathfrak X$. If $(U,u)$ is an object in $\mathfrak X_\text{sub}$ then we have
$$( \underline{\mathfrak g} \oplus \mathbf T \mathfrak X)_{U,u} 
= \underline{\mathfrak g} _{U, p \circ u} 
\cong {(p \circ u)}^! \mathfrak g 
\cong (\mathfrak g \times U) \oplus TU
$$ 
where $(\mathfrak g \times U) \oplus TU$ is the Lie algebroid over $U$ with underlying vector bundle $(\mathfrak g \times U) \oplus TU$, anchor map given by the projection onto $TU$, and Lie bracket
$$ \left[ (\xi,v) , (\xi',v') \right]  = \left(  [\xi,\xi'] + v(\xi') - v'(\xi') ,  [v,v']  \right) $$
for sections $\xi,\xi'$ of $\mathfrak g \times U$ and $v,v'$ of $TU$. This is the `trivial transitive Lie algebroid with isotropy Lie algebra $\mathfrak g$', see \cite{Ma2}. In particular, if $\mathfrak g = 0$, then 
$$\underline 0 = \mathbf T \ast$$ 
and 
$$\underline {0} \oplus \mathbf T \mathfrak X \cong \mathbf T \mathfrak X$$

\newpage

\section{Notation}

\begin{itemize}
\item $\underline x = \mathrm{Hom}_\mathscr C ( - , x)$ - representable presheaf corresponding to object $x$ in a category $\mathscr C$
\item $x \mapsto \underline x$ - Yoneda embedding
\item $\mathrm{PSh}(\mathscr C)$ - category of presheaves of sets / weak presheaves of groupoids over site $\mathscr C$
\item $\mathrm{Sh}(\mathscr C)$ - category of sheaves of sets over site $\mathscr C$
\item $\mathrm{St}(\mathscr C)$ - 2-category of stacks over site $\mathscr C$
\\
\item $\mathbf{Man}$ - category of smooth manifolds and smooth maps
\item $\mathbf{Man}_\mathrm{sub}$ - category of smooth manifolds and submersions
\item $\mathbf{Man}_\text{\'et}$ - category of smooth manifolds and \'etale 
\item $\mathrm{St}(\mathbf {Man})$ - 2-category of stacks over $\mathbf{Man}$
\item $\mathrm{St}(\mathbf {Man})_\mathrm{sub}$ - 2-category of stacks over $\mathbf{Man}$ and representable submersions between them
\item $\mathrm{St}(\mathbf {Man})_\text{\'et}$ - 2-category of stacks over $\mathbf{Man}$ and representable \'etale maps between them
\\
\item $X,Y,Z,U,V$ - smooth manifolds
\item $\mathfrak X, \mathfrak Y, \mathfrak Z, \mathfrak A, \mathfrak B$ - stacks over $\mathbf{Man}$
\item $\mathfrak X _ \text{sub}$ - submersion site of stack $\mathfrak X$
\item $\mathfrak X _ \text{\'et}$ - \'etale site of stack $\mathfrak X$
\item $G\rightrightarrows X$  - Lie groupoid
\item $s,t$ - source and target maps of Lie groupoid
\item $m,u,i$ - other structure maps (multiplication, unit, inverse)
\item $[ G \rightrightarrows X ]$ - classifying stack of $G\rightrightarrows X$
\item $X_\bullet , Y_\bullet, Z_\bullet$ - simplicial manifolds
\\
\item $E,F$ - vector bundles
\item $\mathcal E,F$ - sheaves of sections of vector bundles $E,F$
\item $\phi^\ast E$ - pullback of vector bundle $E$ along a map $\phi$
\\
\item $\mathrm{Vect}$ - category of smooth real vector bundles
\item $\mathrm{Vect}_X$ - category of smooth real vector bundles over a manifold $X$
\\
\item $A,B$ - Lie algebroids
\item $a,b$ - anchor maps of Lie algebroids $A,B$
\item $\mathcal A, \mathcal B$ - sheaves of sections of $A,B$
\item $\Omega^\bullet(A)$ - de Rham algebra of Lie algebroid $A$
\item $\Omega_A$ - sheaf of de Rham algebras of Lie algebroid $A$

\item $\mathrm{Rep}_A$ - category of representations of a Lie algebroid $A$ 
\item $\phi^! A$ - Lie algebroid pullback of Lie algebroid $A$ along a map $\phi$
\\
\item $\mathcal{LA}$ - category of Lie algebroids over smooth manifolds
\item $\mathcal{LA}_X$ - category of Lie algebroids over a manifold $X$
\item $\mathrm{Rep}$ - the stack of representations of Lie algebroids
\item $\mathrm{Rep}_X$ - the category of representations of Lie algebroids over a manifold $X$
\\
\item $\mathbf A$ - Lie algebroid over a stack
\item $\mathbf E$ - vector bundle over a stack
\item $\mathrm{Rep}_\mathbf A$ - category of representations of a Lie algebroid $\mathbf A$ over a stack
\\
\item $(A,\psi)$ - Lie algebroid over a Lie groupoid
\item $\mathrm{Rep}_{(A,\psi)}$ - category of representations of a Lie algebroid $(A,\psi)$ over a Lie groupoid
\\
\item $\mathrm{Flat}_X$ - Category of flat vector bundles over $X$
\item $\mathrm{Flat}$ - Stack over $\mathbf{Man}_\text{sub}$ of flat vector bundles
\item $\mathrm{Loc}_X$ - Category of local systems over $X$
\\
\item $\mathrm{Loc}$ - Stack over $\mathbf{Man}_\text{sub}$ of local systems
\item $\mathrm{Poiss}_X$ - Set of Poisson structures on $X$
\item $\mathrm{Poiss}$ - Sheaf over $\mathbf{Man}_\text{\'et}$ of Poisson structures
\item $\mathrm{Symp}_X$ - Set of symplectic forms on $X$
\item $\mathrm{Symp}$ - Sheaf over $\mathbf{Man}_\text{\'et}$ of symplectic forms
\end{itemize}

\end{document}